\documentclass[11pt]{amsart}
\usepackage[margin= 1.2 in]{geometry}

\usepackage{pdfsync}

 \usepackage[plainpages=false]{hyperref}
 \usepackage{amsfonts,amsmath,amscd,amssymb,amsthm,accents,mathtools}
 \usepackage{latexsym,lscape,rawfonts,mathrsfs}

 \usepackage[all]{xy}
 \usepackage{eufrak}
 \usepackage{graphicx,psfrag}
 \usepackage{pstool}
\usepackage{tikz}
\usetikzlibrary{matrix,arrows,decorations.pathmorphing}
 \usepackage{array,tabularx}

 \usepackage{appendix}

\usepackage{txfonts}
\usepackage[T1]{fontenc}
\usepackage[utf8]{inputenc}
\usepackage{mathtools}   
\usepackage{amsfonts}

 \newcommand{\ba}{\begin{align}}
 \newcommand{\ea}{\end{align}}
 \newcommand{\bal}{\begin{align*}}
 \newcommand{\eal}{\end{align*}}

 \DeclareMathOperator{\diam}{diam}

 \newcommand{\Rm}{\mathbf{Rm}}
 \newcommand{\Rc}{\mathbf{Rc}}
 \newcommand{\Sc}{\mathbf{R}}

 \newcommand{\dvol}{\text{d}V}
\newcommand{\dmu}{\text{d}\mu}

 \newcommand{\darea}{\text{d}\sigma}
 
 \newcommand{\vol}{\text{Vol}\ }

\renewcommand{\epsilon}{\varepsilon}

\newcommand{\rank}{\text{rank}\ }

 \makeatletter
 \def\ExtendSymbol#1#2#3#4#5{\ext@arrow 0099{\arrowfill@#1#2#3}{#4}{#5}}
 
 \makeatother

 \makeatletter
 \def\ExtendSymbol#1#2#3#4#5{\ext@arrow 0099{\arrowfill@#1#2#3}{#4}{#5}}
 
 \makeatother

 \definecolor{hao}{rgb}{1,0.5,0}
 \definecolor{miao}{cmyk}{0.5,0,0.2,0.2}
 \definecolor{qiao}{gray}{0.96}

\newtheorem{prop}{Proposition}[section]

\newtheorem{proposition}[prop]{Proposition}

\newtheorem{theorem}[prop]{Theorem}
\newtheorem{prop-defn}[prop]{Proposition-Definition}

\newtheorem{lemma}[prop]{Lemma}

\newtheorem{corollary}[prop]{Corollary}

\newtheorem*{theorem*}{Theorem}
\theoremstyle{remark}
\newtheorem{remark}{Remark}

\numberwithin{equation}{section}


\title[Long-time behavior of immortal Ricci flows]{On the long-time behavior of
immortal Ricci flows}
\author{Shaosai Huang}
\address{Department of Mathematics, University of Wisconsin - Madison, 480
Lincoln Drive, Madison, WI 53706, U.S.A.}
\email{sshuang@math.wisc.edu}

\date{\today}

\begin{document}
\maketitle

\begin{abstract}
For an immortal Ricci flow on an $m$-dimensional $(m\ge 3)$ closed manifold, 
we show the following convergence results: 
(1) if the curvature and diameter are uniformly bounded, then any
unbounded sequence of time slices sub-converges to a Riemannian orbifold; (2)
if the flow is type-III with diameter growth controlled by
$t^{\frac{1}{2}}$, then any blowdown limit is an $m$-dimensional negative
Einstein manifold, provided that Feldman-Ilmanen-Ni's
$\boldsymbol{\mu}_+$-functional satisfies $\lim_{t\to \infty}
t\boldsymbol{\mu}_+'(t)=0$.
\end{abstract}

\tableofcontents
\section{Introduction}
\addtocontents{toc}{\protect\setcounter{tocdepth}{1}}

Let $M$ be a closed $m$-dimensional smooth manifold, and suppose it
admits a Ricci flow solution $g(t)$ on $[0,T)$ for some $T>0$, i.e. the
Riemannian metric tensor $g(t)$ satisfies the partial differential equation on
$M$:
\begin{align}
\forall t\in [0,T),\quad \partial_tg(t)\ =\ -2\Rc_{g(t)}.
\end{align}  
Here we may think of $T$ as the first time when the smooth Ricci flow
solution develops a singularity. It is natural to expect that the structure of
the possible rescaled limits of $(M,g(t))$ as $t\to T$ can help us understand
the structure of the manifold $(M,g(t))$ for $t<T$.

With the introduction of the $\mathcal{W}$-functional in \cite{Perelman},
Perelman showed that when $T<\infty$, there is a uniform lower bound (depending
on $g(0)$ and $T$) of the volume ratio as $t\nearrow T$. This is a key step in
his completion of Hamilton's program on proving Thurston's geometrization
conjecture, see \cite{HamiltonSurvey, Hamilton99, Perelman, Perelman2}.

In contrast, for immortal Ricci flows, i.e. when $T=\infty$, a key
difficulty in understanding the long-time limit behavior is that the
\emph{global volume ratio}, $|M|_{g(t)} \diam (M,g(t))^{-m}$, may degenerate to
$0$ as $t\to \infty$. This could not only be seen from the dependence of
Perelman's volume ratio lower bound on time (see, for instance,
\cite[(4.9)]{Foxy1809}, but is also illustrated by the behavior of certain
type-III Ricci flows in dimension three (see \cite{Lott10} and \cite{Bamler18}).

While the blowdown limits of homogeneous immortal Ricci flows have been shown to
be homogeneous expanding Ricci solitons through the deep work of B\"ohm and
Lafuente \cite{BL18} (see also \cite{Bohm15} and \cite{BLS17}), the general case
is far from being well understood. In this article we will focus on studying
the rescaled limits of an immortal Ricci flow $(M,g(t))$, as $t\to \infty$,
under the following uniform curvature-diameter bound: there is a uniform
constant $D>0$ such that
\begin{align}\label{eqn: CD_bound}
\forall t\in [0,\infty),\quad \diam (M,g(t))^2\sup_{M}|\Rm_{g(t)}|_{g(t)}\
\le\ D^2.
\end{align}
Especially, we notice that such assumption is naturally satisfied by type-III
Ricci flows with diameter growth of order $t^{\frac{1}{2}}$; in dimension three,
such Ricci flows will produce single geometric pieces in Thurston's
geometrization program, see \cite[Theorem 1.2 and Remark 1.4]{Lott10}.

 Our first result is the following
\begin{theorem}[Limit of controlled Ricci flows]\label{thm: main1}
Let $(M,g(t))$ be an $m$-dimensional ($m\ge 3$) immortal Ricci flow satisfying
(\ref{eqn: CD_bound}). Suppose that the curvature and diameter of $(M,g(t))$
remain uniformly bounded for all $t\ge 0$, then any unbounded sequence of time
slices $\{(M,g(t_i))\}$ sub-converges to a compact Riemannian orbifold.
\end{theorem}

Clearly, if the global volume ratio $|M|_{g(t_i)} \diam (M,g(t_i))^{-m}$ has a
uniform positive lower bound along the flow $(M,g(t))$, then
Theorem~\ref{thm: main1} follows directly from the classical results in
\cite{CGT82}, \cite{Hamilton95}, \cite{Perelman} and \cite{FIN05} --- the limit
is actually a closed $m$-dimensional Ricci flat manifold. The major concern of
the current work is therefore the case when $\liminf_{t\to \infty} |M|_{g(t)}\diam
(M,g(t))^{-m}=0$.

 An immediate consequence of Theorem~\ref{thm: main1} and \cite[Theorem
 7-6]{Fukaya89} is the following structural result concerning the
 \emph{sufficiently collapsed} time slices in an immortal Ricci flow (see \S
 2.1.2 and \cite[\S 7]{Fukaya89} for relavent definitions):
\begin{corollary}\label{cor: cor1}
There is a positive constant $\nu(m)>0$ such that for any immortal Ricci flow
$(M,g(t))$ as in Theorem~\ref{thm: main1} with sectional curvatures bounded by
$1$, if $|M|_{g(t_0)}\diam (M,g(t_0))^{-m}\le \nu(m)$ for some $t_0>0$, then
$M$ is an infranil fiber bundle over a compact (lower dimensional) Riemannian
orbifold.
\end{corollary}

The conclusion of this corollary about $M$ being an infranil bundle over a
Riemannian orbifold could be rephrased in other languages. It is equivalent to
say that $M$ admits a pure polarized $F$-structure \emph{a la} Cheeger and
Gromov \cite{CG86, CG90}. It is also the same as saying that $M$, together with
the fiber-wise infinitesimal nilpotent group actions, is Morita equivalent to
an \'etale groupoid, in the groupoid approach to collapsing geometry pioneered 
by Lott \cite{Lott07}.

In fact, it is expected, as pointed out by Richard Bamler and Aaron Naber, that
the evolution of an immortal Ricci flow with uniformly bounded curvature and
diameter will not cause volume collapsing. But since we cannot make an \emph{a
priori} assumption on the uniform positive lower bound of the global volume
ratio, we have to first study the possible collapsing geometry as $t\to \infty$
in Theorem~\ref{thm: main1}, understand the structure of the Ricci flow when
the metric is sufficiently collapsed as in Corollary~\ref{cor: cor1}, and then
try to obtain a desired positive lower bound of the global volume ratio via a
contradiction argument, \emph{a posteriori}; see also \cite[\S 1 and \S
6]{Foxy1705} for discussions on a similar strategy concerning the uniform
$\boldsymbol{\mu}$-entropy lower bound of $4$-dimensional Ricci shrinkers.

This strategy is illustrated in another natural situation about immortal Ricci
flows satisfying (\ref{eqn: CD_bound}): for compact type-III Ricci flows with
diameter growth of order $t^{\frac{1}{2}}$, we show that the global volume
ratio has a positive lower bound depending on the limit behavior of the
$\boldsymbol{\mu}_+$-functional. To state our second result, we recall that the
$\boldsymbol{\mu}_+$-functional defined by Feldman, Ilmanen and Ni \cite{FIN05}
\begin{align*}
\boldsymbol{\mu}_+(t)\ :=\ \inf\left\{\mathcal{W}_+(g(t),u,t):\ \int_Mu\
\dvol_{g(t)}=1\right\}
\end{align*}  
is non-decreasing along the Ricci flow, and it is differentiable with respect to
$t$. In this case, we have the following
\begin{theorem}[Non-collapsing of certain type-III Ricci flows] 
\label{thm: main2}
Let $(M,g(t))$ be an $m$-dimensional ($m\ge 3$) immortal Ricci flow satisfying
(\ref{eqn: CD_bound}). If $(M,g(t))$ is type-III with $\diam (M,g(t))
=O(t^{\frac{1}{2}})$, then
\begin{align}\label{eqn: main2}
\limsup_{t\to\infty}t\boldsymbol{\mu}_+'(t)\ =\ 0\quad \Rightarrow \quad
\liminf_{t\to \infty} |M|_{g(t)}\diam (M,g(t))^{-m}\ >\ 0.
\end{align}
\end{theorem}

This theorem, to be proven in \S 6, could be seen as a Ricci flow
version of a theorem of Rong \cite[Theorem 0.4]{Rong98}; see also \S 7 for a
simple proof of Rong's theorem. Notice that the asymptotic degeneration of the
global volume ratio forces $\boldsymbol{\mu}_+(t)$ to be unbounded, see
(\ref{eqn: W+_blow}), and the theorem tells that $\boldsymbol{\mu}_+(t)$ should
grow faster than $\ln t$ in this case. In fact, if
$\limsup_{t\to\infty}t\boldsymbol{\mu}_+'(t)=0$, then by \cite{CGT82},
\cite{Hamilton95}, \cite{Perelman} and \cite{FIN05}, any blowdown limit is an
$m$-dimensional negative Einstein manifold.

The proof of Theorems~\ref{thm: main1} and \ref{thm: main2} are inspired by
the works of Lott \cite{Lott10, Lott17} and of Naber and Tian \cite{NaberTian},
and are based on the understanding of collasping geometry in the deep
work of Cheeger, Fukaya and Gromov \cite{CFG92}, as well as the measured
Gromov-Hausdorff convergence introduced by Fukaya \cite{Fukaya87}.

To further understand the content of Theorem \ref{thm: main1}, let us briefly
discuss the structure of the possible limit metric spaces to which the
sequences of manifolds in this theorem may converge. Adapting to the
situations of the above theorems, we will assume to have a sequence of
Riemannian manifolds $\{(M_i,g_i)\}$ with $\diam (M_i,g_i)\le D$, and we will
assume that
\begin{align}\label{eqn: regularity}
\forall l\in \mathbb{N},\quad \sup_{M_i}|\nabla^l\Rm_{g_i}|_{g_i}\ \le\
C_{1.3}(l),\quad \text{and}\quad C_{1.3}(0)=1.
\end{align} 
Thanks to Shi's estimates \cite{Shi}, these assumptions are satisfied for the
sequences $\{(M,g(t_i))\}$ in Theorems \ref{thm: main1} and \ref{thm: main2}.

 Assuming $|M_i|_{g_i}\to 0$ as $i\to \infty$, by Gromov's compactness theorem
 \cite{GLP}, we know that after possibly passing to a sub-sequence (still
 denoted by the original one), $\{(M_i,g_i)\}$ converges in the Gromov-Hausdorff
 topology to a lower (Hausdorff) dimensional metric space $(X,d)$. 
 Although the lack of a uniform injectivity radius lower bound for
 $\{(M_i,g_i)\}$ makes $X$ fail to be a manifold in general, the regularity
 assumption (\ref{eqn: regularity}) gives more information on both the
 collapsing limit space $(X,d)$ and the convergence procedure. By \cite[Theorems
 0.5]{Fukaya88}, we know that $(X,d)$ is, roughly speaking, an orbifold with
 corners --- each point of $X$ has a sufficiently small neighborhood isometric
 to the quotient of some open set in $\mathbb{R}^m$, equipped with a Riemannian
 metric that is invariant under the action of a germ of a nilpotent Lie group
 $N$. The singularity types of a point $x\in X$ then depend on the isotropy
 group $G_x$ of the isometric action by the germ of $N$. The isotropy group can
 be either finite, giving rise to an \emph{orbifold point} (including the
 possibility of being a \emph{regular point} when the isotropy group is
 trivial), or be a finite extension of a torus group, in which case the
 resulting point is a \emph{corner singularity}. We denote the subset of
 regular points of $X$ by $\mathcal{R}$, the collection of orbifold points in
 $X$ as $\tilde{\mathcal{R}}$, and the subset of corner singularities as
 $\tilde{\mathcal{S}}$. Clearly $\tilde{\mathcal{R}}\backslash \mathcal{R}$
 consists of \emph{orbifold singularity}, i.e. those points with finite but
 non-trivial isotropy groups, and $X=\tilde{\mathcal{R}}\sqcup
 \tilde{\mathcal{S}}$. Here we notice that $\tilde{\mathcal{R}}$ is an open
 subset of $X$, and $\tilde{\mathcal{S}}$ is a closed subset of codimension at
 least $1$ in $X$. Consequently, the metric $d$ on $X$ is induced by a
 Riemannian metric $g_X$ on $\tilde{\mathcal{R}}$. See \S 2.1 for more details.
 
 Therefore, the point of Theorem \ref{thm: main1} is to show that under the
 evolution of the Ricci flows, the corner singularity $\tilde{\mathcal{S}}$
 \emph{cannot} possibley appear in the collapsing limit $X$. In view of
 Corollary \ref{cor: cor1}, such reduction of the limit singularity type
 provides rich information about the topological structure of the underlying
 manifold, if the global volume ratio becomes sufficiently small along the Ricci
 flow.

 \begin{remark}\label{rmk: S_nonempty}
 The fact that $\tilde{\mathcal{S}}$ may not be empty is the same as saying that
 the $F$-structure on $M_i$ (for $i$ sufficiently large) is not necessarily
 polarized, see \cite{CG86, CR96}. In terms of the groupoid approach \cite[\S
 5]{Lott07}, this tells that the limit groupoid, in its \emph{natural}
 topology, is \emph{not} necessarily Morita equivalent to an \'etale groupoid.
 For a notion of Riemannian metrics on such limit groupoids, see \cite{GGHR89,
 dHF14}. By \cite{Fukaya88}, the local structure around the corner singularity
 can also be described as a linearized singular Riemannian foliation, equipped
 with a bundle-like metric, see \cite{Molino88}.  Compare also \cite{GL13,
 Hilaire14} for a notion of cross-product groupoid on the orthogonal frame
 bundle.
 \end{remark}
 
 The evolution equation of the Ricci flow plays a necessary role in the 
 reduction of the singularity type --- as pointed out in Remark \ref{rmk:
 S_nonempty}, for a generic collapsing sequence with bounded diameter and
 sectional curvature, it is totally possible that $\tilde{\mathcal{S}}\not=
 \emptyset$, see \cite[Example 1.7]{CG86}. Along the Ricci flow we should
 expect, as $t\to \infty$, certain gradient steady Ricci soliton metric at
 the limit; and the corresponding elliptic equations satisfied by the limit
 metric will impose strong constraint on the possibility of singularity types.
 Here we encounter a major issue caused by the possible volume collapsing --- we
 do not have any local coordinate system in which the limit soliton metric can
 be written down.
 
 This issue can be resolved, at least around the orbifold points, if
 we recall the fiberation theorems \cite[Theorem 0.12]{Fukaya88}, \cite[Theorem
 0-7]{Fukaya89} and \cite[Theorem 2.6]{CFG92}: for all $i$ sufficiently large,
 there is a continuous surjective map $f_i:M_i\to X$, called a \emph{singular
 fibration}, such that for any $x\in \tilde{\mathcal{R}}$, we can find $U\subset
 \tilde{\mathcal{R}}$ sufficiently small, so that $\forall x'\in
 f_i^{-1}(U)\subset M_i$, (a finite covering of) $f_i^{-1}(x')$ is homeomorphic
 to an infranil manifold $F_i$, and the extrinsic diameter of each fiber is
 bounded by $3d_{GH}(M_i,X)$. Moreover, the collapsing of $(M_i,g_i)$ to $(X,d)$
 is exactly caused by the shrinking of the $f_i$ fibers to points. Notice that
 each $F_i$ is a quotient of a simply connected nilpotent Lie group $N_i$ by a
 finite extension $\Gamma_i$ of a cocompact lattice $L_i\le N_i$, and roughly
 speaking, the shrinking of the $f_i$ fibers to a point is caused by the
 increasingly dense action of $L_i$ on the universal covering $N_i$ of the $f_i$
 fibers. Therefore, it is natural to consider $W_{i}$, the universal
 convering of $f_i^{-1}(U)$, which fibers over $U$ by the universal
 coverings of the $f_i$ fibers (homeomorphic to $N_i$). Equipping $W_i$
 with the covering metric $\tilde{g}_i$, the regularity
 assumption (\ref{eqn: regularity}) then ensures a uniform lower bound of the
 injectivity radius. Therefore we can work on the neighborhoods $W_{i}$, and
 take limit out of the metrics $\{\tilde{g}_i\}$. More precisely, we have the
 following
 \begin{theorem}[Unwrapped neighborhoods around orbifold points]\label{thm:
 canonical_nbhd} 
 Let $\{(M_i,g_i)\}$ be a sequence of $m$-dimensional Riemannian
 manifolds satisfying (\ref{eqn: regularity}) that collapses to a
 (Hausdorff) $n$-dimensional metric space $(X,d)$, and let $f_i:M_i\to X$ denote
 the singular fibration described in \cite[Theorem 0.12]{Fukaya88}. For any
 $x_0\in \tilde{\mathcal{R}}$, there is a sufficiently small neighborhood
 $U_{x_0}\subset \tilde{\mathcal{R}}$, an orbifold covering $V_{x_0}\subset
 \mathbb{R}^n$ with a finite covering group $G_{x_0}$, and a
 $G_{x_0}$ invariant Rieannian metric $\hat{g}_X$ on $V_{x_0}$, such that
 $(U_{x_0},d)\equiv (V_{x_0},\hat{g}_X)\slash G_{x_0}$ with quotient map
 denoted by $q_{x_0}$. Moreover, there are $W_{x_0}:= V_{x_0}\times
 \mathbb{R}^{m-n}$ together with the natural projection $p:W_{x_0}\to V_{x_0}$,
 and a small positive number $r_{x_0}>0$ (depending only on $x_0$ and $X$), to the
 following effect:
 \begin{enumerate}
   \item on $W_{x_0}$ there are families of $G_{x_0}$ invariant
   Riemannian metrics $\{\tilde{g}_{i}\}$ and $G_{x_0}$
   equivariant connections $\{\tilde{\nabla}_i\}$ subject to the following
   regularity control:
   for any $l\in \mathbb{N}$, 
   \begin{align*}
   \sup_{W_{x_0}}|\nabla^l\Rm_{\tilde{g}_i}|_{\tilde{g}_i}\ \le\ C_l
   r_{x_0}^{1-l}\quad \text{and}\quad \sup_{W_{x_0}}
   |\nabla^l(\tilde{\nabla}_i-\tilde{\nabla}^{LC}_i)|\ \le\ C_l
   r_{x_0}^{-2-l}d_{GH}(M_i,X),
    \end{align*}
    where $\tilde{\nabla}^{LC}_i$ denotes the Levi-Civita connection of
    $\tilde{g}_{i}$.
   \item $N_{i}:=\left(p^{-1}(x_0),\tilde{\nabla}_{i,x_0},(x_0,o)\right)$
   becomes an $(m-n)$-dimensional simply connected nilpotent Lie group, where
   $\tilde{\nabla}_{i,x_0}$ denotes the restriction of
   $\tilde{\nabla}_i$ to $p^{-1}(x_0)$, and the group structure is
   defined by regarding the $\tilde{\nabla}_{i,x_0}$-parallel vector
   fields as \emph{left} invariant vector fields and $(x_0,o)\in p^{-1}(x_0)$
   as the base point.
   \item there are discrete sub-groups $\Gamma_i\le Aff(N_{i})$ that are
   finite extensions of cocompact lattice subgroups $L_i\le N_i$, which acts
   on $p^{-1}(x_0)$ by \emph{left} translations. Moreover, $\tilde{g}_{i}$
   is invariant under the action of $\Gamma_i$. 
   \item the $\Gamma_i$ action on $N_i$ trivially extends to $W_{x_0}$ in view
   of its product structure, and $G_{x_0}$ acts freely on $W_{x_0}$ in a
   way preserving the $\Gamma_i$ orbits; moreover, the quotient maps 
   $\tilde{q}_i:W_{x_0}\to W_{x_0}\slash \Gamma_i$ and $\hat{q}_i:W_{x_0}\slash
   \Gamma_i\to (W_{x_0}\slash \Gamma_i)\slash G_{x_0}$ induce a
   homeomorphic $\Psi(d_{GH}(M_i,X))$ Gromov-Hausdorff approximation between
   $\hat{q}_i(\tilde{q}_i(W_{x_0}))$ and $f_i^{-1}(U_{x_0}) \subset M_i$, when
   $W_{x_0}$ is equipped with $\tilde{g}_{i}$.
 \end{enumerate}
 Moreover, as $i\to \infty$, we get a limit metric $\tilde{g}_{\infty}$ and a
 limit connection $\tilde{\nabla}_{\infty}$ on $W_{x_0}$, to which 
 $\{\tilde{g}_{i}\}$ and $\{\tilde{\nabla}_i\}$ sub-converges,
 respectively, in the $C^{\infty}_{loc}(W_{x_0})$ topology. Consequently, the
 limit simply connected $(m-n)$-dimensional nilpotent Lie group
 $N_{\infty}=\left(p^{-1}(x_0), \tilde{\nabla}_{\infty,x_0},(x_0,o)\right)$
 acts on $(W_{x_0}, \tilde{g}_{\infty})$ by isometric \emph{left} translations,
 making the projection $p: (W_{x_0},\tilde{g}_{\infty})\twoheadrightarrow
 (V_{x_0},\hat{g}_X)$ a Riemannian submersion.
 \end{theorem}
 
 This theorem, to be discussed in more detail in \S 3, is well known to
 experts --- see \cite[Theorem 0.5]{Fukaya88}, \cite[Theorem 2.1]{NaberTian4d}
 and \cite[Theorem 1.1]{NaberTian} for similar constructions. We record it here
 mostly for the convenience of our discussion in the current paper and claim no
 originality. Different from the above mentioned results, Theorem \ref{thm:
 canonical_nbhd} focuses around orbifold points, and provides a direct
 description of the local collapsing struture without involving the frame
 bundle argument.
 
 Recall that our main goal of proving Theorem \ref{thm: main1} is to rule out
 the possible existence of $\tilde{\mathcal{S}}$. At this point, let us mention
 another characterization of the corner singularities
 $\tilde{\mathcal{S}}\subset X$. In \cite{Fukaya87}, a notion of measured
 Gromov-Hausdorff topology has been defined, and it was shown that the metric
 measure spaces $\left\{(M_i,g_i, |M_i|_{g_i}^{-1}\dvol_{g_i})\right\}$
 sub-converges in the measured Gromov-Hausdorff topology to the metric measure
 space $(X,d,\dmu_X)$, where $\dmu_X$ is absolutely continuous with respect to
 the natural measure induced by $d$. The density function $\chi_X$ is defined,
 for any $x\in \mathcal{R}$, as
 \begin{align}\label{eqn: chi_X_regular}
 \chi_X(x)\ :=\ \lim_{i\to
 \infty}\frac{|f_i^{-1}(x)|_{g_i}}{|M_i|_{g_i}},
 \end{align}
and is extended continuously throughout
 $X$. The corner singularities are then characterized by the zero locus of
 $\chi_X$:  $\tilde{\mathcal{S}}=\chi_X^{-1}(0)$. Heuristically, we can think of
 $\chi_X$ as the asymptotic relative volume distribution of the $f_i$ fibers, 
 and $\chi_X$ vanishes on $\tilde{\mathcal{S}}$ because the fibers over corner
 singularities are of lower dimensions, compared to those over the orbifold
 points; see \cite[Theorem 0.12]{Fukaya88}.
 
 Now fixing $x\in \tilde{\mathcal{R}}$ and letting $W_x$ be constructed in
 Theorem~\ref{thm: canonical_nbhd}, it can be shown, following the same
 argument as in \cite[Lemma 2-5]{Fukaya89}, that $\chi_X$ is actually a
 constant multiple of $\sqrt{\det G}$ over $W_x$; here $G$ is the restriction
 to the $p$ fibers of the limit metric $\tilde{g}_{\infty}$.
 Ideally, in the setting of Theorem \ref{thm: main1},
 since the limit metric is a consequence of the Ricci flow evolution, we expect
 that $\tilde{g}_{\infty}$ to satisfy the gradient steady Ricci
 soliton equation. Suppose for now, that this is indeed the case and let
 $u_{\infty}$ denote the potential function, then on $W_x$ we have the following
 inequality for $\ln \det G$, via O'Neill's formula (see the works \cite{Lott03,
 Lott10} of Lott where such argument originates):
 \begin{align}\label{eqn: fiber_volume_heat}
\Delta^{\perp}_{\tilde{g}_{\infty}}\ln \det G+\frac{1}{2}|\nabla^{\perp} \ln
\det G|_{\tilde{g}_{\infty}}^2-\tilde{g}_{\infty}(\nabla^{\perp} \ln \det
G,\nabla^{\perp} \ln u_{\infty})\ =\ 2\Sc_{G}.
 \end{align}
 Here the derivatives are taken in the directions perpendicular to the fibers of
 $p$, $u_{\infty}$ is constant along the $p$ fibers, and $\Sc_G$
 is the scalar curvature of the fiber metrics $G$.
 Notice that by the constancy of the quantities involved along the $p$
 fibers, and the way we take derivatives, (\ref{eqn: fiber_volume_heat})
 descends to an elliptic equation on $V_x$.
 
 In order to proceed, we further assume for the moment that $N$ is abelian, so
 that $G$ is flat, and the above elliptic equation makes $\ln\det G$ a $\ln
 (\det G)^{-\frac{1}{2}}u_{\infty}$-harmonic function on $V_x$.
 Now we rely on the characterization of $\tilde{\mathcal{S}}$ as the zero locus
 of the non-negative continuous function $\chi_X$ to locate a global maximum
 point of $\chi_X$ within $\mathcal{\tilde{R}}$. Therefore, a maximum principle
 argument around the maximum point $x_0\in \tilde{\mathcal{R}}$ of $\chi_X$ ---
 which furnishes a local maximum of $\ln \det G$ in $V_{x_0}$  --- will lead to
 a contradiction to (\ref{eqn: fiber_volume_heat}) unless $\chi_X$ is constant
 on $V_{x_0}$. But if $\chi_X$ is locally constant on $\tilde{\mathcal{R}}$,
 then by the continuity of $\chi_X$, it has to be a positive constant
 throughout $X$, whence the vacancy of $\tilde{\mathcal{S}}$; see \S 6.2 for
 more details. Once this is shown, we know that $X=\tilde{\mathcal{R}}$ is
 actually a compact Riemannian orbifold --- this is exactly what we hope to
 achieve through Theorem \ref{thm: main1}.
 
 The maximum principle argument we just outlined is originally due to Naber and
 Tian in the proof of \cite[Theorem 1.2]{NaberTian}. In \cite{NaberTian}, an
 $N^{\ast}$-structure has been globally constructed out of the frame bundles
 $\{(FM_i,\bar{g}_i)\}$, where $\bar{g}_i$ is the $O(m)$ invariant metric
 canonically associated with $g_i$. By \cite[Theorem 6.1]{Fukaya88}, we know
 that the collapsing limit is a Riemannian manifold $(Y,g_Y)$ on which $O(m)$
 acts by isometries. Moreover, the collapsing singular fibrations $f_i:M_i\to
 X$ (see \cite[Theorem 0.12]{Fukaya88}) induce corresponding $O(m)$ equivariant
 collapsing fiber bundles $\bar{f}_i:FM_i\to Y$ with fibers being nilmanifolds.
 Therefore the local construction in Theorem~\ref{thm: canonical_nbhd} can be
 extended all over $Y$, and O'Neill's formula for the Ricci curvature of the
 corresponding limit metric can be applied to analyse the global $O(m)$
 equivariant Riemannian submersion structure --- provided that there has
 already been an elliptic equation for the Ricci curvature of the limit metric
 --- this is indeed the case for \cite[Theorem 1.2]{NaberTian}, as the
 collapsing manifolds $\{(M_i,g_i)\}$, to begin with, are assumed to be Ricci
 flat.
 
 In the setting of Theorem \ref{thm: main1}, however, we do not have any
 elliptic equation concerning the Ricci curvature ready at hand; but rather the
 expected elliptic equation is due to the long-time evolution of the Ricci
 flow. We therefore need to adopt the concept of measured Gromov-Hausdorff
 convergence in \cite{Fukaya87} and prove integral convergence results in \S 4,
 to be able to extract a limit gradient steady Ricci soliton
 metric on the unwrapped neighborhoods defined in Theorem~\ref{thm:
 canonical_nbhd}. The main theorem proven in \S 4 is the following:
 \begin{theorem}[Collapsing and convergence of integrals]\label{thm: integral}
 Assume that a sequence $\{(M_i,g_i)\}$ of Riemannian $m$-manifolds, satisfying
 (\ref{eqn: regularity}) and a uniform diameter bound, collapses to $(X,d)$.
 Suppose that there are functions $\rho_i\in C^1(M_i)$ satisfying
 \begin{align*}
 \sup_{M_i} \left|\ln (\rho_i|M_i|_{g_i})\right|\ \le\ \ln C,\quad
 \text{and}\quad \sup_{M_i}|\nabla \rho_i||M_i|_{g_i}\ \le\ C,
 \end{align*}
 and $w_i\in C^k(M_i)$ satisfying $\|w_i\|_{C^k(M_i,g_i)}\le C$, then
 there are continuous functions $\rho_X:X\to [C^{-1},C]$ and $w_{X}$ on $X$,
 such that
 \begin{align*}
 \lim_{i\to \infty} \int_{M_i}w_i\ \rho_i\ \dvol_{g_i}\ =\ \int_Xw_X\ \rho_X\
 \dmu_X.
 \end{align*}
  Moreover, $\forall x_0\in \tilde{\mathcal{R}}$, let $U_{x_0}$, $V_{x_0}$ and
  $W_{x_0}$ be the corresponding neighborhoods in Theorem~\ref{thm:
  canonical_nbhd}, then there is some $w_{\infty}\in
  C_{loc}^{k-1,\alpha}(W_{x_0})$, such that $\lim_{i\to \infty} \tilde{w}_i=
  w_{\infty}$ in the $C^{k-1,1}(W_{x_0})$ topology, where we define
  $\tilde{w}_i:= \tilde{q}_i^{\ast} \hat{q}_i^{\ast} \left(
  w_i|_{f_i^{-1}(U_{x_0})} \right)$ as the pull-back of $w_i$ from
  $f_i^{-1}(U_{x_0})\subset M_i$ to the covering space $W_{x_0}$; furthermore,
  $w_{\infty}$ is constant along the $p$ fibers, and
  $w_{\infty}=p^{\ast}w_X$ on $V_{x_0}$.
 \end{theorem}
 
 We will then rely on the asymptotic vanishing of the derivatives of Perelman's
 $\mathcal{F}$-functioanl \cite{Perelman} and Feldman-Ilmanen-Ni's
 $\mathcal{W}_+$-functional \cite{FIN05} to obtain a limit gradient steady Ricci
 soliton metrics on the unwrapped neighborhoods around the orbifold points in
 $X$; see \S 2.3 and \S 6.1 for more details. Here we emphasize that as pointed
 out in \cite[Page 494]{Lott10}, the induced flow (by the Ricci flow on the
 manifold) on the frame bundle is compliated, let along the evolution of the
 induced functionals. Therefore, compared to the $N^{\ast}$-structure
 constructed in \cite{NaberTian}, the unwrapped neighborhoods obtained in \S 3
 and the integral convergence results in \S 4 better adapt to the setting of
 collapsing and Ricci flows.
 
 Beware, however, that even if we have obtained a gradient steady 
 Ricci soliton limit metric to locally write down an elliptic equation like
 (\ref{eqn: fiber_volume_heat}), we still need to face its possibly negative
 right-hand side, which invalidates the maximum principle argument: in fact, by
 \cite[Theorem 3.1]{Milnor76}, we know that any non-flat left invariant metric
 along the $p$ fibers will have negative scalar curvature, whence the
 negativity of $\Sc_G$, and such metric cannot be flat unless the underlying
 Lie group is abelian.
 
 On the other hand, by \cite[(0.13.2)]{Fukaya88} and (\ref{eqn: chi_X_regular})
 we understand, roughly speaking, that the vanishing of $\chi_X$ at any 
 $x\in\tilde{\mathcal{S}}$ is due to the fact that the singular fiber
 $f_i^{-1}(x)$ is a lower dimensional quotient of the model fibers $F_i$, since
 $f_i^{-1}(x)\approx F_i\slash G_x$ and $G_x$ is of positive dimension.
 Moreover, since a key feature of $G_x$ is that its Lie algebra is contained
 in the center of the Lie algebra of $N_i$ (see \cite[Lemma 5.1]{Fukaya88}), we
 know that the vanishing of $\chi_X$ is caused by the degeneration of the torus
 orbits $\mathbb{T}_i\subset F_i$ as we take quotient of the $G_x^0$ (the
 identity component of $G_x$) action. The importance of understanding these
 torus orbits is also highlighted through the study of the $F$-structure in a
 series of works by Cheeger, Gromov, Rong and others; see, for instance, 
 \cite{CG86,CG90,CR96,CCR01}.
 
 Notice that each torus orbit in any $f_i$ fiber is a sub-manifold
 of $M_i$, and we would wonder if there is another density function defined
 on $X$, in a way similar to (\ref{eqn: chi_X_regular}), that describes the
 limit relative volume distribution of those torus orbits over the collapsing 
 limit space. Such density function should also characterize
 $\tilde{\mathcal{S}}$ as its zero locus, by the same reasoning that implies
 $\tilde{\mathcal{S}} = \chi_X^{-1}(0)$. This is indeed the case, and in \S 5 we
 will prove the following 
 \begin{theorem}[Limit central density]\label{thm: limit_central_density} 
 Assume that a sequence $\{(M_i,g_i)\}$ of Riemannian $m$-manifolds, satisfying
 (\ref{eqn: regularity}) and a uniform diameter bound, collapses to $(X,d)$.
 Then there is a non-negative continuous function $\chi_C:X\to [0,\infty)$,
 such that $\tilde{\mathcal{S}} =\chi_C^{-1}(0)$.
 
 Moreover, $\forall x_0\in \tilde{\mathcal{R}}$, let $V_{x_0}$ and $W_{x_0}$ be
 the neighborhoods constructed in Theorem~\ref{thm: canonical_nbhd}, then
 there is a commutative family of Killing vector vector fields $X_1, \ldots,
 X_{k_0}$ on $W_{x_0}$, tangent to the $p$ fibers, such that
 $q_{x_0}^{\ast}\ \chi_C$ is a constant (only depending on $x_0$) multiple of
 $|X_1\wedge \cdots \wedge X_{k_0}|_{\tilde{g}_{\infty}}$ on $V_{x_0}$.
 \end{theorem} 

 \begin{remark}\label{rmk: local}
 Theorems \ref{thm: integral} and \ref{thm: limit_central_density} enjoy the
 following common flavor: the limit objects are robustly defined over the entire
 $X$ --- they are continuous but are of low regularity; however, around the 
 orbifold points, we can find very regular representations of these quantities
 on the unwrapped neighborhoods, as constructed in Theorem~\ref{thm:
 canonical_nbhd}.
 \end{remark}
 
 At this stage, the natural resolution to the issue of the possibly negative 
 right-hand side of (\ref{eqn: fiber_volume_heat}), as originally noticed in
 the proof of \cite[Theorem 1.2]{NaberTian}, is to focus on the leaves of the
 Riemannian foliation by the commuting Killing vector fields $X_1, \ldots,
 X_{k_0}$. These leaves are intrinsically flat and by applying the O'Neill's
 formula to $\ln |X_1\wedge \cdots \wedge X_{k_0}|_{\tilde{g}_{\infty}}^2$, we
 obtain an elliptic equation similar to (\ref{eqn: fiber_volume_heat}), but with
 vanishing right-hand side. Then we can argue via the maximum principle as
 before, to prove that $\chi_C$ is a positive constant across $X$, and rely on
 Theorem \ref{thm: limit_central_density} to rule out the possible existence of
 the corner singularity.
 
 The proof of Theorem \ref{thm: main2} utilizes the same set of tools: suppose
 $\limsup_{t\to\infty} t\boldsymbol{\mu}_+'(t)=0$ but the global volume ratio
 fails to see a uniform positive lower bound, then for any sequence
 $\{(M,t_i^{-1}g(t_i))\}$ realizing these numerical limits, we have the exact
 same setting as just discussed, except that on the right-hand side of
 (\ref{eqn: fiber_volume_heat}) there is an extra positive term $k_0$, as
 the result of a gradient expanding Ricci soliton metric on locally unwrapped
 neighborhoods (see Proposition~\ref{prop: expanding_soliton_metric}) --- but
 then we could deduce the constacy of $|X_1\wedge \cdots \wedge
 X_{k_0}|_{\tilde{g}_{\infty}}$ via the maximum principle argument, which will
 force $\rank \mathfrak{C}=k_0=0$, whence the non-existence of the
 $F$-structure caused by collapsing, a contradiction to the asymptotic
 degeneration of the global volume ratio; see \S 6.2 for more details.
 
 Besides the proofs of Theorems \ref{thm: main1} and \ref{thm: main2}, we
 believe that the structural results --- Theorems~\ref{thm: canonical_nbhd},
 \ref{thm: integral}, and \ref{thm: limit_central_density} --- will be useful in
 future studies on the metric measure properties of the collapsing limit and
 the collapsing procedure. Especially, in contrast to the global constructions
 carried out in \cite{Lott10, NaberTian}, Theorems \ref{thm: canonical_nbhd}, 
 \ref{thm: integral} and \ref{thm: limit_central_density} are local in nature,
 and should see wider applications; see \S 7.

 For the rest of the paper, we begin with discussing the necessary background on 
 the collapsing geometry and $\mathcal{W}_+$-functional in \S 2. With Theorems
 \ref{thm: canonical_nbhd},  \ref{thm: integral} and \ref{thm:
 limit_central_density} proven in \S 3, \S 4 and \S 5, respectively, we will
 then be ready to prove Theorems \ref{thm: main1} and \ref{thm: main2} in \S 6.
 We will finish the paper with a short proof of Rong's theorem \cite[Thoerem
 0.4]{Rong98} in \S 7, as an application of Theorems \ref{thm: canonical_nbhd}
 and \ref{thm: limit_central_density}.

\section{Background}
In this ection we review and synthesis the relevant facts and fix notations
about the geometry of manifolds that collapse with uniformly controlled
curvature and diameter, as well as the $\mathcal{F}$- and
$\mathcal{W}_+$-functionals along an immortal Ricci flow. 

\subsection{Singular fibration structure associated with the collapsing limit}
Throughout this article we consider a sequence of $m$-dimensional closed
Riemannian manifolds $\{(M_i,g_i)\}$ satisfying (\ref{eqn: regularity}) with
$C_0=1$ and $\diam (M_i, g_i)\le D$, see (\ref{eqn: CD_bound}). We say that the
sequence $\{(M_i,g_i)\}$ collapses to $(X,d)$ with uniformly controlled curvature and
diameter, when there is a metric space $(X,d)$ whose Hausdorff dimension is
$n<m$, and that
\begin{align*}
d_{GH}((M_i,g_i),(X,d))\ =:\ \delta_i\to 0\quad \text{as}\quad i\to \infty.
\end{align*} 
 Our exposition about the collapsing geometry of $(M_i,g_i)$ associated with
 $(X,d)$ will be based on the work of Cheeger, Fukaya and Gromov~\cite{CFG92}
 and the series of works by Fukaya~\cite{Fukaya87, Fukaya87ld, Fukaya88, Fukaya89}.

\subsubsection{Singularity types in the limit}
The limit metric space $(X,d)$ cannot be an aribtrary one. Roughly speaking, the
local structure around any point in $X$ is a quotient of the Euclidean space by a
finitely extended torus action. More specifically, by \cite[Theorem
0.5]{Fukaya88}, we know that for any $x\in X$ there is some open neighborhood
$U$ in $X$, and a compact Lie group $G_x$, admitting a faithful representation
in $O(m)$ and with toral identity component, acting on some open neighborhood 
$V$ of the origin $o\in  \mathbb{R}^l$ ($n\le l\le m$), such that
$(U,d,x)\equiv (V,\bar{g},o)\slash G_x$, with $\bar{g}$ being some
$G_x$-invariant metric on $V$. Especially, $x\in U$ comes from a fixed point of
the $G_x$ action on $V$.

It is therefore convenient to let $\tilde{\mathcal{S}}$ denote the collection of
points in $X$ whose associated isotropy group $G_x$ is not discrete, i.e.
$\tilde{\mathcal{S}}:=\{x\in X:\ \dim G_x>0\}.$ And it follows that
$\tilde{\mathcal{R}}:= X\backslash \tilde{\mathcal{S}}$ is a Riemannian
orbifold, which we call the \emph{orbifold regular part}, since every point in
$\tilde{\mathcal{R}}$ has a neighborhood isometric to the quotient of some open
subsets in $\mathbb{R}^n$ by a finite group action. We denote the regular part
of $\tilde{\mathcal{R}}$ as $\mathcal{R}$, i.e. $\mathcal{R}=\{x\in X:\
G_x=\{Id\}\}$, and we also denote $\mathcal{S}:=X\backslash \mathcal{R}$.
Clearly any $x\in \tilde{\mathcal{R}}\backslash \mathcal{R}$ has its isotropy
group being finite and non-trivial. 

\subsubsection{Singular fibration structure}
To understand the global structure of the collapsing limit, we would like to
relate it to $\{(M_i,g_i)\}$ for all sufficiently large $i$. By \cite[Theorem
0.12]{Fukaya88}, we know that there are continuous maps
$f_i:M_i\twoheadrightarrow X$ which furnish generalized fiber bundle structures:
\begin{enumerate}
  \item there is an infranil manifold $F_i$ such that $\forall x\in
  \mathcal{R}$, $f_i^{-1}(x)$ is diffeomorphic to $F_i$;
  \item if $x\in X\backslash \mathcal{R}$, then $G_x$ acts freely on $F_i$ and
  $f_i^{-1}(x)$ is diffeomorphic to the quotient $F_i\slash G_x$.
\end{enumerate}

In fact, when we focus our attention on $\mathcal{R}$, the restriction 
$f_i:f_i^{-1}(\mathcal{R})\to \mathcal{R}$ is indeed a fiber
bundle over the $n$-dimensional manifold $\mathcal{R}$ with infranil fibers
$F_i$. More precisely, since $(\mathcal{R},g_X)$ is a Riemannian manifold, we
can fix some small $\iota>0$, such that on $\mathcal{R}_{\iota}:=\{x\in
\mathcal{R}:\ d(x,\mathcal{S})\ge\iota\}$, the injectivity radius is bounded
below by $\iota$. Then by \cite[\S 2 and \S 3]{CFG92}, the fibration
$f_i:f_i^{-1}(\mathcal{R}_{\iota})\to \mathcal{R}_{\iota}$ can be chosen to be
sufficiently regular, and consequently, the $f_i$ fibers are not just
diffeomorphic to $F_i$ by arbitrary diffeomorphisms: by the uniform regularity
of $f_i$, each of the $f_i$ fibers is almost flat whenever $i$ is sufficiently
large, and the argument in \cite[\S 3]{CFG92} (see also \cite[\S 5]{Fukaya89}
and \cite{Ruh}) can be carried out to construct smooth connections
$\nabla_i^{\ast}$ (in the notation of \cite[\S 3]{CFG92}) on
$f_i^{-1}(\mathcal{R}_{\iota})$ such that their restrictions
$(\nabla_i^{\ast})_x$ to each $f_i^{-1}(x)$ ($x\in \mathcal{R}_{\iota}$) become
flat connections with parallel torsions. Each fiber $f_i^{-1}(x)$ is then made
in this way into an affine homogeneous space, on which the collection of
$(\nabla_i^{\ast})_x$ parallel vector fields are regarded as \emph{left}
invariant, and the fiber-wise fundamental groups act by affine transformation
on the universal covering, equipped with the naturally lifted connection.


Moreover, such fiber bundle construction can be extended over the orbifold
singularities, as carried out in \cite[\S 7]{Fukaya89}. Locally around an
orbifold singularity $x_0\in \tilde{\mathcal{R}}\backslash \mathcal{R}$, there
is an orbifold neighborhood $U_{x_0}\subset X$ such that such that for some open
neighborhood $V_{x_0}$ of the origin $o\in  \mathbb{R}^n$ and some smooth
Riemannian metric $\hat{g}_X$ on $V_{x_0}$, $G_{x_0}$ acts by discrete
isometries, and that $(U_{x_0},d,x_0)\equiv (V_{x_0},\hat{g}_X,o)\slash G_x$.
The singular fibration $f_i$ can then be chosen as the quotient of a
$G_{x_0}$ equivariant smooth fibration $\hat{f}_i:\hat{V}_{x_0,i} \to
V_{x_0}$. Notice that the finite group $G_{x_0}$ acts simultaneously on
the base $V_{x_0}$ and the $\hat{f}_i$ fibers, and since $f_i^{-1}(U_{x_0})=
\hat{V}_{x_0,i} \slash G_{x_0}$ is smooth, we could equip
$\hat{V}_{x_0,i}$ with the covering metric of $g_i|_{f_i^{-1}(U_{x_0})}$.
Shrinking $U_{x_0}$ to be sufficiently small, we still have uniform regularity
controll of $\hat{f}_i$, and each $\hat{f}_i$ fiber is then an almost flat
manifold. 

To (locally) incorporate the previously described infranil fiber bundle
structure over $U_{x_0}$, we notice that by the construction of the connection
in \cite[\S 3]{CFG92}, it is canonically determined by the underlying metric
structure. Consequently, by the $G_{x_0}$ invariance of the lifted metrics 
on $\hat{V}_{x_0,i}$, the same construction in \cite[\S 3]{CFG92} leads to a
$G_{x_0}$ equivariant connection $\hat{\nabla}_i^{\ast}$ on
$\hat{V}_{x_0,i}$, whose restriction to each $\hat{f}_i$ fiber being
flat with parallel torsion. This connection makes each $\hat{f}_i$ fiber into
an affine homogeneous space, and the group action $G_{x_0}$ is by affine
diffeomorphisms between the $\hat{f}_i$ fibers ove $\hat{V}_{x_0,i}$.

In this article, we call a surjective continuous map $f:M\to X$ an
\emph{infranil fiber bundle} over the Riemannian orbifold $X$, if $f:M\to X$
satisfies \cite[Definition 7-3]{Fukaya89}, and the fiber $F$ is an infranil
manifold equipped with a flat connection $\nabla$ with parallel torsion, with
structure group $G=Aff(F,\nabla)$.

Continuing our discussion around any orbifold point $x_0\in
\tilde{\mathcal{R}}$, with whose isotropy group $G_{x_0}$ identified, via the
connection $\hat{\nabla}_i^{\ast}$, with a finite sub-group of
$Aff(\hat{f}^{-1}(x_0),(\hat{\nabla}^{\ast}_i)_x)$ (see \cite[Proposition
3.6]{CFG92}). Notice that the group of affine diffeomorphisms is isomorphic to
$((N_i)_R\slash C(L_i))\rtimes Aut(\Gamma_i)$ --- here $N_i$ is the universal
covering of $\hat{f}_i^{-1}(x_0)$, made into a simply connected nilpotent Lie
group by equipping with $\tilde{\nabla}^{\ast}_i$, the covering connection of
$(\hat{\nabla}_i^{\ast})_{x_0}$, and fixing a base point; the fiber
fundamental group $\Gamma_i=\pi_i(\hat{f}_i^{-1}(x_0))$ and the group
$(N_i)_R$ of \emph{right} translations, act on $N_i$ by affine diffeomorphisms;
and $C(L_i):=C(N_i)\cap \Gamma_i$ is a sub-group ($\approx \mathbb{Z}^{k_0,i}$)
of $N_i$.
Denoting the quotient torus by $\mathbb{T}_i:=C(N_i)\slash C(L_i)$, we have the short
exact sequence of Lie groups $0\to \mathbb{T}_i\to N_i\slash C(L_i)\to
(N_i\slash C(L_i))\slash \mathbb{T}_i\to 0$, and the quotient group is a simply
connected nilpotent Lie group, whence being torsion free. Consequently, we see
that the action of $G_{x_0}$ on the local fiber bundle
$\hat{f}_i:\hat{V}_{x_0,i}\to V_{x_0}$ is given by a finite group
$S_{x_0,i}\rtimes \Lambda_{x_0,i}$, where $S_{x_0,i} \le \mathbb{T}_i$ acts on
the torus fibers, and $\Lambda_{x_0,i}$ is a finite sub-group of
$Aut(\Gamma_i)$.

\subsubsection{Invariant metric}
The major achievement of the work of Cheeger, Fukaya and Gromov \cite{CFG92} is
the construction of a globally defined Riemannian metric on $M_i$, which
approximates $g_i$ well and is invariant under the infinitesimal action of
$\mathfrak{N}_i$ --- a sheaf of vector fields whose action is determined as
following: integrating the $\nabla_i^{\ast}$ parallel vector fields along the
fibers to obtain germs of \emph{right} translations, and these germs of
\emph{right} translations locally define \emph{right} invariant vector fields,
which specifies the infinitesimal action of $\mathfrak{N}_i$; and the invariance
of the approximating metric amounts to say that these \emph{right} invariant
vector fields are Killing fields. While a main technical difficulty in
\cite{CFG92} involves gluing the locally constructed invariant metrics together
in a conherent and controlled way, in our case, since $f_i$ restricts to an
infranil fiber bundle over $\tilde{\mathcal{R}}_{\iota}$, the approximating
invariant metric is easily constructed by an averaging argument, as done in
\cite[\S 4]{CFG92}. We summarize the relavent results, \cite[Propositions 4.3
and 4.9]{CFG92}, in the following
\begin{prop}[Approximating invariant metric]\label{prop: invariant_metric}
For all $i$ sufficient large, there is an $(N_i)_L$ invariant metric
$g_i^1$ on $f_i^{-1}(\tilde{\mathcal{R}}_{\iota})$, such that for each $l\in
\mathbb{N}$,
\begin{align}\label{eqn: invariant_metric}
\sup_{f_i^{-1}(\tilde{\mathcal{R}}_{\iota})}|\nabla^l(g_i-g_i^1)|\ \le\
C_{2.1}(\{C_{1.3}(l)\},l)\delta_i \iota^{-1-l}.
\end{align} 
Moreover, if $g_i$ is invariant under a compact Lie group action, then so is
$g_i^1$.
\end{prop}
This metric will prove useful in our later arguments of taking various
quotients. By this proposition and \cite[Theorem 2.6]{CFG92}, we know that each
$f_i$ fiber, measured in $g_i^1$, will have the following second fundamental
form control
\begin{align*}
\sup_{x\in \tilde{\mathcal{R}}_{\iota}}\left|II_{f_i^{-1}(x)}\right|_{g_i^1}\
\le\ C_{2.2}\iota^{-1}.
\end{align*}

\subsection{The frame bundle argument}
Associated to a collapsing sequence of manifolds $\{(M_i,g_i)\}$, in \cite[\S
1]{Fukaya88} the coresponding frame bundle manifolds $\{(FM_i, \bar{g}_i)\}$
are defined such that for any $l\in \mathbb{N}$,
\begin{align}\label{eqn: gbar_regularity}
\sup_{FM_i}|\nabla^l\Rm_{\bar{g}_i}|_{\bar{g}_i}\ \le\ \bar{C}_{1.3}(l).
\end{align}
Here the metric $\bar{g}_i$ is defined to make the $TM_i$ directions orthogonal
to the $O(m)$ directions at each point of $FM_i$, and is invariant under the
natural $O(m)$ action, making each $\pi_i:FM_i\to M_i$ a Riemannian submersion
with each $\pi_i$ fiber equipped with the standard metric on $O(m)$. Hereafter
we let $|O(m)|$ denote the correpsonding volume; then $|FM_i|_{\bar{g}_i}
=|M_i|_{g_i}|O(m)|$. It is further shown in \cite[\S 6]{Fukaya88} that this
sequence collapses to a Riemannian manifold $(Y,g_Y)$. Especially, we have the
commutative diagram that determines the singular fibration $f_i:M_i\to X$
discussed in \S 2.1.2:
\begin{align*}
\begin{CD}
FM_i @>\bar{f}_i>> Y\\
\pi_i @VV  V  @VV V \pi_Y\\
M_i @>f_i>> X
\end{CD}
\end{align*}
where $\pi_i$ and $\pi_Y$ are the Riemannian submersions given by taking the
$O(m)$ quotients, and the \emph{smooth} fibration $\bar{f}_i:FM_i\to Y$ is
$\Psi(\delta_i)$-almost $O(m)$ equivariant. The fact that $Y$ is a manifold,
rather than a singular metric space, is essentially due to the fact that local
isometries are determined, around any point, by its $1$-jet at that point. The
frame bundle argument is powerful in that the geometric structure described in
\S 2.2 over the regular part extends over the entire $Y$ as corresponding $O(m)$
equivariant structures. Important geometric applications of the frame bundle
argument, among others, include the classic \cite{CFG92} by Cheeger, Fukaya and
Gromov, where an $N$(ilpotent)-structure is constructed by gluing invariant
metrics on the locally defined frame bundles, and the construction of the
$N^{\ast}$-structure due to Naber and Tian \cite{NaberTian}, ``\emph{in some
sense dual}'' to the $N$-structure.

In our later discussions, it will be convenient to consider the invariant
metrics $\bar{g}_i^1$, naturally associated to the metrics $\bar{g}_i$ by
averaging over the $\bar{f}_i$ fibers (see \cite[(4.8)]{CFG92}), as guaranteed
by Proposition \ref{prop: invariant_metric}. Notice that associated to the
collapsing fibration $\bar{f}_i:FM_i\to Y$, the entire collapsing limit $Y$ is
regular, and thus $\bar{g}_i^1$ is defined globally on $FM_i$ and (\ref{eqn:
invariant_metric}) is valid throughout $Y$. 

Now restricting our attention to each $O(m)$ orbit in $FM_i$, by (\ref{eqn:
invariant_metric}) we could compare the volume of the $\pi_i$ fibers under the
restriction of the approximating invariant metric $\bar{g}_i^1$, with the
volume of $O(m)$ in the standard metric as following:
\begin{align}\label{eqn: O(m)_approximate_volume}
\sup_{M_i}\left|\frac{|\pi_i^{-1}(x)|_{\bar{g}_i^1}}{|O(m)|}-1\right|\ \le\
C_{2.3}\left(\frac{\delta_i}{\iota_Y}\right)^{\frac{m(m-1)}{2}},
\end{align}
where $\iota_Y$ is the injectivity radius of $g_Y$ which has a uniformly
positive lower bound by the compactness of $Y$. Consequently, for any $U\subset
\mathcal{R}_{\iota}$ with $\iota>0$ sufficiently small but fixed, the estimate
(\ref{eqn: O(m)_approximate_volume}) is valid with for any $x\in f_i^{-1}(U)$
with $\iota_Y$ replaced by $\iota$, and we have 
\begin{align} \label{eqn: limit_approximate_metric}
\lim_{i\to \infty}
\frac{\left|FM_i|_{f_i^{-1}(U)}\right|_{\bar{g}_i^1}}{|f_i^{-1}(U)|_{g_i}}\ =\
|O(m)|.
\end{align}

Moreover, in \cite[\S 3]{Fukaya87} the measure theoretic side of the frame
bundle has been explored to define the limit density function $\chi_X$ over
$X$. Since $\bar{f}_i:FM_i\to Y$ are smooth sub-mersions, the density function
\begin{align*}
\chi_Y(y)\ :=\ \lim_{i\to \infty}
\frac{|\bar{f}_i^{-1}(y)|_{\bar{g}_{i}}}{|FM_i|_{\bar{g}_i}}
\end{align*}
is well-defined for any $y\in Y$ (after possibly passing to a sub-sequence). 
Moreover, by the $O(m)$ equivariance of $\bar{f}_i$, we know that $\chi_Y$ is
constant along the $O(m)$ orbits in $Y$, and therefore as
\cite[(3.13)]{Fukaya87}, $\chi_X$ can be determined for $x\in X$ by
\begin{align}\label{eqn: chi_X_integral}
\chi_X(x)\ =\ \int_{\pi_Y^{-1}(x)}\chi_Y\ \darea_{\pi_Y^{-1}(x)},
\end{align}
where $\darea_{\pi_Y^{-1}(x)}$ is the volume form on determined by restricting
$g_Y$ to the sub-manifold $\pi_Y^{-1}(x)$.

In fact, for any $x\in \mathcal{R}$, $\bar{f}_i^{-1}(\pi_Y^{-1}(x))
=\pi_i^{-1}(f_i^{-1}(x))$ is a sub-manifold in $FM_i$, and by Fubini's theorem,
we can compute its volume as 
\begin{align}\label{eqn: FM_Fubini}
\begin{split}
\int_{\pi_Y^{-1}(x)}|\bar{f}_i^{-1}(y)|_{\bar{g}_{i}}\
\darea_{\pi_Y^{-1}(x)}(y)\ =\ &\int_{f_i^{-1}(x)}|\pi_i^{-1}(z)|_{\bar{g}_{i}}\
\darea_{f_i^{-1}(x)}(z)\\
=\ &|O(m)||f_i^{-1}(x)|_{g_{i}},
\end{split}
\end{align}
since each $\pi_i$ fiber is isometric $O(m)$ in its standard metric. 
Therefore, on the regular part of $X$, the definition (\ref{eqn:
chi_X_regular}) agrees with (\ref{eqn: chi_X_integral}): 
\begin{align}\label{eqn: Fubini_f}
\begin{split}
\forall x\in \mathcal{R},\quad 
\int_{\pi_Y^{-1}(x)}\chi_Y\ \darea_{\pi_Y^{-1}(x)}\ 
=\ \lim_{i\to \infty}\frac{|f_i^{-1}(x)|_{g_{i}}}{|M_i|_{g_i}}.
\end{split}
\end{align}

To extend (\ref{eqn: Fubini_f}) over orbifold points, let us fix $x_0\in
\tilde{\mathcal{R}}\backslash \mathcal{R}$ and pick $U_{x_0}\subset
\tilde{\mathcal{R}}$ sufficiently small so that the isotropy group
$G_{x}\le G_{x_0}$ for any $x\in U_{x_0}$, see \cite[Lemma
5.5]{Fukaya88}. Let $V_{x_0}\subset \mathbb{R}^n$ be an orbifold covering
equipped with a Riemannian metric $\hat{g}_X$ which is $G_{x_0}$
invariant and descends to $d$ on $U_{x_0}$ under the quotient map $q_{x_0}$.
Restricting the frame bundle to $f_i^{-1}(U_{x_0})$, we get the $O(m)$
equivariant fibration $\bar{f}_i:FM_i|_{f_i^{-1}(U_{x_0})}\to
\pi_Y^{-1}(U_{x_0})\subset Y$ by nil-manifolds $N\slash L$. Further shrinking
$U_{x_0}$ if necessary, we know that $G_{x_0}$ can be regarded as a
normal sub-group of $O(m)$ (by \cite[(10.2.4)]{Fukaya88} and
\cite[(6.1.10)]{CFG92}), and we have the following commutative diagram:
\begin{align}\label{eqn: diagram_f}
\begin{xy}
<0em,0em>*+{V_{x_0}}="v", <-7em,2em>*+{\hat{V}_{x_0,i}}="w",
<0em,-5em>*+{U_{x_0}}="u", <-7em,-3em>*+{f_i^{-1}(U_{x_0})}="z",
<5em,7em>*+{FM|_{f^{-1}_i(U_{x_0})}}="fm",
<12em,6em>*+{\pi_Y^{-1}(U_{x_0})}="p", 
"w";"v" **@{-} ?>*@{>}?<>(.5)*!/_0.5em/{\scriptstyle \hat{f}_i}, 
"w";"z" **@{-} ?>*@{>} ?<>(.5)*!/^0.5em/{\scriptstyle \slash G_{x_0}\ \ }, 
 "v";"u" **@{-} ?>*@{>} ?<>(.5)*!/^0.5em/{\scriptstyle q_{x_0}}, 
 "z";"u" **@{-} ?>*@{>} ?<>(.5)*!/_0.5em/{\scriptstyle f_i},
  "fm";"p" **@{-} ?>*@{>} ?<>(.5)*!/_0.5em/{\scriptstyle \bar{f}_i},
 "fm";"w" **@{-} ?>*@{>} ?<>(.5)*!/^0.5em/{\scriptstyle \hat{\pi}_i}, 
  "p";"v" **@{-} ?>*@{>} ?<>(.5)*!/^0.5em/{\scriptstyle \hat{\pi}_Y}, 
  "p";"u" **@{-} ?>*@{>} ?<>(.5)*!/_0.5em/{\scriptstyle \pi_Y}, 
  "fm";"z" **@{--} ?>*@{>} ?<>(.5)*!/_0.5em/{\scriptstyle \pi_i}
\end{xy}
\end{align}
Here $\hat{f}_i$ denotes the covering fiberation of $f_i$, which is
$G_{x_0}$ equivariant, and both $\hat{V}_{x_0,i}$ and $V_{x_0}$ are
$|G_{x_0}|$ fold coverings of $f^{-1}_i(U_{x_0})$ and $U_{x_0}$,
respectively. Also $\hat{\pi}_i$ and $\hat{\pi}_Y$ denote the quotient map by
the group action on $O(m)\slash G_{x_0}$. Notice that all the group
actions in (\ref{eqn: diagram_f}) are isometric, if we equip $\hat{V}_{x_0,i}$
with $\hat{g}_{x_0,i}$, the (finite) covering metric of
$g_i|_{f_i^{-1}(U_{x_0})}$.

Since $\hat{f}_i$ is a smooth fibration over $V_{x_0}$, $d_{GH}(\hat{V}_{x_0},
V_{x_0})\le \Psi(\delta_i)$, and $\hat{g}_{x_0,i}$ has the same uniform
sectional curvature bound as $g_i$, we could define, similar to the limit
(\ref{eqn: chi_X_regular}), a density $\hat{\chi}$ on $V_{x_0}$ as
\begin{align}\label{eqn: hat_chi}
\hat{x}\in V_{x_0},\quad 
\hat{\chi}(\hat{x})\ :=\ \lim_{i\to \infty}
\frac{|\hat{f}_i^{-1}(\hat{x})|_{\hat{g}_{x_0,i}}}{|\hat{V}_{x_0,i}|_{\hat{g}_{x_0,i}}}.
\end{align}
Moreover, since $\mathcal{R}\cap U_{x_0}$ is dense in $U_{x_0}$, $\chi_X$
is continuous on $U_{x_0}$, $|f_i^{-1}(x)|_{g_i}=
|\hat{f}_i^{-1}(\hat{x})|_{\hat{g}_{x_0,i}}$ for any $x\in \mathcal{R}\cap
U_{x_0}$ and $\hat{x}\in q_{x_0}^{-1}(x)$, and
$|\hat{V}_{x_0,i}|_{\hat{g}_{x_0,i}}
=|G_{x_0}||f_i^{-1}(U_{x_0})|_{g_i}$, we have
\begin{align}\label{eqn: chi_hat_chi}
q_{x_0}^{\ast}\ \chi_X\ =\ |G_{x_0}|\ \frac{\mu_X(U_{x_0})}{\mu_X(M)}\
\hat{\chi}\quad \text{on}\ V_{x_0}.
\end{align}
Moreover, concerning the approximating invariant metric $\bar{g}_i^1$,
(\ref{eqn: limit_approximate_metric}) is valid for the open set $U\subset
\tilde{\mathcal{R}}_{\iota}$ in this circumstance. In \S 4.2, a similar analysis
will be carried out for a locally constructed central sub-bundle around an
orbifold point.

\subsection{Functionals associated with immortal Ricci flows}
In \cite{Perelman}, Perelman introduced the following $\mathcal{F}$-functional 
\begin{align}
\mathcal{F}(g(t),u(t))\ =\ \int_M \left(|\nabla \ln u(t)|^2+\Sc_{g(t)}\right)\
u(t)\ \dvol_{g(t)},
\end{align}
along the Ricci flow on $M$, where $u(t)\in C^{\infty}(M)$ solves the conjugate
heat equation
\begin{align}\label{eqn: conjugate_heat}
\square^{\ast}u\ :=\ \left(\partial_t+\Delta_{g(t)}-\Sc_{g(t)}\right)\ u\ =\ 0,
\end{align}
which ensures that the measure $u(t)\dvol_{g(t)}$ has a fixed total mass as the
Ricci flow evolves.

Notice here we will always need to fix a finite time interval
$[0,T']\subset[0,T)$, and solve the final value problem for some (generalized)
given function $u_{T'}$ on $M$:
\begin{align}\label{eqn: final_value_problem}
\begin{cases}
\square^{\ast}u\ &=\ 0;\\
u(T')\ &=\ u_{T'}.
\end{cases}
\end{align}
Any solution $u(t)$ provides a desired function in the definition of
$\mathcal{F}(g(t),u(t))$ for $t\in [0,T')$.

The key property of the $\mathcal{F}$-entropy is its monotonicity along the
Ricci flow coupled with (\ref{eqn: conjugate_heat}), more specifically, 
\begin{align}
\mathcal{F}'(g(t),u(t))\ =\ 2\int_M \left|\Rc_{g(t)}-\nabla^2_{g(t)} \ln
u(t)\right|^2\ u(t)\ \dvol_{g(t)},
\end{align}
where we have abrieviated $\mathcal{F}'(g(t),u(t))=
\frac{\text{d}}{\text{d}t}\mathcal{F}(g(t),u(t))$, with the understanding that
$g(t)$ solves the Ricci flow equation and $u(t)$ solves equation (\ref{eqn:
conjugate_heat}).

 Notice that since $\Delta u=\left(|\nabla \ln u|^2+\Delta \ln u\right)u$,
 elementary inequalities together with integration by parts lead to
\begin{align}\label{eqn: F'_lb}
\begin{split}
\mathcal{F}'(g(t),u(t))\ \ge\
&\frac{2}{m}\int_M \left(\Sc_{g(t)}-\Delta_{g(t)}\ln u(t)\right)^2\ u(t)\
\dvol_{g(t)}\\
\ge\ &\frac{2}{m}\mathcal{F}(g(t),u(t))^2.
\end{split}
\end{align}

Now if $(M,g(t))$ is an immortal Ricci flow with uniformly bounded sectional 
curvature, then a solution $u(t)\in C^{\infty}(M\times [0,\infty))$ to
(\ref{eqn: conjugate_heat}) could be constructed as following: Pick any
sequence $t_i\to \infty$, and let $u_i(t)\in C^{\infty}(M\times [0,t_i])$ solve
the final value problems (\ref{eqn: final_value_problem}) on $[0,t_{i}+1]$ with
final value $u_i(t_i+1)=\delta_x$ for an arbitrarily fixed point $x\in M$.
Then for each $i$, we have the uniform magnitude and gradient bound of $u_i$ on
compact subsets of $M\times [0,t_i]$, by \cite[Proposition 5.1]{Foxy1809} and
\cite[Theorem 3.3]{QiZhang06}, and uniform higher regularities are guaranteed
by parabolic bootstrapping. Therefore for any $T>0$ fixed, $\{u_i\}$
sub-converges, uniformly on the compact space-time $M\times [0,T]$, to a
solution to (\ref{eqn: conjugate_heat}), and a diagonal argument gives a desired
limit solution $u(t)\in C^{\infty}(M\times [0,\infty))$ that solves (\ref{eqn:
conjugate_heat}). The uniform curvature bound implies the stochastic
completeness of the limit function, and thus $\int_Mu(t)\ \dvol_{g(t)}=1$  for
any $t\ge 0$.

 With the $u(t)$ just defined, we clearly see that the ordinary differential
 inequality (\ref{eqn: F'_lb}) holds for any $t>0$, and as observed in
 \cite{FIN05}, we must have
\begin{align}\label{eqn: F_lb}
-\frac{2}{m\ t}\ \le\ \mathcal{F}(g(t),u(t))\ \le\ 0.
\end{align}
This ensures that $\lim_{t\to \infty}\mathcal{F}(g(t),u(t))=0$. The asymptotic
vanishing of the $\mathcal{F}$-functional, together with the uniform curvature
bound, will force the asymptotic vanishing of $\mathcal{F}'(g(t),u(t))$ for
immortal Ricci flows with uniformly bounded curvature and diameter. This will
provide the desired (local) gradient steady Ricci soliton equation; see \S 6.1
for more details.


To deal with type-III Ricci flows with diameter growth controlled by
$t^{\frac{1}{2}}$, we need to rescale the metric $g(t)\mapsto t^{-1}g(t)$ to
obtain a meaningful limit space. But notice that the $\mathcal{F}$-functional
is not scaling invariant, making it inconvenient in dealing with the blowdown of
type-III Ricci flows. In \cite{FIN05}, the following $\mathcal{W}_+$-functional
is introduced to handle the rescaling of immortal Ricci flows:
\begin{align*}
\mathcal{W}_+(g(t),u(t),t)\ :=\ t\ \mathcal{F}(g(t),u(t))+\int_Mu(t)\ \ln u(t)\ 
\dvol_{g(t)}+\frac{m}{2}\ln (4\pi t)+m,
\end{align*}
where $u\in C^{\infty}(M\times [0,\infty))$ solves the equation (\ref{eqn:
conjugate_heat}). Clearly, the $\mathcal{W}_+$-functional is invariant
under the rescaling of the metric $g(t)$. Moreover, we notice that for Ricci
flows with diameter growth controlled by $t^{\frac{1}{2}}$, the lack of a
uniform positive lower bound of the global volume ratio in time forces the
$\mathcal{W}_+$-functional to explode as time elapses:
\begin{align}\label{eqn: W+_blow}
\liminf_{t\to \infty}|M|_{g(t)}\diam (M,g(t))^{-m}\ =\ 0\quad \Rightarrow\quad
\lim_{t\to\infty}\mathcal{W}_+(g(t),u(t),t)\ =\ \infty.
\end{align}
This is due to the uniform boundedness of $t\ \mathcal{F} (g(t),u(t))$ for any
$t>0$ on the one hand, and on the other hand, when the global volume
asymptotically degenerates, we see that as $t\to \infty$,
\begin{align*}
\int_Mu(t)\ \ln u(t)\ 
\dvol_{g(t)}+\frac{m}{2}\ln (4\pi)+m\ \ge\ \frac{m}{2}\ln (4\pi
e)-\ln |M|_{g(t)}\diam (M,g(t))^{-m} - \ln D\ \to \infty,
\end{align*} 
where we assume $\diam (M,g(t))\le D t^{\frac{1}{2}}$ and applied
Jensen's inequality; see \cite[Page 53]{FIN05}.

Moreover, the time derivative of the
$\mathcal{W}_+$-functional is computed as following:
\begin{align*}
\mathcal{W}_+'(g(t),u(t),t)\ =\
2\left(\mathcal{F}(g(t),u(t))+\frac{m}{2t}\right) +t\ \mathcal{F}'(g(t),u(t))
-\frac{m}{2t}.
\end{align*}
In \cite{FIN05}, a $\boldsymbol{\mu}_+$-functional is defined as the infimum of
the $\mathcal{W}_+$-functional over all smooth probability densities over
$(M,g(t))$. As shown in \cite[Theorem 1.7 (a)]{FIN05}, for each $t\ge 0$ there
is a unique minimizer $u_t\in C^{\infty}(M)$ such that $\int_Mu_t\
\dvol_{g(t)}=1$ and $\boldsymbol{\mu}_+(t)=\mathcal{W}_+(g(t),u_t,t)$. Moreover,
$u_t$ depends on $t\ge 0$ smoothly. As a consequence, $\boldsymbol{\mu}_+(t)$
varies smoothly in $t$ and we have
\begin{align}\label{eqn: mu+'}
\mu'_+(t)\ =\ 2t\int_M\left|\Rc_{g(t)}-\nabla_{g(t)}^2\ln
u_t+\frac{g(t)}{2t}\right|^2_{g(t)}\ u_t\ \dvol_{g(t)}.
\end{align}

\subsection*{Notational conventions}
Throughout this article, we employ the following notations:
\begin{enumerate}
  \item[---] $\Psi(\delta)$ denotes a positive quantity satisfying
  $\lim_{\delta\to 0}\Psi(\delta)=0$; it also depends on other parameters
  independent of $\delta>0$, and may vary from line to line.
  \item[---] $C_{a.b}(c_1,c_2,\ldots,c_l)$ denotes the constant appeared in item
  $a.b$ and it is determined by the constants $c_1,c_2,\ldots,c_l$.
  \item[---] We will frequently pass to a possible subsequence when considering
  convergence, and we will always use the original notation for the convergent
  subsequence.
  \item[---] We will let $(X_i,d_i) \xrightarrow{GH} (Y,d)$ denote the
  Gromov-Hausdorff convergence for a sequence of compact metric spaces, let
  $(X_i,d_i,x_i)\xrightarrow{pGH} (Y,d,y)$ denote the pointed Gromov-Hausdorff
  convergence, let $(X_i,d_i,\mu_i)\xrightarrow{mGH}(Y,d,\mu)$ denote Fukaya's
  measured Gromov-Hausdorff convergence, and let $(X_i,d_i)\xrightarrow{eGH}
  (Y,d)$ denote the equivariant Gromov-Hausdorff convergence, assuming isometric
  group actions on $X_i$ and $Y$ respectively. Similarly,
  $(M_i,g_i)\xrightarrow{CG}(M,g)$ denotes the Cheeger-Gromov convergence, while
  $(M_i,g_i,x_i)\xrightarrow{pCH}(M,g,x)$ denotes the pointed Cheeger-Gromov
  convergence.
\end{enumerate}

\section{Unwrapping the fibration around orbifold points and taking limits}
This section is devoted to Theorem \ref{thm: canonical_nbhd}, which describes
the local covering structure of a collapsing sequence around an orbifold point
in the collapsing limit. 
We will begin with a discussion on the local infranil fiber bundle structure
around an orbifold point, following \cite{Fukaya89, CFG92}, and unwrap the
fibers while keeping track of the regularity of the related geometric
structures. In order to study the convergence property, we then follow \cite[\S
4]{CFG92} to construct local trivializations by finding a controlled local
section. Finally, with the regularity control of the local trivializations, we
can take the pointed Cheeger-Gromov limits of the unwrapped neighborhoods. The
discussion in this section should be well-known to experts in the field and we
claim no originality here.

Recall that our setting is a sequence of closed Riemannian manifolds
$\{(M_i,g_i)\}$, satisfying (\ref{eqn: regularity}) with $C_0=1$, collapsing to
a metric space $(X,d)$ as $i\to \infty$, and we let $\delta_i:=d_{GH}(M_i,X)$.
We begin with omitting the index $i$ and focus on one of those sufficiently
collapsed manifolds in the sequence.

\subsection{Unwrapping the fibration around orbifold points}
Now we fix $x_0\in \tilde{\mathcal{R}}$ and unwrap the (singular) fibration
around $x_0$. We know that there is an open neighborhood $U\subset
\tilde{\mathcal{R}}_{\iota}$ of $x_0$ (with $\iota\in
(0,d(x_0,\tilde{\mathcal{S}}))$ sufficiently small), a finite group
$G_{x_0}$ and an invariant metric $\hat{g}_{X}$ on some open set
$V\subset \mathbb{R}^n$ such that $(U,d)\equiv (V,\hat{g}_{X})\slash
G_{x_0}$. We can reduce the size of $V$ to get some contractible and
$G_{x_0}$ invariant subset $V_{x_0}\subset V$, such that $\diam
(V_{x_0},\hat{g}_{X})$ does not excess the injectivity radius of $\hat{g}_{X}$.
By \cite[Lemma 5.5]{Fukaya88}, we may further shrink $V_{x_0}$ and ensure that
$\forall x\in U_{x_0}:=V_{x_0}\slash G_{x_0}$, the isotropy group
$G_x$ of $x$ satisfies $G_x \le G_{x_0}$. We let
$q_{x_0}:V_{x_0}\to U_{x_0}$ denote the quotient map.
Since $q_{x_0}^{-1}(x_0)\subset V_{x_0}$ consists of a single point, we still
denote that point by $x_0\in V_{x_0}$.

Moreover, by (\ref{eqn: regularity}) and \cite[\S 3, \S 5, \S 7 and
\S 10]{Fukaya88}, we know that $\hat{g}_X$ has the uniform regularity control
\begin{align}\label{eqn: hatg_X_regularity}
\forall l\in \mathbb{N},\quad
\sup_{V_{x_0}}|\nabla^l\Rm_{\hat{g}_X}|_{\hat{g}_X}\ \le\ \hat{C}_{1.3}(l,x_0).
\end{align}

 As discussed in \S 2.1.2, there is a $G_{x_0}$ equivariant fibration
 $\hat{f}:\hat{V}_{x_0}\to V_{x_0}$ that becomes a $|G_{x_0}|$ fold
 covering of the singular fibration $f:f^{-1}(U_{x_0})\to U_{x_0}$, where
 $f^{-1}(U_{x_0})\subset M$ is $\delta$ Gromov-Hausdorff close to
 $U_{x_0}\subset X$, with the subspace metrics. The action of $G_{x_0}$ is
 discrete and free on the total space $\hat{V}_{x_0}$, having $f^{-1}(U_{x_0})$
 as the smooth quotient; therefore the metric $g|_{f^{-1}(U_{x_0})}$ can be
 lifted to a covering metric $\hat{g}_{x_0}$, whose regularity is readily
 controlled in the same way as (\ref{eqn: regularity}):
 \begin{align}\label{eqn: hat_g_regularity}
 \forall l\in  \mathbb{N},\quad
 \sup_{\hat{V}_{x_0}}|\Rm_{\hat{g}_{x_0}}|_{\hat{g}_{x_0}}\ \le\ C_{1.3}(l),
 \end{align} 
 with $C_0=1$. Further shrinking $V_{x_0}$ if necessary, by \cite[Theorem
 2.6]{CFG92} we could also choose $\hat{f}$ so that its regularity is
 controlled, with respect to the metrics $\hat{g}_{x_0}$ and $\hat{g}_X$, as
 following:
 \begin{align}\label{eqn: hat_f_regularity}
\forall l\in \mathbb{N},\quad  \sup_{\hat{V}_{x_0}}|\nabla^l \hat{f}|\ \le\
C_{3.3}(l)r_{x_0}^{1-l},\quad
\text{and}\quad \sup_{\hat{x}\in
V_{x_0}}|II_{\hat{f}^{-1}(\hat{x})}|_{\hat{g}_{x_0}}\ \le\ C_{3.3} r_{x_0}^{-1}.
\end{align}
Notice that here $\nabla \hat{f}=D\hat{f}$ is a section of the vector bundle
$\hat{f}^{\ast}TV_{x_0}\otimes T^{\ast}\hat{V}_{x_0}$ over $V_{x_0}$, and its
magnitude is measured with respect to the natural bundle metric
$\hat{g}_{X}\otimes \hat{g}_{x_0}$; similarly for any $l> 1$,
$|\nabla^l\hat{f}|$ is measured as the magnitude of the tensor field
$\nabla^{l-1}(D\hat{f})$ over $V_{x_0}$.
 
 Moreover, by \cite[\S 7]{Fukaya89} via the frame bundle argument, or by
 \cite[\S 3]{CFG92} via an averaging argument, on $\hat{V}_{x_0}$ there is a
 smooth connection $\hat{\nabla}^{\ast}$ which restricts to each $\hat{f}$
 fiber to be a flat connection with parallel torsion. As explained in \cite[\S
 3]{CFG92}, this connection is canonically associated to the $G_{x_0}$
 invariant metric $\hat{g}_{x_0}$, thus being equivariant under the
 $G_{x_0}$ action, and descends to a smooth connection $\nabla^{\ast}$ on
 $U_{x_0}\cap \mathcal{R}$, being fiber-wise flat with parallel torsion. The
 regulariy of $\hat{\nabla}^{\ast}$ is readily controlled as in
 \cite[Proposition 3.6]{CFG92}, when compared against the Levi-Civita
 connection $\hat{\nabla}^{LC}$ of $\hat{g}_{x_0}$:
 \begin{align}\label{eqn: hat_nabla_regularity}
 \forall l\in \mathbb{N},\quad
\sup_{\hat{V}_{x_0}}|\nabla^l(\hat{\nabla}^{\ast}
-\hat{\nabla}^{LC})|_{\hat{g}_{x_0}}\ \le\ C_{3.4}(l)\delta r_{x_0}^{-2-l}.
\end{align}
The above mentioned $G_{x_0}$ equivariance of $\hat{\nabla}^{\ast}$ means
that the tensor field $\hat{\nabla}^{\ast}-\hat{\nabla}^{LC}$ is $G_{x_0}$
equivariant.
 
 From the discussion on \cite[Page 346]{CFG92}, we know that each fiber
 $(\hat{f}^{-1}(\hat{x}), \nabla^{\ast}_{\hat{x}})$ is affine diffeomorphic to a
 model space $(N\slash \Gamma, \nabla^{can})$. Here $N$ is an
 $(m-n)$-dimensional simply connected nilpotent Lie group, with group structure
 defined by specifying those $\nabla^{can}$-parallel vector fields as
 \emph{left} invariant, and the fundamental group of each $\hat{f}$ fiber is
 isomorphic to some $\Gamma\le Aff(N,\nabla^{can})$, as a finite extension of a
 cocompact lattice sub-group $L\le N$. Moreover, $N$ canonically defines two
 groups $N_L,N_R\le Aff(N,\nabla^{can})$ respectively, as the group of
 \emph{left} translations and \emph{right} translations by elements in $N$; see
 also \cite[Remark 3.1]{CFG92}.  The identification of each $\hat{f}$
 fiber with $(N\slash \Gamma,\nabla^{can})$ is provided by a local
 trivialization $\phi_{x_0}:V_{x_0}\times \hat{f}^{-1}(x_0)\to \hat{V}_{x_0}$ of
 the smooth fibration $\hat{f}:\hat{V}_{x_0}\to V_{x_0}$, and in the next
 sub-section we will construct such trivialization with uniformly controlled
 regularity.
 
 Let us now take a more detailed look at each $\hat{f}$ fiber. If we are at a
 regular point $x\in U_{x_0}\cap \mathcal{R}$, then $\forall \hat{x}\in
 q_{x_0}^{-1}(x) \subset V_{x_0}$, the fiber $(\hat{f}^{-1}(\hat{x}),
 \hat{\nabla}^{\ast}_{\hat{x}})$ is affine diffeomorphic to $(N\slash
 \Gamma,\nabla^{can})$. On the other hand, for any orbifold singular point
 $x\in U_{x_0}\backslash \mathcal{R}$, $q_{x_0}^{-1}(x)\subset V_{x_0}$
 consists of $|G_{x_0}\slash G_{x}|$ points, and over each $\hat{x}\in
 q_{x_0}^{-1}(x)$, the fiber $\hat{f}^{-1}(\hat{x})$ is a $|G_x|$
 fold covering of $f^{-1}(x) \approx (N\slash \Gamma)\slash G_x$ --- this
 is in accordance with the picture discribed in \cite[Theorem 0.12]{Fukaya88}. 
 Moreover, by \cite[(10.2.4)]{Fukaya88} and \cite[(6.1.10)]{CFG92}, we know that
 $G_x$, the isotropy group of $x\in U_{x_0}$, can be regarded as a
 sub-group of the holonomy action of $Aff(\hat{f}^{-1}(\hat{x}),
 \hat{\nabla}^{\ast}_{\hat{x}}) \approx (N_R\slash C(L))\rtimes Aut(\Gamma)$,
 or equivalently speaking, $G_{x}\le Aut(\Gamma)$.

Now we consider the universal covering space $\tilde{V}_{x_0}$ of
$\hat{V}_{x_0}$, which fibers over $V_{x_0}$ by the universal covering
$N_{\hat{x}}$ of $\hat{f}^{-1}(\hat{x})$, for any $\hat{x}\in V_{x_0}$. We let
$\tilde{q}:\tilde{V}_{x_0}\to \hat{V}_{x_0}$ denote the covering map, and let
$\tilde{f}:\tilde{V}_{x_0}\to V_{x_0}$ denote the covering fibration, so that
$\tilde{f}^{-1}(\hat{x})=N_{\hat{x}}$. Notice that since $\hat{V}_{x_0}$ is a
fiber bundle over the contractible base $V_{x_0}$, and every $\hat{f}$ fiber has
fundamental group isomorphic to $\Gamma$, it is also the fundamental group of
$\hat{V}_{x_0}$. Moreover, we since $G_{x_0}$ acts freely on
$\hat{V}_{x_0}$, sending $\hat{f}$ fibers to $\hat{f}$ fibers, the induced
action on $\tilde{V}_{x_0}$ sends $\tilde{f}$ fibers to $\tilde{f}$
fibers, and preserves the $\Gamma$ orbits within the correpsonding $\tilde{f}$
fibers. This makes $\tilde{f}:\tilde{V}_{x_0}\to V_{x_0}$ a $G_{x_0}$
equivariant fibration.

 We also lift the metric $\hat{g}_{x_0}$ to the covering metric 
 $\tilde{g}_{x_0} =\tilde{q}^{\ast}\hat{g}_{x_0}$, as well as the connection
 $\hat{\nabla}^{\ast}$ to the covering connection $\tilde{\nabla}^{\ast}
 =\tilde{q}^{\ast}\hat{\nabla}^{\ast}$. Since the $\Gamma$ action is discrete
 and free, the regularity control of the lifted structures, being
 infinitesimal in nature, are readily checked just as before. Especially, the
 lifted metric satisfies the same regularity control as before:
\begin{align}\label{eqn: tilde_g_regularity}
\forall l\in \mathbb{N},\quad
\sup_{\tilde{V}_{x_0}}|\nabla^l\Rm_{\tilde{g}_{x_0}}|_{\tilde{g}_{x_0}}\ \le\
C_{1.3}(l);
\end{align}
the regularity of $\tilde{f}:(\tilde{V}_{x_0},\tilde{g}_{x_0})\to
(V_{x_0},\hat{g}_X)$ is controlled as
\begin{align}\label{eqn: tilde_f_regularity}
\forall l\in \mathbb{N},\quad  \sup_{\tilde{V}_{x_0}}|\nabla^l \tilde{f}_i|\
\le\ C_{3.3}(l)r_{x_0}^{1-l},\quad
\text{and}\quad \sup_{\hat{x}\in V_{x_0}}
|II_{\tilde{f}_i^{-1}(\hat{x})}|_{\tilde{g}_{x_0}}\ \le\ C_{3.3} r_{x_0}^{-1};
\end{align}
and consequently, the lifted connection $\tilde{\nabla}^{\ast}$ satisfies the
regularity control
 \begin{align}\label{eqn: tilde_nabla_regularity}
\forall l\in \mathbb{N},\quad
\sup_{\tilde{V}_{x_0}}|\nabla^l(\tilde{\nabla}^{\ast}-\tilde{\nabla}^{LC})|\
\le\ C_{3.4}(l)\delta r_{x_0}^{-2-l},
\end{align} 
where $\tilde{\nabla}^{LC}$ is the Levi-Civita connection of $\tilde{g}_{x_0}$.
Clearly, the fundamental group $\Gamma$ acts by isometries on
$(\tilde{V}_{x_0}, \tilde{g}_{x_0})$; and under the $G_{x_0}$ action, the
metric $\tilde{g}_{x_0}$ is invariant and the tensor field
$\tilde{\nabla}^{\ast} -\tilde{\nabla}^{LC}$ is equivariant, whence the
$G_{x_0}$ equivariance of $\tilde{\nabla}^{\ast}$. Moreover, on each  
$\tilde{f}^{-1}(\hat{x})$, the restriction $\tilde{\nabla}^{\ast}_{\hat{x}}$ of
the connection $\tilde{\nabla}^{\ast}$ makes it an affine homogeneous space
$(N_{\hat{x}}, \tilde{\nabla}_{\hat{x}}^{\ast})$, and $\Gamma$ acts as affine
isometries on these fibers. Finally, by Malcev's rigidity \cite[Theorem
3.7]{CFG92}, elements of $G_{x_0}$ act as affine diffeomorphisms between
the $\tilde{f}$ fibers, since they preserve the $\Gamma$ orbits, which contain
cocompact lattices.

Summarizing the discussion in this sub-section, we have the following
\begin{prop}[Unwrapped neighborhoods around orbifold points]\label{prop:
unwrapping} 
Let $\{(M_i,g_i)\}$ be a sequence of $m$-dimensional Riemannian manifolds
collapsing to $(X,d)$ with unform regularity control (\ref{eqn: regularity})
with $C_0=1$. For any $x_0\in \tilde{\mathcal{R}}$ and any $i$ sufficiently
large, there are the following data determined by $x_0$ and $X$:
\begin{enumerate}
  \item[(1)] $r_0>0$ sufficiently small,
  \item[(2)] $U_{x_0}\subset \tilde{\mathcal{R}}$ an open neighborhood of
  $x_0$;
  \item[(3)] a contractible open set $V_{x_0}\subset \mathbb{R}^n$ equipped
  with a Riemannian metric $\hat{g}_{X}$, on which a finite group
  $G_{x_0}$ acts as isomtries,
  \item[(4)] a sequence of surjective continuous maps $f_i:f_i^{-1}(U_{x_0})
  \to U_{x_0}$, with $f_i^{-1}(U_{x_0})\subset M_i$,
  \item[(5)] a sequence of $G_{x_0}$ equivariant smooth sub-mersions
  $\hat{f}_i: \hat{V}_{x_0,i}\to V_{x_0}$, where $\hat{V}_{x_0,i}$ is a
  $|G_{x_0}|$ fold covering of $f_i^{-1}(U_{x_0})$ with covering maps
  $\hat{q}_i$, and sequences of $G_{x_0}$ invariant Riemannian metrics
  $\hat{g}_{x_0,i}$ and $G_{x_0}$ equivariant connections
  $\hat{\nabla}^{\ast}_i$ on $\hat{V}_{x_0,i}$, and
  \item[(6)] a sequence of $G_{x_0}$ equivariant smooth sub-mersions
  $\tilde{f}_i:\tilde{V}_{x_0,i}\to V_{x_0}$, where $\tilde{V}_{x_0,i}$ is the
  universal covering of $\hat{V}_{x_0,i}$ with covering maps $\tilde{q}_i$,
  together with sequences of $G_{x_0}$ invariant Riemannian metrics
  $\tilde{g}_{x_0,i}$ and $G_{x_0}$ equivariant connections
  $\tilde{\nabla}_i^{\ast}$, which are the pull-back of $\hat{g}_{x_0,i}$ and
  $\hat{\nabla}_i^{\ast}$ by $\tilde{q}_i$, respectively,
\end{enumerate}
to the following effects:
\begin{enumerate}
  \item[(7)] $(V_{x_0},o,\hat{g}_X)\slash G_{x_0}\equiv (U_{x_0},x_0,d)$ with
  quotient map $q_{x_0}$ (an orbifold covering map);
  \item[(8)] $(\hat{V}_{x_0},\hat{g}_{x_0,i})\slash G_{x_0}\equiv
  (f_i^{-1}(U_{x_0}),g_i)$ with the quotient map given by $\hat{q}_i$;
  \item[(9)] $\hat{\nabla}_i^{\ast}$ restricts to each $\hat{f}^{-1}(\hat{x})$
  to be a flat connection with parallel torsion, making it an infranil manifold, and
  each $\hat{f}^{-1}(\hat{x})$ is a finite covering of $f_i^{-1}(x)$,
  whenever $x\in U_{x_0}$ and $\hat{x}\in q_{x_0}^{-1}(x)$;
  \item[(10)] $\tilde{\nabla}_i^{\ast}$ restrict to each $\tilde{f}$ fiber to
  be a flat connection with parallel torsion, making it affine diffeomorphism to a
  $(m-n)$-dimensional simply connected nilpotent Lie group;
  \item[(11)] finally, we have the regularity estimates (\ref{eqn:
  hat_g_regularity}), (\ref{eqn: hat_f_regularity}), (\ref{eqn:
  hat_nabla_regularity}), (\ref{eqn: tilde_g_regularity}), (\ref{eqn:
  tilde_f_regularity}) and (\ref{eqn: tilde_nabla_regularity}) hold.
\end{enumerate}
\end{prop}
As an illustration, we have the following commutative diagram:
\begin{align}\label{eqn: diag_unwrapping}
\begin{xy}
<0em,0em>*+{\hat{V}_{x_0,i}}="x", <-7em, 0em>*+{\tilde{V}_{x_0,i}}="w",
<0em,-5em>*+{V_{x_0}}="v", <7em,0em>*+{f_i^{-1}(U_{x_0})}="y",
<7em,-5em>*+{U_{x_0}}="u", 
 "w";"x" **@{-} ?>*@{>}?<>(.5)*!/_0.5em/{\scriptstyle \tilde{q}_{i}}, 
 "w";"v" **@{-} ?>*@{>} ?<>(.5)*!/_0.5em/{\scriptstyle \tilde{f}_i}, 
 "y";"u" **@{-} ?>*@{>} ?<>(.5)*!/^0.5em/{\scriptstyle f_i}, 
 "v";"u" **@{-} ?>*@{>} ?<>(.5)*!/_0.5em/{\scriptstyle q_{x_0}},
 "x";"y" **@{-} ?>*@{>} ?<>(.5)*!/_0.5em/{\scriptstyle \hat{q}_i}, 
  "x";"v" **@{-} ?>*@{>} ?<>(.5)*!/_0.5em/{\scriptstyle \hat{f}_i}, 
\end{xy}
\end{align}
Here $\hat{q}_i$ and $\tilde{q}_i$ are covering maps, and $\hat{f}_i$ and 
$\tilde{f}_i$ are smooth $G_{x_0}$ equivariant fibrations. For all $i$
large enough, $f_i^{-1}(U_{x_0})\subset M_i$ is equipped with
$g_i|_{f_i^{-1}(U_{x_0})}$ and $f_i:f_i^{-1}(U_{x_0})\to U_{x_0}$ is a singular
fibration and a $\Psi(\delta)$ Gromov-Hausdorff approximation as well.

\begin{remark}\label{rmk: NaberTian}
The idea of unwrapping the infranil fibers around regular points is indeed very
natural, and in \cite{NaberTian4d, NaberTian}, such considerations have
already been taken; see especially \cite[Theorem 2.1]{NaberTian4d} and
\cite[Theorem 1.1]{NaberTian}. We notice that these theorems are proven for
$O(m)$ equivariant (smooth) fibrations over the regular part, and their
extensions to the general case rely on the frame bundle argument. In
comparison, Proposition \ref{prop: unwrapping} works directly on the orbifold
part; it is less general than the above mentioned results (as it does not give
the structure around the corner singularity), but works without involving any
auxiliary strucutre like the frame bundle --- a feature important in the setting
of Ricci flows; see also \cite[Per. 3 of Page 494]{Lott10}.
\end{remark}

\subsection{Controlled local trivialization around the orbifold points}
The above mentioned identification of each $\hat{f}$ fiber to a model space
$(N\slash \Gamma, \nabla^{can})$ via affine diffeomorphisms is realized by a
local trivialization of the fibration $\hat{f}$, as done in \cite[\S 4]{CFG92}.
Moreover, in order to consider the pointed Cheeger-Gromov limit of a sequence
of unwrapped neighborhoods with a uniform regularity control, we also need to
consider a fixed topological space on which the metrics and connections
converge in the $C^{\infty}$ sense, and this depends on the local
trivialization of the unwrapped fibration $\tilde{f}$. 
 Our discussion here closely follows \cite[\S 4]{CFG92}, and we include this
 part here for the convenience of readers. 
 
Fixing an orbifold point $x_0\in \tilde{\mathcal{R}}$, we will follow the
notations of Proposition~\ref{prop: unwrapping} and (\ref{eqn: diag_unwrapping})
to construct a controlled local trivialization of $\hat{f}:
\hat{V}_{x_0}\to V_{x_0}$. We pick a smooth section $s:V_{x_0}\to
\hat{V}_{x_0}$ and define a smooth trivialization of $\hat{V}_{x_0}$ by a
family of fiber-wise affine maps:
\begin{align*}
\phi_{x_0}: V_{x_0} \times \hat{f}^{-1}(o)\ &\to\ \hat{V}_{x_0}\\
(\hat{x},\ [\xi])\quad &\mapsto\ [\psi_{\hat{x}}(\xi)], 
\end{align*}
where $\forall \hat{x}\in V_{x_0}$ and $\forall \xi\in N_{\hat{x}}$,
$[\xi]=\Gamma\xi$ is the equivalence class of $\xi$ in $N_{\hat{x}}\slash
\Gamma$.
Here $\psi_{\hat{x}}: N_{o}\to N_{\hat{x}}$ is the unique affine
diffeomorphism (with respect to the affine structure defined respectively by
the fiber-wise flat connections $\tilde{\nabla}^{\ast}_{o}$ on $N_{o}$ and
$\tilde{\nabla}^{\ast}_{\hat{x}}$ on $N_{\hat{x}}$) determined by the conditions
\begin{align}\label{eqn: local_psi}
\psi_{\hat{x}}(\tilde{s}(o))\ =\ \tilde{s}(\hat{x})\quad
\text{and}\quad \forall \gamma\in L_L,\ \psi_{\hat{x}}\gamma\ =\
\gamma\psi_{\hat{x}}.
\end{align} 
Here recall that $L_L=N_L\cap \Gamma$ is a cocompact lattice in $N$ acting on
$(N_{\hat{x}},\tilde{\nabla}^{\ast}_{\hat{x}})$ by \emph{left} translations.
Also $\tilde{s}: V_{x_0}\to \tilde{V}_{x_0}$ is a lift of the local section $s$
to the universal covering $\tilde{V}_{x_0}$ of $\hat{V}_{x_0}$. The unique
existence of such affine homeomorphism is guaranteed by Malcev's rigidity
theorem, see \cite[Theorem 3.7]{CFG92}; in fact, the same rigidity also implies
that $\forall \gamma\in N_L$, $\gamma \psi_{\hat{x}}= \psi_{\hat{x}}\gamma$.
 
As a more detailed description, we notice that $\psi_{\hat{x}}$ is uniquely
determined by
\begin{align}\label{eqn: initial_psi}
\psi_{\hat{x}}(\tilde{s}(o))\ =\ \tilde{s}(\hat{x})\quad
\text{and}\quad D_{\tilde{s}(o)}\psi_{\hat{x}}:T_{\tilde{s}(o)}N_{x_0}\to
T_{\tilde{s}(\hat{x})}N_{\hat{x}}
\end{align} 
in the following manner: given $\xi\in N_{o}$, there is a unique
tangent vector vector $v\in T_{\tilde{s}(o)}N_{o}$ such that the integral
curve $c_{\xi}$ of the \emph{right} invariant vector field (determined by
$\tilde{\nabla}_{o}^{\ast}$) with initial data $(\tilde{s}(o), v)$ satisfies
$(c_{\xi}(0),\dot{c}_{\xi}(0))=(\tilde{s}(o),v)$ and $c_{\xi}(1)=\xi\in
N_{o}$; now the tangent vector $D_{\tilde{s}(o)}\psi_x\cdot v\in
T_{\tilde{s}(\hat{x})}N_{\hat{x}}$ determines a unique \emph{right} invariant
vector field on $N_{\hat{x}}$ (by the connection $\nabla_{\hat{x}}^{\ast}$ along
the fiber), and the integral curve $c^{\hat{x}}_{\xi}$ with initial data
$(\tilde{s}(\hat{x}), D_{\tilde{s}(o)} \psi_{\hat{x}}\cdot v)$ gives the desired
image $c^{\hat{x}}_{\xi}(1)=:\psi_{\hat{x}}(\xi)$.

With the local trivialization $\phi_{x_0}$, we now identify the model simply
connected nilpotent Lie group $N=(\tilde{f}^{-1}(o),
\tilde{\nabla}_{o}^{\ast}, \tilde{s}(o))$. The affine diffeomorphisms
$\psi_{\hat{x}}$ defined by (\ref{eqn: initial_psi}) provide the desired
identification between $(N,\nabla^{can})$ and $(N_{\hat{x}},
\tilde{\nabla}_{\hat{x}}^{\ast}, \tilde{s}(\hat{x}))$ for any $\hat{x}\in
V_{x_0}$. The action of $\Gamma$, $N_L$ and $N_R$ on each $N_{\hat{x}}$ are
defined accordingly for any $\hat{x}\in V_{x_0}$, and by (\ref{eqn: local_psi})
we know $\psi_{\hat{x}}$ descends to the identification of the affine
homogeneous spaces $(N\slash \Gamma,\nabla^{can})$ with
$(\hat{f}^{-1}(\hat{x}),\nabla^{\ast}_{\hat{x}})$ for each $\hat{x}\in V_{x_0}$.

With the above understanding, we know that choices of local sections canonically
determine local trivializations both for $\hat{f}$ and $\tilde{f}$. Therefore,
we do not need to stick to a specific trivialization, but rather we consider
$\hat{V}_{x_0}$, together with a smooth local section $s:V_{x_0}\to
\hat{V}_{x_0}$. We could define a local section $s$ by picking any base point
$s(o)\in \hat{f}^{-1}(o)$ and extend it in all directions via the normal
exponential map, and it can be lifted to a local section $\tilde{s}:
V_{x_0}\to \tilde{V}_{x_0}$ into the covering space. To check the regularity
of the section $s$, we notice that locally we can pick orthonormal tangent
vectors $E_1,\ldots, E_n \perp T_{s(o)}\hat{f}^{-1}(o)$ so that
$D_{s(o)}\hat{f}.E_1, \ldots, D_{s(o)}\hat{f}.E_n$ form a $\Psi(\delta)$-almost
orthonormal basis of $T_{o}V_{x_0}$. Then $s$ is essentially the composition of
the exponential maps $\exp_{x_0}^{-1}$ and $\exp_{s(o)}$:
\begin{align*}
s:\ \hat{x}=\exp_{x_0} \left(\sum_{j=1}^nt_jD_{s(o)}\hat{f}.E_j\right)\ \mapsto\
\exp_{s(o)}\left(\sum_{j=1}^nt_jE_j\right)\ =:\ s(\hat{x}).
\end{align*} 
The regularity of $s$ is then given by the regularity of $\exp_{x_0}$ and
$\exp_{s(o)}$, which are uniformly controlled by those of $\hat{g}_X$ and
$\hat{g}_{x_0}$, respectively, and thus
\begin{align}\label{eqn: s_regularity}
\forall l\in \mathbb{N},\quad |\nabla^l s|\ \le\ C_{3.11}(l)\quad
\text{and}\quad |\nabla^l \tilde{s}|\ \le\ C_{3.11}(l),
\end{align}
where $\nabla s=D s\in \Gamma(V_{x_0},s^{\ast}T\hat{V}_{x_0}\otimes
T^{\ast}V_{x_0})$ is a tensor field and $\nabla^{l+1}s=\nabla^l(Ds)$, and
$\tilde{s}$ is a lift of $s$ to the universal covering of $\hat{V}_{x_0}$. 

\subsection{Pointed Cheeger-Gromov limit of the unwrapped neighborhoods} Recall
that our purpose of unwrapping the singular fibration around orbifold points is
to locally obtain a unform lower bound of the injectivity radius, so that we
could take pointed Cheeger-Gromov limits. Based on the discussion in the last
sub-section, the pointed Cheeger-Gromov convergence of the unwrapped
neighborhoods will be realized by the controlled local trivialization.
Consequently, we will also study the limit structure and define the limit
central distribution.

 \subsubsection{The pointed Cheeger-Gromov convergence}
 Let us fix $x_0\in \tilde{\mathcal{R}}$, and let $U_{x_0}$ be the
 corresponding neighborhoods defined in the last sub-section --- notice that if
 $x_0\in \mathcal{R}$, then the construction in \S 3.2 for orbifold applies
 with $G_{x_0}= \{Id\}$, and the resulting structures and estimates
 coincide with those in \S 3.1. Now let $V_{x_0}$, $\tilde{V}_{x_0}$ and
 $\tilde{s}:V_{x_0}\to \tilde{V}_{x_0}$ denote, respectively, the orbifold
 covering, the (fiber-wise) universal covering of $\hat{V}_{x_0}$, and the lifted
 local section of the fibration  $\tilde{f}: \tilde{V}_{x_0} \to V_{x_0}$.
 Recall that such data canonically defines a local trivialization
 $\phi_{x_0}:V_{x_0}\times \hat{f}^{-1}(o)\to \hat{V}_{x_0}$; and in fact, by
 (\ref{eqn: local_psi}) and (\ref{eqn: initial_psi}), we know that the
 definition of $\phi_{x_0}$ canonically extends over the entire universal
 covering of $\hat{f}^{-1}(o)$:
 the fiber identifications $\psi_{\hat{x}}$ defined in (\ref{eqn: initial_psi})
 are indeed defined for the universal coverings of the $\hat{f}$ fibers. Let us
 denote by $\tilde{\phi}_{x_0}:V_{x_0}\times N_{x_0}\to \tilde{V}_{x_0}$ the
 corresponding covering trivialization. 
We also identify $N=(N_{o},\tilde{\nabla}_{o},\tilde{s}(o))$ as the simply
connected nilpotent Lie group acting on $\tilde{V}_{x_0}$ by \emph{left}
translations on each $\tilde{f}$ fibers, and regard the fundamental group of
$\hat{f}^{-1}(o)$ as $\Gamma \le Aff(N)$. Notice the simply connected
group $N$ is nothing but the pointed topological space
$(\mathbb{R}^{m-n},\vec{o})$ equipped with a group structure determined by
$\tilde{\nabla}^{\ast}_{o}$, and consequently we have the identification 
\begin{align*}
\left(V_{x_0}\times N_{o},(o,\tilde{s}(o))\right)\ \approx\ 
\left(V_{x_0}\times \mathbb{R}^{m-n},(o,\vec{o})\right)
\end{align*}
 as pointed topological manifolds. Denoting $W_{x_0}:=V_{x_0}\times
 \mathbb{R}^{m-n}$ and equip it with the pull-back metric  $\tilde{g}:=
 \phi_{x_0}^{\ast}\tilde{g}_{x_0}$, and the pull-back connection
 $\tilde{\nabla}:=\phi_{x_0}^{\ast}\tilde{\nabla}^{\ast}$, we recover
 $V_{x_0}\times N_{o}$ with the same pull-back metric $\tilde{g}$ and the
 product affine structure determined by extending $\tilde{\nabla}^{\ast}_{o}$
 trivially in the $V_{x_0}$ directions. Letting $p:W_{x_0}\to V_{x_0}$ denote
 the projection onto the first factor, we have  $N_{o}= \tilde{f}^{-1}(o)=
 p^{-1}(o)$ and
\begin{align}\label{eqn: restricted_metric_connection}
\tilde{g}|_{p^{-1}(o)}\ =\
 \tilde{g}_{x_0}|_{\tilde{f}^{-1}(o)}\quad \text{and}\quad
 \tilde{\nabla}|_{p^{-1}(o)}\ =\ \tilde{\nabla}^{\ast}_{o}.
\end{align}

The regularity of $\tilde{g}$ is determined not
only by the regularity (\ref{eqn: tilde_g_regularity}) of $\tilde{g}_{x_0}$,
but also by the regularity of the local section (\ref{eqn: s_regularity});
similarly, since $\tilde{\nabla}$ is nothing but $\tilde{\nabla}^{\ast}_{o}$
along each $p$ fiber, the estimate (\ref{eqn: tilde_nabla_regularity}) carries
over for $\tilde{\nabla}^{\ast}$.
Therefore, we have for any $l\in \mathbb{N}$,
\begin{align}
\sup_{W_{x_0}} |\nabla^l\Rm_{\tilde{g}}|_{\tilde{g}}\ &\le\
C_{3.13}(l),
\label{eqn: gpullback_regularity}\\
\text{and}\quad \sup_{W_{x_0}} |\nabla^l(\tilde{\nabla}
-\phi_{x_0}^{\ast}\nabla^{LC})|_{\tilde{g}}\ &\le\ C_{3.14}(l)\delta
r_{x_0}^{-2-l},
\label{eqn: nabla_pullback_regularity}
\end{align}
where $\tilde{\phi}_{x_0}^{\ast}\nabla^{LC}$ is the pull-back of the Levi-Civita
connection of $\tilde{\nabla}^{LC}$, which is the same as the Levi-Civita
connection of $\tilde{g}$. Especially, the
restricted metric $\tilde{g}|_{p^{-1}(o)}$ and connection
$\tilde{\nabla}|_{p^{-1}(o)}$ on the central $p$ fiber enjoy the same
estimates as before, given the uniform bound on
$|II_{\hat{f}^{-1}(o)}|_{\hat{g}_{x_0}}$.

With this understanding, we could now restore the index $i$ and take limits.
Recall that associated with the singular fibration
$\hat{f}_i:\hat{V}_{x_0,i})\to V_{x_0}$ there are controlled fibrations
of the unwrapped neighborhoods $\tilde{f}_i:\tilde{V}_{x_0,i}\to V_{x_0}$,
together with the metrics $\tilde{g}_{x_0,i}$ and connections
$\tilde{\nabla}^{\ast}_{i}$ on $\tilde{V}_{x_0,i}$. These structure, as
discussed above, can be tranplanted to $W_{x_0}$ for all $i$ sufficiently
large, with the help of the local trivialization $\phi_{x_0,i}$, and we get a
family of Riemannian metrics $\{\tilde{g}_i\}$ and connections
$\{\tilde{\nabla}_i\}$ on $W_{x_0}$, with the uniform (independent of $i$)
regularity control as (\ref{eqn: gpullback_regularity}) and (\ref{eqn:
nabla_pullback_regularity}). As it is easy to see that $\{\tilde{g}_i\}$ has
uniform injectivity radius lower bound (depending on $x_0$ but independent of
$i$, see also \cite[Lemma 2.5]{NaberTian4d}), the uniform regularity ensures
that we can extract sub-sequences of $\{\tilde{g}_i\}$ and
$\{\tilde{\nabla}_i\}$ that converge, in the $C^{\infty}_{loc}(W_{x_0})$
topology, to a limit Riemannian metric $\tilde{g}_{\infty}$ and a limit
connection $\tilde{\nabla}_{\infty}$.

Moreover, for the sequence of simply connected nilpotent Lie groups
$\{N_i:=(N_{o,i}, \tilde{\nabla}^{\ast}_{o,i}, \tilde{s}_i(o))\}$
equipped with the Riemannian metrics $\{\tilde{g}_{x_0,i}|_{N_{o}}\}$, by
(\ref{eqn: restricted_metric_connection}) we know that
$N_i=(\mathbb{R}^{m-n},\tilde{\nabla}_i|_{p^{-1}(o)},\vec{o})$, equipped with
the metric $\tilde{g}_i|_{p^{-1}(o)}$. Then the uniform regularity of the
metrics and connections ensures that as $i\to \infty$, possibly passing to a
sub-sequence, the pointed Cheeger-Gromov convergence 
\begin{align*}
N_i\ \xrightarrow{pCG}\
N_{\infty}:=(\mathbb{R}^{m-n},\tilde{\nabla}_{\infty}|_{p^{-1}(o)},\vec{o}),
\end{align*}
where the pointed Cheeger-Gromov map is the identity map. Moreover, on
$N_{\infty}$ the limit Riemannian metric is $\tilde{g}_{\infty}|_{p^{-1}(o)}$.
Just as before taking limit, $N_{\infty}$ acts on all $p$ fibers by \emph{left}
translations --- in fact, $\forall \hat{x}\in V_{x_0}$, since all
$\tilde{s}_i(\hat{x})$ are identified as $\vec{o}\in \mathbb{R}^{m-n}$, we have
$(p^{-1}(\hat{x}), \tilde{\nabla}_{\infty}|_{p^{-1}(o)},\vec{o})$ isomorphic to
$N_{\infty}$ as Lie groups, with isomorphisms provided by
$\psi_{\hat{x},\infty}$, the limit of $\{\psi_{\hat{x},i}\}$, see (\ref{eqn:
initial_psi}).

Since $G_{x_0}$ is finite and acts by isometries with respect to the 
pull-back metrics $\{\phi_{x_0,i}^{\ast}\tilde{g}_{x_0,i}\}$, the limit metric
$\tilde{g}_{\infty}$ remains invariant under the $G_{x_0}$ action; the
same reasoning ensures that the $G_{x_0}$ equivariance of the connections
$\{\phi_{x_0,i}^{\ast}\tilde{\nabla}^{\ast}_i\}$ passes to the equivariance of
the limit connection $\tilde{\nabla}_{\infty}$.

Recall that before taking limits, the fundamental group $\Gamma_i$ acts on each
$p$ fiber, and each $L_i=N_i\cap \Gamma_i$ is a cocompact lattice sub-group
acting on the $p$ fibers by \emph{left} translations.  Since the covering
metric $\tilde{g}_{x_0}$ is invariant under the left translation by $L_i$
(therefore $L_i$ being uniformly locally bounded and $1$-Lipschitz), by
Arzela-Ascoli's theorem we have $L_i$ converges in the $C^{\alpha}$ topology to
some group $G_{\infty}$ acting fiber-wise on $(W_{x_0},\tilde{g}_{\infty})$ by
isometries. On the one hand, by \cite[Lemma 3.1]{Fukaya88} we know that
$G_{\infty}$ is actually a Lie group. Since $L_i$ are close sub-groups of
$N_i$, the limit Lie group $G_{\infty}$ becomes a closed Lie sub-group of
$N_{\infty}$. But on the other hand, since $L_i\cap N_i$ is $\Psi(\delta_i)$
dense in $N_i$, which is equipped with the metric $\tilde{g}_i|_{p^{-1}(o)}$
(see (\ref{eqn: restricted_metric_connection})), we know that as $i\to \infty$,
\begin{align*}
\left(L_i,\tilde{g}_{i}|_{p^{-1}(o)}, \vec{o}\right)\ \xrightarrow{pGH}\ 
\left(N_{\infty}, \tilde{g}_{\infty}, \vec{o}\right).
\end{align*} 
Therefore, $G_{\infty} =N_{\infty}$, and consequently we know that
$\tilde{g}_{\infty}$ is invariant under the \emph{left} translations by
$N_{\infty}$, and consequently, the \emph{right} invariant vector fields along
the $N_{\infty}$ fibers are Killing vector fields.

\subsubsection{The limit central distribution and its density}
Let us now consider a Riemannian foliation structure of the fibration
$p:(W_{x_0},\tilde{g}_{\infty})\to (V_{x_0},\hat{g}_X)$. Among all Killing
vector fields tangent to the $p$ fibers, there is a communiting family
$\mathfrak{C}$, called the \emph{limit central distribution}, essentially
consisting of the limit of the Lie algebra of the center sub-groups.
More specifically, recall that each element in $C(N_i)\triangleleft N_i$ is
characterized by the vanishing of the commutator, therefore the smooth
convergence of the pull-back connections ensure that $C(N_i)$ accumulates to be
a closed sub-group $Z$ of $C(N_{\infty})$, whence a Lie sub-group. The limit
central distribution is then the collection $\mathfrak{C}
=\{X_1,\ldots,X_{k_0}\}$ of \emph{right} invariant vector fields along the
fibers, determined by the Lie algebra of $Z$ as following: let
$X_1(o,\vec{o}),\ldots,X_{k_0}(o,\vec{o})$ be an orthonormal basis of the Lie
algebra of $Z$, then they determine a collection of \emph{right} invariant vector
fields along the central fiber $p^{-1}(o)$; for any other $\hat{x}\in
V_{x_0}$, consider the collection of tangent vectors
$D\psi_{\hat{x},\infty}\cdot X_1(o,\vec{o}), \ldots, D\psi_{\hat{x},\infty}\cdot
X_{k_0}(o,\vec{o})$ and let them generate \emph{right} invariant vector fields
along the fiber $p^{-1}(\hat{x})$. Since $\psi_{\hat{x},\infty}$ varies
smoothly with respect to $\hat{x}\in V_{x_0}$, we obtain $X_1,\ldots,X_{k_0}$
as a family of smooth \emph{right} invariant vector fields along the $p$
fibers. Since $Z\le C(N_{\infty})$, we see that vector fields in $\mathfrak{C}$
are also \emph{left} invariant and that $\mathfrak{C}$ is a commutative,
therefore being an integrable distribution. Moreover, since each vector field
in $\mathfrak{C}$ is Killing, it actually provides a regular Riemannian
foliation of $(W_{x_0},\tilde{g}_{\infty})$. Along each leaf $L$ by integrating
$\mathfrak{C}$, we can consider the restricted metric
$H=\tilde{g}_{\infty}|_L$, and the \emph{density} of the limit central
distribution is defined to be
\begin{align*}
\det\ [H(X_a,X_b)]_{k_0\times k_0}\ :=\ \det\ 
[\tilde{g}_{\infty}(X_a,X_b)]_{k_0\times k_0}\ =\ |X_1\wedge \cdots \wedge
X_{k_0}|_{\tilde{g}_{\infty}}^2.
\end{align*}
Clearly, this density is independent of the choice of the vector fields
$X_1,\ldots,X_{k_0}$, and is constant along the leaves of $\mathfrak{C}$.
Moreover, we notice that since $(N_{\infty})_L$ acts by isometries, $|X_1\wedge
\cdots \wedge X_{k_0}|_{\tilde{g}_{\infty}}$ is not just constant on each leaf,
but also constant along the entire fiber $p^{-1}(\hat{x})$, for each $\hat{x}\in
V_{x_0}$. As a consequence, $|X_1\wedge \cdots \wedge
X_{k_0}|_{\tilde{g}_{\infty}}$ descends to a smooth function on $V_{x_0}$, which
is equal to $1$ at $x_0$.

Finally, we notice that the restricted metrics $(\tilde{\phi}_{x_0,i}^{\ast}
\tilde{g}_{x_0,i})|_{C(N_i)} =\tilde{g}_{x_0,i}|_{C(N_i)}$, smoothly converges
to the metric $\tilde{g}_{\infty}|_{Z}$, therefore $H$ can be
directly realized as the limit of the restricted metrics to the center: 
\begin{align}\label{eqn: H_convergence}
\left(C(N_i),\tilde{g}_{x_0,i}|_{C(N_i)}\right)\ \xrightarrow{pCG}\ (Z,H(x_0))
\end{align}
as $i\to \infty$, by the uniform regularity control of $\tilde{g}_{x_0,i}$ and
the uniform bound of the second fundamental form $|II_{p^{-1}(x_0)}|$ of the
central fiber with respect to the metrics $\tilde{g}_{x_0,i}$.

\section{Collapsing and convergence of integrals}
In this section we prove Theorem \ref{thm: integral}, which concerns about the
convergence of integrals over a sequence of collapsing Riemannian manifolds, and
is needed later to extract elliptic equations on the limit unwrapped
neighborhoods via the asymptotic behavior of certain functionals along the
Ricci flow. We will begin with proving Proposition \ref{prop:
measure_collapsing}, a generalization of the measured Gromov-Hausdorff
convergence theorem \cite[Theorem 0.6]{Fukaya87} to a wider family of measure
densities, then prove the convergence of integrals against suitable probability
measures in Proposition \ref{prop: collapsing_integral}, and finally we prove
Proposition\ref{prop: local_convergence}, which gives a nice representation of
the limit function on a limit unwrapped neighborhood around an orbifold point.
The proof of Theorem \ref{thm: integral} is then a direct conseqeunce of these
propositions. 

Recall that in \cite{Fukaya87}, a family of collapsing Riemannian manifolds
$\{(M_i,g_i)\}$ is naturally associated to a sequence of metric measure spaces
$\{(M_i,g_i,|M_i|_{g_i}^{-1}\dvol_{g_i})\}$, which converges to a limit metric
measure space $(X,d,\mu_X)$ in the measured Gromov-Hausdorff topology.
Especially, $\mu_X$ is absolutely continuous with respect to the limit
Riemannian metric on the regular part of $X$, with a density function
$\chi_X\ge 0$. Noticing that $|M_i|_{g_i}^{-1}\dvol_{g_i}$ is nothing but a
sufficiently regular probability measure on each $(M_i,g_i)$, we can actually
consider a more generalized setting and prove the following
\begin{proposition}\label{prop: measure_collapsing}
Assume $\{(M_i,g_i)\}$ collapses to $(X,d_X)$ with bounded curvature and
diameter, and let  $f_i:M_i\to X$ denote the singular fibration maps discussed
in \S 2.1.2.
Suppose there are $C^1$ functions $\rho_i>0$ on $M_i$ satisfying 
 \begin{align}\label{eqn: rho_condition}
\sup_{M_i} |\ln (\rho_i|M_i|_{g_i})|\ \le\ C_{4.1}\quad \text{and}\quad
 \sup_{M_i}|\nabla \rho_i||M_i|_{g_i}\ \le\ C_{4.1}',
 \end{align}
 for some uniform constants $C_{4.1}, C_{4.1}'>0$, then there is a continuous 
 density $\rho_X:X\to [e^{-C_{4.1}},e^{C_{4.1}}]$
 , such that for any open subset $U\subset X$,
\begin{align*}
\lim_{i\to \infty}\int_{f_i^{-1}(U)}\rho_i\ \dvol_{g_i}\ =\ \int_U\rho_X\ 
\dmu_{X},
\end{align*}
whre $\mu_{X}$ is the limit measure defined in (\ref{eqn: chi_X_regular}) and
(\ref{eqn: chi_X_integral}) (see also \cite[Theorem 0.6]{Fukaya87}), as a
consequence of the measured Gromov-Hausdorff convergence.
\end{proposition}

\begin{proof}
We recall the notations $\delta_i:=d_{GH}(M_i,X)$, and that $g_X$ denoting the
Riemannian metric on the regular part $\mathcal{R}\subset X$. If $X$ has no
singularity and (therefore) $f_i$ are genuine almost Riemannian sub-mersions, 
the proof follows from the co-area formula. Since $f_i$ is an almost Riemannian
submersion, \cite[(0-1-3)]{Fukaya87ld} implies that
$\sup_{M_i} ||J_n(f_i)|_{g_i}- 1|\le \Psi(\delta_i)$, where $|J_n(f_i)|$ is
the Jacobian of $f_i$. For any open $U\subset X$, we have, with $\chi_i(x)
:=\frac{|f_i^{-1}(x)|_{g_i}}{|M_i|_{g_i}}$, the following estimate:
\begin{align*}
\int_{f_i^{-1}(U)}\rho_i\ \dvol_{g_i}\ =\ 
&\int_{f_i^{-1}(U)}\rho_i|J_n(f_i)|_{g_i}\
\dvol_{g_i}\pm \Psi(\delta_i)\\
=\ &\int_U\left(\fint_{f_i^{-1}(x)}\rho_i|M_i|_{g_i}\
\darea_i(x)\right)\ \chi_i(x)\
\dvol_{g_X}(x)\pm \Psi(\delta_i).
\end{align*}

Denoting $\rho_{X,i}(x):=\fint_{f_i^{-1}(x)}\rho_i|M_i|_{g_i}\
\darea_i(x)$, we notice that the family $\{\rho_{X,i}\}$ are equicontinuous on
$X$: by \cite[\S 3 and \S 4]{Fukaya87ld} and \cite[(3.5)]{Fukaya87}, we know
that there is a constant $\rho>0$ which only depends on $X$ (and espeically, is
independent of $i$), and a positive function $\Psi_X$ determined only by $X$,
such that whenever $d_X(x,y)<\rho$, we have 
\begin{align*}
|\rho_{X,i}(x)-\rho_{X,i}(y)|\ \le\
&\frac{1}{|f_i^{-1}(x)|}\left|\int_{f_i^{-1}(x)} \rho_i|M_i|_{g_i}\ \darea_i(x)
-\int_{f_i^{-1}(y)} \rho_i|M_i|_{g_i}\ \darea_i(y) \right|\\
&+\left|1-\frac{|f_i^{-1}(y)|}{|f_i^{-1}(x)|}\right|
\fint_{f_i^{-1}(y)}\rho_i|M_i|_{g_i}\ \darea_i(y);
\end{align*}
now picking any $x_i\in f^{-1}_i(x)$ and $y_i\in f_i^{-1}(y)$, we know that
$d(x_i,y_i)\le \rho+3\delta_i$, and can proceed as 
\begin{align}\label{eqn: rho_equicontinuity}
\begin{split}
 |\rho_{X,i}(x)-\rho_{X,i}(y)|\ \le\ &\left|\rho_i(x_0)
-\rho_i(y_0)\frac{|f_i^{-1}(y)|}{|f_i^{-1}(x)|}\right||M_i|_{g_i}
+(6C_{4.1}'\delta_i+e^{C_{4.1}})\left|1-\frac{|f_i^{-1}(y)|}{|f_i^{-1}(x)|}\right|\\
\le\ &|\rho_i(x_0)-\rho_i(y_0)||M_i|_{g_i}
+(6C_{4.1}'\delta_i+2e^{C_{4.1}})\left|1-\frac{|f_i^{-1}(y)|}{|f_i^{-1}(x)|}\right|\\
\le\ &6C_{4.1}'(d_X(x,y)+\delta_i)+(6C_{4.1}'\delta_i+2e^{C_{4.1}})
O\left(\Psi_X(d_X(x,y))+\delta_i\right),
\end{split}
\end{align}
where $O\left(\Psi_X(d_X(x,y))+\delta_i\right)=\sum_{n=1}^{\infty}
\frac{1}{n!}\left(\Psi_X(d_X(x,y))+\delta_i\right)^n$, whence the desired
equicontinuity. By the uniform boundedness assumption (\ref{eqn: rho_condition})
and the equicontinuity, Arzela-Ascoli's Theorem guarantees the existence of some
$\rho_X$ such that (after possibly passing to a sub-sequence) $\rho_{X,i}\to
\rho_X$ uniformly on $X$, as $i\to \infty$. Noice also that throughout $X$,
$\rho_X$ only takes values in $[e^{-C_{4.1}},e^{C_{4.1}}]$. Now we could
estimate for all $i$ sufficiently large, that 
\begin{align*}
\left|\int_{f_i^{-1}(U)}\rho_i\ \dvol_{g_i}-\int_U\rho_X\ \dmu_X\right|\
 \le\ &\left|\int_U\rho_X\ \chi_i\ \dvol_{g_X} -\int_U\rho_X\
 \dmu_X\right|+2\mu_X(U)\sup_X |\rho_{X,i}-\rho_X| 
 +\Psi(\delta_i),
\end{align*}
and by \cite[Lemma 3.8]{Fukaya87} we get the desired convergence.

In the general case, the usual frame bundle argument would finish the proof.
More specifically, we first notice that for the canonical Riemannian submersions
$\pi_i:(FM_i,\bar{g}_i)\to (M_i,g_i)$, we have $\pi_i^{\ast}\rho_i$
being constant along all $O(m)$ fibers. Here recall that each $O(m)$ fiber is
equipped with its canonical metric. Therefore, by condition (\ref{eqn:
rho_condition}), we have the following estimates on $FM_i$:
\begin{align*}
e^{-C_{4.1}}|O(m)|\ \le\ (\pi_i^{\ast}\rho_i)|FM_i|_{\bar{g}_i}\ \le\
e^{C_{4.1}}|O(m)|\quad\text{and}\quad \sup_{FM_i}|\nabla
\pi_i^{\ast}\rho_i||FM_i|_{\bar{g}_i}\ \le\ C_{4.1}'|O(m)|.
 \end{align*} 
  There is also an $(n+\frac{m(m-1)}{2})$-dimensional Riemannian manifold
  $(Y,g_Y)$, on which $O(m)$ acts by isometries, with quotient isometric to
  $(X,d_X)$.
  Moreover, there are almost Riemannian submersions $\bar{f}_i:FM_i\to Y$,
  such that $\pi_Y\circ \bar{f}_i=F_i\circ \pi_i$, with $\pi_Y: Y\to X$ denoting
  the quotient map of the $O(m)$ actions; recall the discussions in \S 2.2.

The key point here is that for any open $U\subset X$,
$f_i^{-1}(U)=\pi_i(\bar{f}_i^{-1}(\pi_Y^{-1}(U)))$. Therefore, applying the
previous case to $\pi_Y^{-1}(U)\subset Y$, we see
\begin{align*}
\lim_{i\to \infty}\int_{\bar{f}_i^{-1}(\pi_Y^{-1}(U))}\pi_i^{\ast}\rho_i\
\dvol_{\bar{g}_i}\ =\ \int_{\pi_Y^{-1}(U)} \rho_Y\ \dmu_Y.
\end{align*}
Notice that $\rho_Y$ is constant on each $O(m)$ orbit in $Y$, and
therefore taking $O(m)$ quotients on both sides of the equation gives the
desired equality for $\rho_X$ on $U$. In fact, we define $\rho_X$ such that 
$\rho_Y|O(m)|^{-1}=\pi_Y^{\ast}\rho_X$ --- in this way, we have on the one
hand for all $i$ sufficiently large,
\begin{align*}
\int_{\bar{f}_i^{-1}(\pi_Y^{-1}(U))}\pi_i^{\ast}\rho_i\ \dvol_{\bar{g}_i}\ &=\
|O(m)|\int_{f_i^{-1}(U)}\rho_i\ \dvol_{g_i},
\end{align*}
 and on the other hand, by (\ref{eqn: chi_X_integral}) we have 
\begin{align*}
\int_{\pi_Y^{-1}(U)}\rho_Y\ \dmu_Y\ =\ &\int_U\left(\int_{\pi_Y^{-1}(x)}\rho_Y\
\darea_{\pi_Y^{-1}(x)}\right)\ \dvol_{g_X}\\
=\ &\int_U\rho_X|O(m)|\ \chi_X\ \dvol_{g_X}\\
=\ &|O(m)|\int_U\rho_X\ \dmu_X.
\end{align*} 
This implies the desired convergence, as well as the desired value bounds for
$\rho_X$.
\end{proof}

In our later applications to Ricci flows as discussed in the last sub-section,
we would let $\rho_i=u(t_i)$, the conjugate heat density solving
$\square^{\ast}u=0$ at various time instances $t_i\to\infty$. Notice that the
total heat is always a constant, i.e. $\int_Mu(t_i)\ \dvol_{g(t_i)}=1$ for any
$i$.

It is therefore natural to think of $\rho_i\dvol_{g_i}$ in
Proposition~\ref{prop: measure_collapsing} as a sequence of measure densities
with uniformly bounded and positive total mass and certain regularity assumptions,
and consider the collapsing and convergence about integrating a family of
functions against such measures:
\begin{prop}\label{prop: collapsing_integral}
Let $(M_i,g_i,\rho_i)\xrightarrow{mGH} (X,d_X,\rho_X)$ be the data described
in Proposition~\ref{prop: measure_collapsing}, with the singular fibration 
structures $f_i:M_i\to X$ as described in \S 2.1.2. We further assume that each
$g_i$ satisfies (\ref{eqn: regularity}). Suppose there are $w_i\in
C^{\infty}(M_i)$ satisfying the uniform $C^1$ control
\begin{align}\label{eqn: w_i_C2}
\sup_{M_i}\left(|w_i|+|\nabla_{g_i}w_i|\right)\ \le\ C_{4.2},
\end{align}
for some constant $C_{4.2}>0$. Then there is a subsequence, still denoted by
$\{w_i\}$, and a continuous function $w_X$ on $X$ such that for any open
set $U\subset X$,
 \begin{align}\label{eqn: integral_limit}
  \lim_{i\to \infty}\int_{f_i^{-1}(U)} w_i\ \rho_i\ \dvol_{g_i}\ =\ \int_Uw_X\
  \rho_X\ \dmu_X.
  \end{align}
\end{prop}

\begin{proof}
We appeal to the frame bundle argument again. Let $\pi_i:FM_i\to M_i$ be the
frame bundle over $M_i$, as discussed in the proof of Proposition~\ref{prop:
measure_collapsing}. Notice that $\pi_i^{\ast}w_i$ are constant along the
$O(m)$ fibers, and therefore we have the estimates
\begin{align}\label{eqn: pi_hi_C2}
\sup_{FM_i}\left(|\pi_i^{\ast}w_i|
+|\nabla_{\bar{g}_i}(\pi_i^{\ast}w_i)|\right)\ \le\ \bar{C}_{4.2}(C_{4.2},m).
\end{align}
Also recall that $(FM_i,\bar{g}_i)\xrightarrow{eGH}(Y,g_Y)$ with $O(m)$
equivariant fibraions $\bar{f}_i:FM_i\to Y$, which are also
$\Psi(\delta_i)$ almost Riemannian submersions --- here remember that 
$\delta_i=d_{GH}(M_i,X)$.


Now by (\ref{eqn: pi_hi_C2}) and the same argument leading to (\ref{eqn:
rho_equicontinuity}), we know that the functions $\bar{w}_i$ defined on $Y$ as
\begin{align*}
\bar{w}_i(y)\ :=\ \fint_{\bar{f}_i^{-1}(y)}\pi_i^{\ast}w_i\
\darea_{\bar{f}_i^{-1}(y)}
\end{align*}
are uniformly bounded and equicontinuous. Therefore there is a limit continuous
function $w_Y$ to which a sub-sequence, still denoted by $\{\bar{w}_i\}$,
converges uniformly. 

Notice that $w_Y$ is constant along the $O(m)$ orbits in $Y$, since so are all
$\pi_i^{\ast}w_i$ and consequently all $\bar{w}_i$. Therefore, $w_Y$ naturally
descends to a continuous function $w_X$ defined on $X$.

Moreover, it is easy to see, by the co-area formula, that for any $U\subset X$
open, with $\pi_Y^{-1}(U)\subset  Y$, we have
\begin{align*} 
\lim_{i\to \infty}\int_{\bar{f}_i^{-1}(\pi_Y^{-1}(U))}
(\pi_i^{\ast}w_i)(\pi_i^{\ast}\rho_i)\ \dvol_{\bar{g}_i}\ &=\
\lim_{i\to\infty}\int_{\pi_Y^{-1}(U)} \left(\int_{\bar{f}_i^{-1}(y)}
(\pi_i^{\ast}w_i)(\pi_i^{\ast}\rho_i)\ \darea_{\bar{f}_i^{-1}(y)}\right)\
\dvol_{g_Y}(y)\\
&=\ \lim_{i\to \infty}\int_{\pi_Y^{-1}(U)} \bar{w}_i(y)\ \rho_{Y,i}(y)\
\dmu_Y(y)\\
&=\ \int_{\pi_Y^{-1}(U)}w_Y\ \rho_Y\ \dmu_Y,
\end{align*}
where for the second equality we need the estimate
\begin{align}\label{eqn: bar_w_i_fiber_value}
\sup_{z\in FM_i}\left|(\pi_i^{\ast}w_i)(z)-\bar{w}_i(\bar{f}_i(z))\right|\ \le\
3\bar{C}_{4.2}\Psi(\delta_i),
\end{align}
which easily follows from (\ref{eqn: pi_hi_C2}) and the definition of
$\bar{w}_i$, as
\begin{align*}
|\pi_i^{\ast}w_i(z)-\pi_i^{\ast}w_i(z')|\ \le\ \bar{C}_{4.2}d_{\bar{g}_i}(z,z')
\end{align*}
for any $z,z' \in \bar{f}_i^{-1}(y)$ and the \emph{extrinsic}
diameter of $\bar{f}_i^{-1}(y)$ is bounded above by $\Psi(\delta_i)$ for any
$y\in Y$. Now by the $O(m)$ invariance in taking the previous limit, we see that
\begin{align*}
\lim_{i\to \infty}\int_{f_i^{-1}(U)}w_i\ \rho_i\ \dvol_{g_i}\ =\
\int_Uw_X\ \rho_X\ \dmu_X,
\end{align*}
which is the desired inetegral equality. Especially, $\pi_Y^{\ast}w_X=w_Y$.
\end{proof}


We now finish the proof of Theorem \ref{thm: integral} by establishing the
following representation of the limit function around an orbifold point:
\begin{prop}\label{prop: local_convergence}
Besides the assumptions in Proposition \ref{prop: collapsing_integral}, assume
that $\sup_{M_i}\|w_i\|_{C^k}\le C_{4.3}$ for some $k\ge 2$, and fix any $x_0\in
\tilde{\mathcal{R}}$. Let $U_{x_0}\subset \tilde{\mathcal{R}}$, $V_{x_0}$, and 
$\{(\tilde{V}_{x_0,i}, \tilde{g}_{x_0,i})\}$ denote the data obtained
from Proposition \ref{prop: unwrapping}, and let $(W_{x_0},\tilde{g}_{\infty})$
denote the limit unwrapped neighborhood, with the limit Riemannian submersion 
$p:W_{x_0}\to V_{x_0}$, as obtained in Theorem \ref{thm:
canonical_nbhd}.  Denoting $\tilde{w}_i:= \tilde{q}_i^{\ast} \hat{q}_i^{\ast}
w_i|_{f_i^{-1}(U_{x_0})}$ as the pull-back of $w_i$ to $\tilde{V}_{x_0, i}$ by
the covering maps $\hat{q}_i$ and $\tilde{q}_i$, see diagram (\ref{eqn:
diag_unwrapping}), then there is a function $\tilde{w}_{\infty}\in
C^{k-1,\alpha}(W_{x_0})$, to which $\tilde{w}_i$ converges in the
$C^{k-1,\alpha}$ topology for any $\alpha\in (0,\alpha')$ and any $\alpha'<1$.
Morevoer, $\tilde{w}_{\infty}$ is constant along the $p$ fibers, so
that $p^{\ast}q_{x_0}^{\ast}\left(w_X|_{U_{x_0}}\right)=\tilde{w}_{\infty}$.
\end{prop} 
\begin{proof} 
Since $x_0$ is an orbifold point, the singular fibration structure
$f_i:f_i^{-1}(U_{x_0})\to U_{x_0}$ is relatively simple --- especially
(\ref{eqn: chi_X_integral}) holds around on $U_{x_0}$, and we could express
$w_X$ as the asymptotic average of values of $w_i$ on the $f_i$ fibers directly,
rather than ivoking the frame bundle strcture: by (\ref{eqn: w_i_C2}) and the
smallness of the fibers of the singular fibration $f_i:f_i^{-1}(U_{x_0})\to
U_{x_0}$, we know that
\begin{align}\label{eqn: zz'}
\forall x\in U_{x_0},\  \forall z,z'\in f_i^{-1}(x),\quad |w_i(z)-w_i(z')|\
\le\ 3C_{4.2}\delta_i,
\end{align}
and the limit $\lim_{i\to\infty} \fint_{f_i^{-1}(x)}w_i\
\darea_{f_i^{-1}(x)}$ exists for every $x\in U_{x_0}$ with uniform convergence;
on the other hand, recalling the definition of $\rho_{X,i}$ in the proof of
Proposition \ref{prop: measure_collapsing}, by (\ref{eqn: chi_X_integral}) and
(\ref{eqn: integral_limit}), we have for any $B(x,r)\subset U_{x_0}$,
\begin{align*}
\fint_{B(x,r)} w_X\ \rho_X\ \dmu_X\ =\ &\lim_{i\to \infty}
\fint_{B(x,r)}\left(\fint_{f_i^{-1}(x)}w_i\ \darea_{f_i^{-1}(x)}\right)\
\rho_{X,i}\ |M_i|_{g_i}^{-1} \dvol_{g_i}\\
=\ &\fint_{B(x,r)} \left(\lim_{i\to\infty} \fint_{f_i^{-1}(x)}w_i\
\darea_{f_i^{-1}(x)}\right)\ \rho_X\ \dmu_X;
\end{align*} 
by the continuity of $w_X$ and $\rho_X\dmu_X$, we see that 
\begin{align}\label{eqn: w_X_limit}
w_X(x)\ =\ \lim_{i\to\infty}\fint_{f_i^{-1}(x)}w_i\ \darea_{f_i^{-1}(x)}.
\end{align}

On the other hand, by the proof of the Key Lemma~\cite[Lemma 4.3]{Fukaya87}, we
could show that $w_i$ are alomst constant along the fiber $f_i^{-1}(x)\subset
M_i$. To see this, fix any $z\in f_i^{-1}(x)$, and any $v\in T_zf_i^{-1}(x)$ of
unit length, and we could consider the curve $c(t):=\exp_z tv$. By \cite[Lemma
4-8]{Fukaya87ld}, we know that there is a uniform (independent of $\delta_i$)
$t_0>0$, such that
 \begin{align}
 \sup_{t\in [0,t_0]}d_{g_i}(c(t), z)\ \le\ 10\delta_i.
 \end{align}
Therefore, (\ref{eqn: pi_hi_C2}) implies the following estimate:
 \begin{align*}
 \sup_{t\in [0,t_0]}|w_i(c(t))-w_i(z)|+\left|\langle
 \nabla_{g_i}w_i,\dot{c}(t)\rangle- \langle \nabla_{g_i}w_i(z),v\rangle\right|\ \le\
 10C_{4.3}'(C{4.2},m)\delta_i.
 \end{align*}
 This estimate implies that
 \begin{align*}
 \left|t_0\langle \nabla_{g_i}w_i(z),v\rangle\right|\ \le\ 
  &\left|\int_0^{t_0}\langle \nabla_{g_i}w_i,\dot{c}(t)\rangle\
  \text{d}t\right|+\left|\int_0^{t_0}\int_0^t\frac{\text{d}}{\text{d}s}\langle
 \nabla_{g_i}w_i,\dot{c}(s)\rangle\ \text{d}s\
 \text{d}t\right|\\
 \le\ &\left|w_i(c(t_0))-w_i(z)\right|+10C_{4.3}\delta_it_0\\
 \le\ &10(1+t_0)C_{4.3}'\delta_i,
 \end{align*}
 and consequently
 \begin{align}\label{eqn: hi_constant_fiber}
 \left|\langle \nabla_{g_i}w_i(z),v\rangle\right|\ \le\
 10(1+t_0^{-1})C_{4.3}'\delta_i.
 \end{align}
 
 We now pull every thing back to $\tilde{V}_{x_0,i}$. Since the magnitude of the
 tangential derivatives of $\tilde{w}_i$ along the $\tilde{f}_i$ fibers are
 measured with respect to the pull-back mertrics $\tilde{g}_{x_0,i}$,
 (\ref{eqn: hi_constant_fiber}) still holds trivially:
 \begin{align}\label{eqn: wi_constant_fiber}
 \sup_{\tilde{V}_{x_0,i}}|\tilde{\nabla}^T\tilde{w}_i|_{\tilde{g}_{x_0,i}}\ \le\
 \Psi(\delta_i),
 \end{align}
 where $\tilde{\nabla}^T$ denotes $\tilde{\nabla}^{LC} \tilde{w}_i$ orthogonally
 projected to the direction $T_z\tilde{f}_i^{-1}(\hat{x})$ for some $\hat{x}\in
 V_{x_0}$ and $z\in \tilde{f}_i^{-1}(\hat{x})$.
 
 Moreover, with respect to the pull-back metrics $\tilde{g}_{x_0,i}$, the
 estimates $\sup_{M_i}\|w_i\|_{C^k}\le C_{4.3}$ implies that
 $\sup_{\tilde{V}_{x_0,i}}\|\tilde{w}_i\|_{C^k}\le C_{4.3}$, and we can take a
  subsequence of $\{\tilde{w}_i\}$ converging to certain $\tilde{w}_{\infty}\in
 C^{k-1,\alpha}(W_{x_0})$, along with the convergence $(\tilde{V}_{x_0,i},
 \tilde{g}_{x_0,i}) \xrightarrow{pCG} (W_{x_0},\tilde{g}_{\infty})$ as $i\to
 \infty$. Moreover, the constancy of $\tilde{w}_{\infty}$ along each $\tilde{f}$
 fiber is guaranteed by (\ref{eqn: wi_constant_fiber}).
 
 We now show that $\tilde{w}_{\infty}=\tilde{f}^{\ast}w_X$.
 To see this, we first notice that the fiber $(\tilde{f}_i^{-1}(\hat{x}),
 \tilde{g}_{x_0,i}|_{\tilde{f}_i^{-1}(\hat{x})})$ has its intrinsic distance
 comparable to the extrinsic distance, uniformly on compact subsets of
 $\tilde{V}_{x_0}$, by a factor depending on the regularity control (\ref{eqn:
 hat_g_regularity}) on $g_i$ and the uniform control (\ref{eqn:
 hat_f_regularity}) of the second fundamental form of the $\tilde{f}_i$ fibers.
 Consequently, for any $x\in U_{x_0}$ and any $z\in \tilde{f}_i^{-1}
 (q_{x_0}^{-1}(x))$, we could estimate by (\ref{eqn: wi_constant_fiber}) with
 any fixed $z_0\in \tilde{f}_i^{-1} (q_{x_0}^{-1}(x))$ that
 \begin{align*}
 \left|\tilde{w}_i(z)-\fint_{f_i^{-1}(x)}w_i\ \darea_{f_i^{-1}(x)}\right|\ \le\
 &\ |\tilde{w}_i(z)-\tilde{w}_i(z_0)|
 +\left|\tilde{w}_i(z_0)-\fint_{f_i^{-1}(x)}w_i\ \darea_{f_i^{-1}(x)}\right|\\
\le\ &d_{\tilde{g}_{x_0,i}}(z,z_0)\Psi(\delta_i|r_{x_0}^{-1}) 
+\left|w_i(\hat{q}_i(\tilde{q}_i(z_0)))-\fint_{f_i^{-1}(x)}w_i\
\darea_{f_i^{-1}(x)}\right|,
 \end{align*}
 which asymptotically vanishes as $i\to\infty$, in view of (\ref{eqn: zz'}).
 Since $z\in \tilde{f}_i^{-1}(q_{x_0}^{-1}(x))$ is arbitrary, the proposition
 then follows from (\ref{eqn: w_X_limit}).
\end{proof}


\begin{remark}
If we only assume uniform bounds on the sectional curvature of $g_i$, then the
same conclusion should still hold for uniformly $C^2$ bounded functions, but we
will not need to push the regularity estimate to this level --- in our later
applications, the Ricci flow will provide the desired extra regularity of the
metric.
\end{remark}

\section{Density of the limit central distribution}
In this section, we will prove Theorem~\ref{thm: limit_central_density}.
Heuristically, we could think of central sub-fibrations $c_i:CM_i\to X$ of the
singular fibration $f_i:M_i\to X$, with each $c_i$ fiber being a torus orbit
lying in an $f_i$ fiber, and the limit central density function $\chi_C$ can be
defined as the asymptotic relative volume distribution of the $c_i$ fibers at
the limit; but such central sub-fibration structure cannot be constructed
globally over $X$, because to identify a single torus orbit in each $f_i$
fiber, we need to specify a base point of the fiber, but the possible occurance
of the corner singularity prevents us from consistently and smoothly choosing
such base points of the $f_i$ fibers over the entire $X$. 

In order to define the limit central density function $\chi_C$ globally over
$X$, we need to instead consider the quotient fibration and define a continuous
quotient density $\chi_Q:X\to (0,\infty)$, so that $\chi_C$ is defined as
$\chi_X\slash \chi_Q$; this construction, to be carried out in the first
sub-section, relies on the frame bundle argument in \cite[\S 3]{Fukaya87} and
the construction of the invariant metric in \cite[\S 4]{CFG92}, see also
Proposition \ref{prop: invariant_metric}.

Around an orbifold point $x\in \tilde{\mathcal{R}}$, however, $\chi_C$ admits
much better representation: in fact, locally on $U_x$, the desired central
sub-fibration $c_i:C_iU_x\to U_x$ can be constructed, and we will show in the
second sub-section that $\chi_C$ is indeed a constant multiple of the asymptotic 
relative volume distribution of the $c_i$ fibers. Consequently, in the last
sub-section we will follow the argument of \cite[Lemma 2-5]{Fukaya89} to show
that $\chi_C$ can be expressed geometrically as a constant multiple of the
volume form of the limit central distribution $\mathfrak{C}$, as discussed in
\S 3.3.2.

\subsection{Defining the limit central density function}
Consider the frame bundles $(FM_i,\bar{g}_i)$, the collapsing gives almost
Riemannian submersions $\bar{f}_i:FM_i\to Y$ with $(Y,g_Y)$ being
some lower dimensional Riemannian manifold. Recall that by \cite[Theorem
2.6]{CFG92} we can choose $\bar{f}_i$ so that they are $\Psi(\delta_i)$-almost
$O(m)$ equivariant $\Psi(\delta_i)$ Gromov-Hausdorff approximations, and that
they satisfy the uniform regularity control for any $l\in \mathbb{N}$,
\begin{align}\label{eqn: barf_regularity}
\sup_{FM_i}|\nabla^l\bar{f}_i|_{\bar{g}_i}\ \le\ C_{4.1}(l,Y),\quad
\text{and}\quad \sup_Y |II_{\bar{f}_i^{-1}(y)}|_{\bar{g}_i}\ \le\ C_{4.1}(Y). 
\end{align} 
Also recall that by the existence of a global fiber-wise flat connection, each
fiber of $\bar{f}_i$ is affine diffeomorphic to a nilmanifold $N_i\slash L_i$
with $N_i$ being some simply connected nilpotent Lie group and $L_i$ being a
cocompact lattice subgroup of $N_i$, see \cite[(6.1.10)]{CFG92}. Notice that
there is a free action by the centeral torus $\mathbb{T}_i:=C(N_i)\slash
C(L_i)$ on each $\bar{f}_i$. We further recall that by \cite[Theorem
0.6]{Fukaya87}, the collapsing sequence $\{(FM_i,\bar{g}_i)\}$ gives, after
possibly passing to a subsequence still denoted by the original one, a limit
density function $\bar{\chi}$ on $Y$, which is continuous, strictly positive
and constant along the $O(m)$ orbits, and a key property of $\chi_X$, as
illustrated by \cite[(0.7.3)]{Fukaya87}, is
\begin{align}
\chi_X^{-1}(0)\ =\ \tilde{\mathcal{S}}.
\end{align}

The limit central density function $\chi_C$ to be defined on $X$, will be
another continuous function whose zero locus also captures the entire
$\tilde{\mathcal{S}}$. This function will be constructed as the quotient of
$\chi_X$ by some continuous positive function $\chi_Q$ on $X$, which is defined
by some $O(m)$ invariant positive continuous function $\bar{\chi}_Q$ on $Y$.

To define $\bar{\chi}_Q$ on $Y$, we start with a topological consideration:
we notice that for each fiber bundle $\bar{f}_i:FM_i\to Y$, the central torus
$\mathbb{T}_i=C(N_i)\slash C(L_i)$ acts on the $N_i\slash L_i$ fibers freely by
the quotient action of the \emph{left} translations, and therefore we could
form a topological quotient $FM_i\slash \mathbb{T}_i$; since the free action is
fiber-wise, the quotient still furnishes a fiber bundle $[\bar{f}_i]:FM_i\slash
\mathbb{T}_i\to Y$ with fibers being nilmanifold affine diffeomorphic to
$(N_i\slash C(N_i))\slash (L_i\slash C(L_i))$. This bundle is still $O(m)$
equivariant: by \cite[Proposition 4.3]{CFG92}, the infinitesimal action of $N_i$
commutes with the $O(m)$ action, and as a consequence, $O(m)$ sends
$\mathbb{T}_i$ orbits to $\mathbb{T}_i$ orbits.

Notice that the Riemannian manifold $(FM_i,\bar{g}_i)$ is $\Psi(\delta_i)$ close
to $(Y,g_Y)$, with $\bar{f}_i$ being the actual Gromov-Hausdorff approximation.
Suppose that the $\mathbb{T}_i$ action is isometric, the correpsonding
quotient $FM_i\slash \mathbb{T}_i$, equipped with the quotient metric, should
be still $\Psi(\delta_i)$ close to $Y$ in the $O(m)$ equivariant
Gromov-Hausdorff sense; if we further have a uniform sectional curvature bound
of the quotient manifold $FM_i\slash \mathbb{T}_i$, we could then define the
limit density $\bar{\chi}_Q$, in the same way just as defining $\bar{\chi}$, for
the collapsing fibrations $[\bar{f}_i]: FM_i\slash \mathbb{T}_i\to Y$.

However, although the $\mathbb{T}_i$ action, being affine in each fiber, is
close to being isometric, it is \emph{not} necessarily the case. In order to
overcome this difficulty, we recall that by averaging the metrics $\bar{g}_i$
along the the $\bar{f}_i$ fibers as \cite[(4.8)]{CFG92}, there are $(N_i)_L$
invariant metrics $\bar{g}^1_i$, as stated in Proposition~\ref{prop:
invariant_metric}, satisfying
\begin{align}\label{eqn: FM_invariant_metric}
\sup_{FM_i}|\nabla^l(\bar{g}_i-\bar{g}^1_i)|\ \le\ C_{5.3}(l)\Psi(\delta_i),
\end{align}
since the collapsing limit $Y$ is a smooth manifold and each $\bar{f}_i:FM_i\to
Y$ is a smooth fiber bundle. Besides their invariance under the $\mathbb{T}_i$
actions, we also notice that the metrics $\bar{g}_i^1$ are invariant under the
$O(m)$ actions. Consequently, each $O(m)$ orbit in $FM_i$ has volume
approximately equal to $|O(m)|$ in the standard metric, as (\ref{eqn:
O(m)_approximate_volume}) shows, and by (\ref{eqn: limit_approximate_metric}) 
we also know that
\begin{align}\label{eqn: chi_Y_invariant_metric}
\forall y\in Y,\quad \chi_Y(y)\ =\ \lim_{i\to \infty}
\frac{|\bar{f}_i^{-1}(y)|_{\bar{g}_{i|}^1}}{|FM_i|_{\bar{g}_i^1}}.
\end{align}

By the $\mathbb{T}_i$ invariance, each metric $\bar{g}_i^{1}$ indeed descend to
a quotient metric $[\bar{g}_i^{1}]$ on $FM_i\slash \mathbb{T}_i$, which remains
to be $O(m)$ equivariant. Notice that since a quotient map does not increase the
distance, we have $(FM_i\slash \mathbb{T}_i,[\bar{g}_i^{1}])$ being
$\Psi(\delta_i)$ close to $(Y,g_Y)$ in the Gromov-Hausdorff topology, with the
 $[\bar{f}_i]$ providing the $O(m)$ equivariant Gromov-Hausdorff
approximation.

Morevoer, the sectional curvature of $(FM_i\slash \mathbb{T}_i,[\bar{g}_i^1])$
can be uniformly (independent of $i$) controlled accordingly. Since
$\mathbb{T}_i$ is abelian, each $\mathbb{T}_i$ orbit, equipped with the
invariant metric which is the restriction of $\bar{g}_i^{1}$, is actually flat.
Therefore, by the approximation (\ref{eqn: FM_invariant_metric}), the uniform
regularity of $\bar{g}_i$ (\ref{eqn: gbar_regularity}) and O'Neill's formula,
we could uniformly bound the sectional curvature of the quotient metric
$[\bar{g}_i^{1}]$ provided that we have have a uniform estimate on the second
fundamental form of each $\mathbb{T}_i$ orbit. To obtain such estimate, we
notice that each $\mathbb{T}_i$ orbit sits in some $\bar{f}_i$ fiber, and the
$(N_i)_L$ invariance of the metric $\bar{g}_i^1$ tells that the estimate
consists of two parts: those directions orthogonal to the $\bar{f}_i$ fibers,
and those directions within each $\bar{f}_i$ fiber but orthogonal to the
$\mathbb{T}_i$ orbits. Along the directions orthogonal to the $\bar{f}_i$
fibers, since $(Y,g_Y)$ is a closed Riemannian manifold, there is a uniform
(independent of $i$) lower bound of the injectivity radius for each point in
$Y$, whence a uniform (independent of $i$) upper bound of the second fundamental
form in the directions perpendicular to the $\bar{f}_i$ fibers, controlled by
(\ref{eqn: barf_regularity} and (\ref{eqn: FM_invariant_metric}). In the
directions within the $\bar{f}_i$ fiber but perpendicular to a given
$\mathbb{T}_i$ orbit, by (\ref{eqn: FM_invariant_metric}), we notice that
$\bar{\nabla}^{LC1}_i$, the Levi-Civita connection for $\bar{g}^1_i$, is close
to $\bar{\nabla}^{LC}_i$, the Levi-Civita connection of $\bar{g}_i$, and the
connection $\bar{\nabla}^{\ast}_i$, which is fiber-wise flat with parallel
tortion and defines the fiber-wise affine structure, is also $C^{\infty}$ close
to $\bar{\nabla}^{LC}$, see \cite[Proposition 3.6]{CFG92}; therefore, since
tagent vectors tangent to the $\mathbb{T}_i$ orbits generate \emph{left}
invariant vector fields along the $\mathbb{T}_i$, they are
$\bar{\nabla}^{\ast}_{i}$ parallel, and consequently, the
$\bar{\nabla}^{LC1}_{i}$-covariant derivative of \emph{left} invariant vector
fields along the $\mathbb{T}_i$ orbits are of uniformly controlled size.
Summarizing, we see that
\begin{align}\label{eqn: II_form_T}
\sup_{z\in FM}|II_{\mathbb{T}_i(z)}|\ \le\ C_{5.5},
\end{align}
whence a desired uniform (independent of $i$) curvature bound of $(FM_i\slash
\mathbb{T}_i,[\bar{g}_i^1])$.

 Therefore, $\{(FM_i\slash \mathbb{T}_i,[\bar{g}_i^1])\}$ is a sequence
 collapsing to $(Y,g_Y)$ with bounded curvature and diameter, and we could
 define a limit weighted density function $\bar{\chi}_Q$ on $Y$ as
\begin{align*}
\forall y\in Y,\quad \bar{\chi}_Q(y)\ =\
\lim_{i\to \infty}
\frac{\left|[\bar{f}_i]^{-1}(y)\right|_{[\bar{g}_i^{1}]}}{ |FM_i\slash
\mathbb{T}_i|_{[\bar{g}_i^{1}]}}.
\end{align*}
Notice that here we may have passed to a sub-sequence. By \cite[Theorem
0.6]{Fukaya87}, $\bar{\chi}_Q$ is positive and continuous on $Y$, and is
constant along the $O(m)$ orbits. 
Consequently, $\bar{\chi}_Q$ naturally descends to a continous and positive
function $\chi_Q$ on $X$, with $\bar{\chi}_Q =\pi^{\ast}_Y\ \chi_Q$, where
$\pi_Y: Y\to X$ is the quotient map.
Now we define the \emph{limit central density function} $\chi_C:=\chi_X\slash
\chi_Q$. Clearly $\chi_C$ is a continuous and non-negative function on $X$, and
for any $x\in \mathcal{R}$, by the constancy of $\bar{\chi}_Q$ on $O(m)$
orbits, we have
\begin{align}\label{eqn: chi_bar_chi}
\chi_C(x)\ =\ \int_{\pi_Y^{-1}(x)}\frac{\bar{\chi}}{\bar{\chi}_Q}\
\darea_{\pi^{-1}(x)}.
\end{align}
Consequently, $\chi_C$ also characterizes
$\tilde{\mathcal{S}}$ as its zero locus:
\begin{align}
\chi_C^{-1}(0)\ =\ \tilde{\mathcal{S}}.
\end{align}

\subsection{Local representation over the orbifold regular part}
The limit central density function $\chi_C$ we just defined is global, 
continuous and characterizes $\tilde{\mathcal{S}}$ as its zero locus; however,
the definition via taking quotient rarely provides any insight into the local
geometry. In this sub-section, we would like to further understand the local
behavior of $\chi_C$ on $\tilde{\mathcal{R}}$ by taking locally defined central
sub-bundles of the fibrations $f_i:M_i\to X$, which is the $O(m)$ equivariant
quotient of $\bar{f}_i:FM_i\to Y$.

\subsubsection{A local central sub-bundle} To begin with, we first consider a
fixed $i$ and omit writing the index $i$ for a while. We also fix some $x_0\in
\tilde{\mathcal{R}}$, i.e. $x_0$ may be a regular point or an orbifold singular
point in $X$. 
Recall that by Proposition \ref{prop: unwrapping} we have the following data: an
orbifold neighborhood $U_{x_0}$ with an orbifold covering $q_{x_0}:V_{x_0}\to
U_{x_0}$, a finite group $G_{x_0}$ acting on $V_{x_0}$ giving the quotient, a
$G_{x_0}$ invariant metric $\hat{g}_{x_0}$ on $V_{x_0}$ that descends to
the metric on $X$, a $G_{x_0}$ equivariant fiber bundle
$\hat{f}:\hat{V}_{x_0}\to V_{x_0}$ with infranil fibers, and a local section
$\hat{s}^{\ast}: V_{x_0}\to \hat{V}_{x_0}$. Moreover, there is a $G_{x_0}$
equivariant connection $\hat{\nabla}^{\ast}$ whose restriction
$\hat{\nabla}^{\ast}_{\hat{x}}$ to $\hat{f}^{-1}(\hat{x})$ is flat with
parallel torsion, for each $\hat{x}\in V_{x_0}$.

The local section $\hat{s}^{\ast}$ helps us construct the local central
sub-bundle $\hat{c}:C\hat{V}_{x_0}\to V_{x_0}$ which is $G_{x_0}$ equivariant.
More specifically, since $\hat{\nabla}^{\ast}$ makes each $\hat{f}$ fiber into
a homogeneous space $N\slash \Gamma$, on which $N_L$ acts, we may focus on the
action of the center sub-group $C(N)$. At each point in a $\hat{f}$ fiber, the
orbit of the $C(N)$ action is nothing but a $k_0$-dimensional torus
$\mathbb{T}=C(N)\slash (C(N)\cap \Gamma)$. Now for each $\hat{x}$, we specify a
torus orbit that passes through $\hat{s}^{\ast}(\hat{x})$, denoted by
$\mathbb{T}(\hat{s}^{\ast}(\hat{x}))$.
Letting $\hat{x}\in V_{x_0}$ vary, we can form a subset of $\hat{V}_{x_0}$:
\begin{align*}
C\hat{V}_{x_0}\ :=\ \bigcup_{\hat{x}\in V_{x_0}}\mathbb{T}
(\hat{s}^{\ast}(\hat{x})).
\end{align*}
Notice that each $\mathbb{T}(\hat{s}^{\ast}(\hat{x}))$ is affine diffeomorphic
to $C(N)\slash C(L)$, as specified by the flat connection
$\hat{\nabla}^{\ast}_{\hat{x}}$. To see that $C\hat{V}_{x_0}$ is a smooth
sub-manifold of $\hat{V}_{x_0}$, we notice that the center sub-algebra of the
fiber-wise nilpotent Lie algebra can be characterized as the kernel of the
torsion tensor $T_|$ of $\hat{\nabla}^{\ast}_|$ (denoting 
the restriction of corresponding objects to an arbitrary $\hat{f}$ fiber).
By the smoothness of $\hat{\nabla}^{\ast}$, we know that $T_|$ varies smoothly
throughout $\hat{V}_{x_0}$, therefore specifing a smooth distribution of
commuting vector fields along the $\hat{f}$ fibers, and integrating these vector
fields we get the leaves as torus orbits within the $\hat{f}$ fibers. Now the
smoothness of $C\hat{V}_{x_0}$ is determined by the smoothness of
$\hat{s}^{\ast}$: these are the initial values telling us which leaf in each
$\hat{f}$ fiber to choose. With this understanding, we clearly see that the
map
\begin{align*}
\hat{c}:\ C\hat{V}_{x_0}\ \to\ V_{x_0}
\end{align*}
sending each $\mathbb{T}(\hat{s}^{\ast}(\hat{x}))$ to $\hat{x}\in V_{x_0}$ is a
smooth fibration with $\hat{c}^{-1}(\hat{x})
=\mathbb{T}(\hat{s}^{\ast}(\hat{x}))$.

Moreover, the fibration $\hat{c}:C\hat{V}_{x_0}\to V_{x_0}$ is $G_{x_0}$
equivariant: recall that the action of $G_{x_0}$ on the fibers is
decomposed into two parts --- a finite central rotation part $S_{x_0}\le
\mathbb{T}$ and a finite automorphism part $\Lambda_{x_0}\le Aut(\Gamma)$; since
$\Lambda_{x_0}$ sends $\Gamma$ orbits to $\Gamma$ orbits in the corresponding
$\tilde{f}$ fibers, elements of $\Lambda_{x_0}$ determine affine diffeomorphisms
between the corresponding $\tilde{f}$ fibers, by Malcev's rigidity theorem
(\cite[Theorem 3.7]{CFG92}), and consequently an element of $\Lambda_{x_0}$
sends an entire $\mathbb{T}$ orbit to a $\mathbb{T}$ orbit in the corresponding
$\hat{f}$ fiber; on the other hand, elements in $S_{x_0}$ only rotates the
$\mathbb{T}$ orbit, therefore keeping the entire orbit invariant. Equivalently,
we have $\forall g\in G_{x_0}$ and $\forall \hat{x}\in V_{x_0}$,
\begin{align*}
g.\mathbb{T}(\hat{s}^{\ast}(\hat{x}))\ =\ \mathbb{T}(g.\hat{s}^{\ast}(\hat{x}))\
=\ \mathbb{T}(\hat{s}^{\ast}(g.\hat{x})),
\end{align*}
whence the $G_{x_0}$ equivariance of the fibration $\hat{c}$. Notice that
it may be the case that $g.\hat{s}^{\ast}(\hat{x})\not=
\hat{s}^{\ast}(g.\hat{x})$, due to the possibly non-trivial part in $S_{x_0}$;
but the corresponding $\mathbb{T}$ orbits have to agree, since elements
$S_{x_0}$ only rotate the $\mathbb{T}$ orbits. As a consequence,
$\hat{c}:C\hat{V}_{x_0}\to V_{x_0}$ descends to a singular fibration $c:
CU_{x_0}\to U_{x_0}$, where $CU_{x_0}=C\hat{V}_{x_0}\slash G_{x_0}$ is a smooth
sub-manifold of $M$. A regular $c$ fiber is diffeomorphic to $\mathbb{T}$,
while a singular fiber $c^{-1}(x)$ is diffeomorphic to $\mathbb{T}\slash G_{x}$
with $G_x\le G_{x_0}$ being the isotropy group of $x\in
\tilde{\mathcal{R}}\backslash \mathcal{R}$.

In order to connect the locally constructed central sub-bundle with the limit
cnetral density function, we still need to explain its relation with the frame
bundle $FM$ restricted to the sub-manifold $CU_{x_0}$. Recall that
$G_{x_0}$ can be regarded as a normal sub-group of $O(m)$, and we have the
following commutative diagram, in coorespondence to (\ref{eqn: diagram_f}):
\begin{align}\label{eqn: diagram_c}
\begin{xy}
<0em,0em>*+{V_{x_0}}="v", <-7em,2em>*+{C\hat{V}_{x_0}}="w",
<0em,-5em>*+{U_{x_0}}="u", <-7em,-3em>*+{CU_{x_0}}="z",
<5em,7em>*+{FM|_{CU_{x_0}}}="fm", <12em,6em>*+{\pi^{-1}(U_{x_0})}="p",
 "w";"v" **@{-} ?>*@{>}?<>(.5)*!/_0.5em/{\scriptstyle \hat{c}}, 
 "w";"z" **@{-} ?>*@{>} ?<>(.5)*!/^0.5em/{\scriptstyle \slash G_{x_0}\ \
 }, 
 "v";"u" **@{-} ?>*@{>} ?<>(.5)*!/^0.5em/{\scriptstyle q_{x_0}}, 
 "z";"u" **@{-} ?>*@{>} ?<>(.5)*!/_0.5em/{\scriptstyle c},
  "fm";"p" **@{-} ?>*@{>} ?<>(.5)*!/_0.5em/{\scriptstyle \bar{c}},
 "fm";"w" **@{-} ?>*@{>} ?<>(.5)*!/^0.5em/{\scriptstyle \hat{\pi}}, 
  "p";"v" **@{-} ?>*@{>} ?<>(.5)*!/^0.5em/{\scriptstyle \hat{\pi}_Y}, 
  "p";"u" **@{-} ?>*@{>} ?<>(.5)*!/_0.5em/{\scriptstyle \pi_Y}, 
  "fm";"z" **@{--} ?>*@{>} ?<>(.5)*!/_0.5em/{\scriptstyle \pi}
\end{xy}
\end{align}
Here $\hat{\pi}$ and $\hat{\pi}_Y$ are taking quotients by the group action
$O(m)\slash G_{x_0}$. Let us further explain the fibration $\bar{c}:
FM|_{CU_{x_0}}\to \pi_Y^{-1}(U_{x_0})\subset Y$ given by restricting $\bar{f}$
to $FM|_{CU_{x_0}}$. Considering the natural $O(m)$ equivariant fibration
$\bar{f}:FM\to Y$ associated to $f:M\to Y$, each whose fibers being affine
diffeomorphic to a nilmanifold, and recalling that in defining the quotient
fibration $[\bar{f}]: FM\slash \mathbb{T}\to Y$, we also have an action of the
central torus $\mathbb{T}$ on each $\bar{f}$ fiber. In order to express
$\chi_C$ in terms of the limit relative volume distribution of the central
sub-bundles of the frame bundle, we need to check that $FM|_{CU_{x_0}}$
consists of entire $\mathbb{T}$ orbits, one in each $\hat{f}$ fiber over
$\pi_Y^{-1}(U_{x_0})$.
To see this, for any $\hat{x}\in V_{x_0}$, we have the embedded sub-manifold 
\begin{align*}
S_{\hat{x}}\ :=\ \hat{\pi}^{-1}(\hat{f}^{-1}(\hat{x}))\ =\
\bar{f}^{-1}(\hat{\pi}_Y^{-1}(\hat{x})),
\end{align*} 
which is invariant under the $O(m)$ actions and infinitesimal $N_L$ actions. 
Since $C(N)\slash C(L)=\mathbb{T}$ is compact, the infinitesimal
action of $C(N)$ on $S_{\hat{x}}$ integrates to the action by the compact
abelian group $\mathbb{T}$. By \cite[Proposition 4.3]{CFG92}, we know that the
$O(m)$ action commutes with the $\mathbb{T}$ action. Therefore, picking any 
$z\in \hat{\pi}^{-1}(\hat{c}^{-1}(\hat{x}))\subset S_{\hat{x}}$, we know that a
typical $\pi$ fiber in $FM|_{CU_{x_0}}$ is 
\begin{align*}
\hat{\pi}^{-1}(\hat{c}^{-1}(\hat{x}))\ =\ O(m)(\mathbb{T}(z))\ =\
\mathbb{T}(O(m)(z)),
\end{align*}
which is a union of $\mathbb{T}$ orbits. On the other hand, by the $O(m)$
equivariance of $\bar{f}$, we know that
\begin{align*}
\bar{f}(\hat{\pi}^{-1}(\hat{c}^{-1}(\hat{x})))\ =\ O(m)(\bar{f}(z))\ \subset\
\pi_Y^{-1}(U_{x_0}).
\end{align*}
Therefore, we have verified that $\bar{c}:FM|_{CU_{x_0}}\to \pi_Y^{-1}(U_{x_0})$
is indeed an $O(m)$ equivariant fiber bundle, each of whose fiber being a
$\mathbb{T}$ orbit in the fiber of $\bar{f}$.

\subsubsection{Quantitative and limit behavior of the local central sub-bundle} 
To study the metric measure structure related to $\chi_C$, we again appeal to
the approximating invariant metric $\bar{g}^{1}$ defined on $FM$, see
Proposition~\ref{prop: invariant_metric}. Recall that $\bar{g}^{1}$ is
$\Psi(\delta)$ close to $\bar{g}$ in the $C^{\infty}$ sense, see (\ref{eqn:
FM_invariant_metric}), and it is invariant under both the (infinitesimal) $N_L$
and $O(m)$ actions.
We could then put the Riemannian metric $\bar{g}^{1}_|$, the restriction of
$\bar{g}^{1}$ to each $\bar{c}$ fiber, making $(FM|_{CU_{x_0}},\bar{g}^{1}_|)$
an embedded Riemannian sub-manifold of $(FM,\bar{g}^1)$, fibering over
$(\pi_Y^{-1}(U_{x_0}),g_Y)$ by the Riemannian submersion $\bar{c}$ --- since the
fiber-wise $\mathbb{T}$ action leaves $\bar{g}^{1}_|$ invariant, and
$\bar{g}^{1}$ is taken as the average of $\bar{g}$ along the $\bar{f}$ fibers,
the quotient metric of $\bar{g}^1$ coincides with $g_Y$.



 The invariance of $\bar{g}^{1}$ and the commutativity of the infinitesimal
 $N_L$ and $O(m)$ actions ensure that for each $y\in \pi_Y^{-1}(U_{x_0})$,
\begin{align}\label{eqn: fiber_vol_prod}
\left|\bar{f}^{-1}(y)\right|_{\bar{g}^{1}}\ =\ 
\left|\bar{c}^{-1}(y)\right|_{\bar{g}^{1}_|} 
\left|[\bar{f}]^{-1}(y)\right|_{[\bar{g}^{1}]},
\end{align}
where we recall that $[\bar{f}]:(FM,[\bar{g}^{1}])\slash \mathbb{T}\to
(Y,g_Y)$ is the quotient bundle of $FM$ by the isometric $\mathbb{T}$ action on
each fiber, equipped with the quotient metric $[\bar{g}^{1}]$.

We notice that the Riemannian submersion $\bar{c}:(FM|_{CU_{x_0}},\bar{g}^1_|)
\to (\pi_Y^{-1}(U_{x_0}),g_Y)$ is a $\Psi(\delta)$-Gromov-Hausdorff
approximation, and that $(FM|_{CU_{x_0}},\bar{g}^1_|)$ has uniformly bounded
sectional curvature, thanks to the uniform second fundamental form estimate
\begin{align*}
\sup_{z\in FM|_{f^{-1}(U_{x_0})}}|II_{\mathbb{T}(z)}|\ \le\
C_{5.5},
\end{align*}
which can be derived in a way similar to (\ref{eqn: II_form_T}). Consequently,
by \cite[(3.5)]{Fukaya87} we see that there is a uniform constant
$C_{5.10}(x_0)>1$ depending only on $x_0\in X$ (especially independent of
$\delta =d_{GH}(M,X)$), such that
\begin{align}
\sup_{y\in \pi_Y^{-1}(U_{x_0})}|\bar{c}^{-1}(y)|_{\bar{g}^{1}_|}\ \le\ 
C_{5.10}(x_0)\inf_{y\in \pi_Y^{-1}(U_{x_0})}|\bar{c}^{-1}(y)|_{\bar{g}^{1}_|}.
\end{align}
Now by the co-area formula and integrating (\ref{eqn: fiber_vol_prod}) over
$\pi_Y^{-1}(U_{x_0})$, we see that for any $y\in \pi_Y^{-1}(U_{x_0})$,
\begin{align*}
C_{5.10}(x_0)^{-1}\ \le\ 
\frac{\left|\bar{f}^{-1}(\pi_Y^{-1}(U_{x_0}))\right|_{\bar{g}^{1}}}
{\left|[\bar{f}]^{-1}(\pi_Y^{-1}(U_{x_0}))\right|_{[\bar{g}^{1}]}
\left|\bar{c}^{-1}(y)\right|_{\bar{g}^{1}_|}}\ \le\ C_{5.10}(x_0)
\end{align*}
and therefore, integrating over $\pi_Y^{-1}(U_{x_0})$ again we get,
\begin{align} \label{eqn: local_measure_product}
C_{5.10}(x_0)^{-1}\ \le\ C_{5.11}(x_0)\ :=\ 
\frac{\left|\bar{f}^{-1}(\pi_Y^{-1}(U_{x_0}))\right|_{\bar{g}^{1}}
|\pi_Y^{-1}(U_{x_0})|_{g_Y}}{\left|[\bar{f}]^{-1}(\pi_Y^{-1}(U_{x_0}))
\right|_{[\bar{g}^{1}]} \left|FM|_{CU_{x_0}}\right|_{\bar{g}^{1}_|}}\ \le\
C_{5.10}(x_0).
\end{align}

On the other hand, since the Riemannian metric $\bar{g}^{1}$ is obtained
from $\bar{g}$ by averaging along the $N$ directions, and both are $O(m)$
invariant, there is a natural $\mathbb{T}$ invariant metrics $\hat{g}^{1}_|$
defined on $C\hat{V}_{x_0}$, which is the restriction to the central sub-bundle
of the invariant metric defined as \cite[(4.8)]{CFG92}. Moreover, by the
definition of $\bar{g}$ and $\bar{g}^1$, we know that the $O(m)$ orbits in
$\bar{c}^{-1}(\pi_Y^{-1}(U_{x_0})) =FM|_{CU_{x_0}}$ has volume approximated by 
$|O(m)|$ in its standard metric, see (\ref{eqn: O(m)_approximate_volume}). 
Therefore, for any $\hat{x}\in V_{x_0}$, by (\ref{eqn: diagram_c}) we have
$\hat{\pi}_Y^{-1}(\hat{x})=\pi_Y^{-1}(q_{x_0}(\hat{x}))\subset Y$, and can
compute the volume of the closed submanifold $\bar{c}^{-1}
(\hat{\pi}_Y^{-1}(\hat{x})) =FM|_{c^{-1}(q_{x_0}(\hat{x}))}\subset FM$ by
Fubini's theorem:
\begin{align}\label{eqn: Fubini}
\begin{split}
\int_{\hat{\pi}_Y^{-1}(\hat{x})}|\bar{c}^{-1}(y)|_{\bar{g}^1_|}\
\darea_{\hat{\pi}_Y^{-1}(\hat{x})}(y)\ =\ &\int_{\hat{c}^{-1}(\hat{x})}
|\hat{\pi}^{-1}(z)|\ \darea_{\hat{c}^{-1}(\hat{x})}(z)\\
=\ &(1+\Psi(\delta|r_{x_0}^{-1}))|G_{x_0}|^{-1}|O(m)|
|\hat{c}^{-1}(\hat{x})|_{g^1_|},
\end{split}
\end{align}
where for each $z\in \hat{c}^{-1}(\hat{x})$, $|\hat{\pi}^{-1}(z)|
=(1+\Psi(\delta|r_{x_0}^{-1}))|G_{x_0}|^{-1}|O(m)|$ by (\ref{eqn:
O(m)_approximate_volume}) and (\ref{eqn: diagram_c}); see also (\ref{eqn:
FM_Fubini}).

Now we restore the index $i$ and the density functions $\bar{\chi}$ and
$\bar{\chi}_Q$ are defined respectively by the limit weighted volume of the
fibers of $\bar{f}_i$ and $[\bar{f}_i]$. From (\ref{eqn: diagram_c}) and the
singular nature of $U_{x_0}$, our goal will be to construct a $G_{x_0}$
invariant function $\hat{\chi}_C$ on $V_{x_0}$, so that it descends to a
constant multiple of $\chi_C$ on $U_{x_0}$. The function $\hat{\chi}_C$ could be
defined as the asymptotic relative volume distribution of the $\hat{c}_i$
fibers, in a similar way to (\ref{eqn: hat_chi}):
\begin{align}\label{eqn: hat_chi_C}
\forall \hat{x}\in V_{x_0},\quad \hat{\chi}_C(\hat{x})\ :=\ \lim_{i\to
\infty}\frac{\left|\hat{c}_i^{-1}(\hat{x})\right|_{\hat{g}_{i|}^1}}
{\left|C\hat{V}_{x_0,i}\right|_{\hat{g}_{i|}^1}}.
\end{align}
Here notice that the metrics $\hat{g}_{i|}$ enjoy the uniform regularity control
due to (\ref{eqn: II_form_T}). Moreover, defining $\hat{\chi}_C$ using
$\hat{g}_{i|}^1$ or $\hat{g}_{i|}$ makes no difference, in view of Proposition
\ref{prop: invariant_metric}. We may have already passed to a further
sub-sequence in taking limit, and $\hat{\chi}_C$ could be thus defined because
$\{(C\hat{V}_{x_0,i},\hat{g}_{i|}^1)\}$ has uniform curvature bound, and the
collapsing fibration $\hat{c}_i:C\hat{V}_{x_0,i}\to V_{x_0}$ is \emph{regular},
over the open set $V_{x_0}\subset \mathbb{R}^n$, with
$d_{GH}(C\hat{V}_{x_0,i},V_{x_0})\le \Psi(\delta_i)$.
Our next goal is then to express $q_{x_0}^{\ast} (\chi_C|_{U_{x_0}})$ as a
constant multiple of $\hat{\chi}_C$ on $V_{x_0}$.

To achieve this, we start with understanding $\bar{\chi}$ on
$\pi_Y^{-1}(U_{x_0})$. From (\ref{eqn: fiber_vol_prod}) and (\ref{eqn:
local_measure_product}) we see for any $y\in \pi_Y^{-1}(U_{x_0})$ that
\begin{align}\label{eqn: bar_chi_y}
\begin{split}
\bar{\chi}(y)\ =\ &\lim_{i\to
\infty}\frac{\left|\bar{f}_i^{-1}(\pi_Y^{-1}(U_{x_0}))\right|_{\bar{g}^{1}_i}}
{|FM_i|_{\bar{g}^{1}_i}}
\frac{\left|\bar{f}_i^{-1}(y)\right|_{\bar{g}^{1}_i}}
{\left|\bar{f}_i^{-1}(\pi_Y^{-1}(U_{x_0}))\right|_{\bar{g}^{1}_i}}\\
=\ &\frac{\mu_X(U_{x_0})}{\mu_X(X)}\lim_{i\to \infty}
\frac{\left|\pi_Y^{-1}(U_{x_0})\right|_{g_Y}}{C_i'}
\frac{\left|[\bar{f_i}]^{-1}(y)\right|_{[\bar{g}^1_i]}}
{\left|[\bar{f}]^{-1}(\pi_Y^{-1}(U_{x_0}))\right|_{[\bar{g}^{1}]}}
\frac{\left|\bar{c}^{-1}_i(y)\right|_{\bar{g}^{1}_{i|}}}
{\left|FM_i|_{C_iU_{x_0}}\right|_{\bar{g}_{i|}^{1}}}\\
 =\ &C_{5.14}(x_0)\frac{\mu_X(U_{x_0})}{\mu_X(X)}
 \left(\int_{\pi_Y^{-1}(U_{x_0})} \bar{\chi}_Q\ \dvol_{g_Y}\right)^{-1}
 \bar{\chi}_Q(y)
\lim_{i\to \infty}
\frac{\left|\bar{c}^{-1}_i(y)\right|_{\bar{g}^{1}_{i|}}}
{\left|FM_i|_{C_iU_{x_0}}\right|_{\bar{g}_{i|}^{1}}},
\end{split}
\end{align} 
where $C_{5.14}(x_0):=|\pi_Y^{-1}(U_{x_0})|_{g_Y}C_{5.11}(x_0,\infty)^{-1}$,
with some 
\begin{align*}
C_{5.11}(x_0,\infty)\ :=\ \lim_{i\to \infty} C_{5.11}(x_0,i)\ \in\ 
[C_{5.10}(x_0)^{-1},C_{5.10}(x_0)]
\end{align*} 
as the limit (possibly passing to a sub-sequence) determined by the quantity in
(\ref{eqn: local_measure_product}), and the next factor in the same line is a
consequence of the facts that $\bar{f}_i^{-1}(\pi_Y^{-1}(U_{x_0}))
=FM_i|_{f_i^{-1}(U_{x_0})}$ and that $U_{x_0}\subset \tilde{\mathcal{R}}$.
 Consequently, for any $x\in U_{x_0}\cap \mathcal{R}$ and any $\hat{x}\in
 q_{x_0}^{-1}(x)$, since $\pi_Y^{-1}(x)=\hat{\pi}^{-1}_Y(\hat{x})\subset Y$, by
 (\ref{eqn: bar_chi_y}), (\ref{eqn: chi_bar_chi}) and (\ref{eqn: Fubini}), we
 deduce, in a similar manner leading to (\ref{eqn: Fubini_f}), that
\begin{align*}
\begin{split}
 \chi_C(x)\ 
=\ C_{5.2}(x_0) \lim_{i\to \infty}
\frac{\left|\hat{c}_i^{-1}(\hat{x})\right|_{g^1_{i|}}}
{\left|C\hat{V}_{x_0,i}\right|_{g^1_{i|}}},
\end{split}
\end{align*}
where the coefficient
\begin{align*}
C_{5.2}(x_0)\ :=\ 
C_{5.14}(x_0)\frac{\mu_X(U_{x_0})}{\mu_X(X)}
\left(\int_{\pi_Y^{-1}(U_{x_0})} \bar{\chi}_Q\ \dvol_{g_Y}\right)^{-1} 
\end{align*} 
is a constant independent of $x\in U_{x_0}\cap \mathcal{R}$. This tells
that $q_{x_0}^{\ast}(\chi_C|_{U_{x_0}\cap \mathcal{R}})= \hat{\chi}_C$ on
$V_{x_0}\cap q_{x_0}^{-1}(U_{x_0}\cap \mathcal{R})$. But by the same argument
leading to (\ref{eqn: chi_hat_chi}), we know that the above identity extends
all over $V_{x_0}$ --- $\mathcal{R}\cap U_{x_0}$ is dense in $U_{x_0}$ and
$\chi_C$ is continuous on $U_{x_0}$ --- and we have
\begin{align}\label{eqn: rep_hat_chi_C}
q_{x_0}^{\ast}\ (\chi_C|_{U_{x_0}})\ =\ C_{5.2}(x_0)\ \hat{\chi}_C\quad
\text{on}\ V_{x_0}.
\end{align}

\subsection{Limit central distribution and its density} We recall that the 
construction of the the central sub-bundles $\hat{c}_i:C\hat{V}_{x_0,i}\to
V_{x_0}$ in \S 5.2.1 are based on the same local section $\hat{s}^{\ast}_i$
employed in the local trivialization, as discussed in \S 3.2 (especially \S
3.2.2). Moreover, in trivializing the unwrapped neighborhoods
$\tilde{V}_{x_0,i}$, such local sections are lifted to $\tilde{s}_i:
V_{x_0,i}\to \tilde{V}_{x_0,i}$. It is therefore straightforward to see that
the central sub-bundles $C\hat{V}_{x_0,i}$ can also be lifted to sub-bundles
$\tilde{c}_i: C\tilde{V}_{x_0,i}\to V_{x_0}$ of $\tilde{V}_{x_0,i}$.
Each fiber $\tilde{c}_i^{-1}(\hat{x})= C(N_i)(\tilde{s}_i(\hat{x}))\subset
\tilde{f}_i^{-1} (\hat{x})$ is then affine diffeomorphic to $C(N_i)\triangleleft
N_i$, when equipped with the flat connection as the restriction of
$\tilde{\nabla}^{\ast}_{i,\hat{x}}$ to the sub-manifold, and with
$\tilde{s}_i(\hat{x})$ chosen as the base point. In fact, by the way we define
the local trivializations $\phi_{x_0,i}: W_{x_0}\to \tilde{V}_{x_0}$ according
to the lifted sections $\tilde{s}_i$, we have associated trivializations of
$C\tilde{V}_{x_0,i}$ by directly restricting $\phi_{x_0,i}$ to $CW_{x_0}
:=V_{x_0} \times \mathbb{R}^{k_0}$.
For each induced connection $\tilde{\nabla}_i$ on $W_{x_0}$ (see \S 3.3.1), the
null space of the torsion tensor (of $\tilde{\nabla}_i$) defines a foliation of
$W_{x_0}$, and $CW_{x_0}$ can also be characterized as a smoothly parametrized
family of its leaves passing through $\tilde{s}_i$ within each $p$ fiber ---
recalling that $p:W_{x_0}\to V_{x_0}$ is the projection to the first factor of
$W_{x_0}=V_{x_0}\times \mathbb{R}^{m-n}$.

By the way we define the metrics and connections on $C\hat{V}_{x_0,i}$, it is
obvious that the lifted metrics and connections on $C\tilde{V}_{x_0,i}$ are
nothing but the restriction of the lifted metrics $\tilde{g}_i$ and connections
$\tilde{\nabla}_i$ to the fibers of $\tilde{c}_i$. The regularity of the
restricted metrics $\{\tilde{g}_{i|}\}$ and connections
$\{\tilde{\nabla}_{i|}\}$ are then readily controlled as that of $\tilde{g}_i$
and $\tilde{\nabla}_i$, in view of the uniform second fundamental form control
(\ref{eqn: II_form_T}).

Since each $\tilde{c}_i$ fiber is equipped with the restricted connection of
$\tilde{\nabla}_i$, the action of $\Gamma_i$ on the $p$ fibers restricts to an
action on the $\tilde{c}_i$ fibers by affine isometries. Since $\Gamma_i$ is a
finite extension of a cocompact lattice $L_i\le (N_i)_L$, the translation part
of its action on the $\tilde{c}_i$ fibers is by $C(L_i)=L_i\cap C(N_i)$. Letting
$G_i:=\Gamma_i\slash L_i$, we know that $|G_i|\le C(m,n)$ and
$G_i\le Aut(N_i)$. Moreover, since $\Gamma_i\le N_L\rtimes Aut(N_i)$ and
$G_i$ preserves the center $C(N_i)$, we have $C\Gamma_i:= C(L_i)\rtimes
G_i$ as the fundamental group of the $c_i$ fibers, acting on the
$\tilde{c}_i$ fibers by affine isometries, with respect to the restricted
metrics and connections.

We could therefore equip $C\Gamma_i$ with a metric restricted from $Aff(C(N_i))=
C(N_i)\rtimes Aut(N_i)$, where the size of the $Aut(N_i)$ part is measured by
the standard metric on $O(m-n)$. Since the action of $C\Gamma_i$, when restriced
on each $\tilde{c}_i$ fiber, preserves the lattice $C(L_i)\subset C(N_i)$,
which is isomorphic to an integral lattice in the abelian group
$(\mathbb{R}^{k_{0}},+,o)$, we can think of the $G_i$ part of the
$C\Gamma_i$ action as in $GL(k_0,\mathbb{Z})$, and therefore, by the uniform
upper bound of $|G_i|$, we know that elements of $C\Gamma_i$ has, for
their $G_i$ part, a uniform lower bound $3\varepsilon_0(m,n)$ in norm
(see \cite{BuserKarcher81}), independent of $i$ and $x\in U_{x_0}$.

We also recall the definition of the limit central distribution $\mathfrak{C}$
in \S 3.3.2, and notice that the leaves of $\mathfrak{C}$ passing through
$V_{x_0}\times \{o\}$ are exactly the fibers of $CW_{x_0}$. Moreover, by
(\ref{eqn: H_convergence}) we have 
\begin{align}
\forall \hat{x}\in V_{x_0},\quad \det H(\hat{x})\ =\ \det
\tilde{g}_{\infty}|_{\{\hat{x}\}\times \mathbb{R}^{m-n}}\ =\ \lim_{i\to
\infty}\det \tilde{g}_{i|}(\hat{x}).
\end{align}
Also notice that we have normalized so that $\det{H}(x_0)=1$. Now we have the
following proposition:
\begin{prop}\label{prop: central_density} 
Fix $x_0\in \tilde{\mathcal{R}}$. Let $U_{x_0}$ and $V_{x_0}$ be neighborhoods
that fit into the diagram (\ref{eqn: diagram_c}), and let $H$ be defined in \S
3.3.2. There is a constant $C_{5.3}=C_{5.3}(x_0)$ such that
$\hat{\chi}_C=C_{5.3}\sqrt{\det H}$ on $V_{x_0}$.
\end{prop}

\begin{proof}(\emph{Following \cite[Lemma 2-5]{Fukaya89}}.)  
For each $\gamma\in C\Gamma_i$, let $\|\gamma\|$ denote its norm
mentioned above. For $\gamma \in Z$, we let $\|\gamma\|$ denote its norm in
$Aff(Z)$, defined in a similar way. Then we consider the subsets
\begin{align*}
C\Gamma_i(\varepsilon_0)\ &:=\ \{\gamma\in C\Gamma_i:\ \|\gamma\|\le
\varepsilon_0\}\\
\text{and}\quad Z(\varepsilon_0)\ &:=\ \{\gamma\in Z:\ \|\gamma\|\le
\varepsilon_0\}.
\end{align*}
Clearly, by the discussion above, any $\gamma\in C\Gamma_i(\varepsilon_0)$ acts
on the $C(N_i)$ fibers by a left translation, and thus
$C\Gamma_i(\varepsilon_0)\subset CL_i$ as a finite subset. As a consequence of
(\ref{eqn: H_convergence}), $C\Gamma_i(\varepsilon_0)$ converges to
$Z(\varepsilon_0)$ in the pointed Gromov-Hausdorff topology (fixing the
identity element).

Now we see that the following subsets of $CW_{x_0}$ defined for any $\hat{x}\in
V_{x_0}$,
\begin{align*}
E_i(\hat{x},\varepsilon)\ :=\ \bigcup_{\gamma\in C\Gamma_i(\varepsilon_0)}
B_{\tilde{g}_{i|}}(\gamma (\hat{x},o),\varepsilon)\quad \text{and}\quad
E_{\infty}(\hat{x},\varepsilon)\ :=\ \bigcup_{\gamma\in
Z(\varepsilon_0)} B_{H}(\gamma (\hat{x},o),\varepsilon),
\end{align*}
 satisfy for any $\varepsilon <\varepsilon_0\slash 100$, that 
\begin{align}\label{eqn: limit_E}
\lim_{i\to \infty}\sup_{\hat{x}\in V_{x_0}} 
\left|\frac{|E_i(\hat{x},\varepsilon)|}{|E_{\infty}(\hat{x},\varepsilon)|}-1\right|\
=\ 0,
\end{align}
by the convergence of $C\Gamma_i(\varepsilon_0)$ and the convergence
of the underlying metrics $\tilde{g}_{i |}$ to $\tilde{g}_{\infty |}=H$ on
$CW_{x_0}$, see (\ref{eqn: H_convergence}); compare also
\cite[(2-10)]{Fukaya89}.

On the one hand, we can show that that for any $\hat{x}\in V_{x_0}$,
\begin{align}\label{eqn: E_and_det_H}
\lim_{\varepsilon\to 0} \frac{|E_{\infty}(\hat{x},\varepsilon)|}{\omega_n
\varepsilon^{n}}\ =\ v(\varepsilon_0)\sqrt{\det H}(\hat{x}),
\end{align}
where
$v(\varepsilon_0)=|Z(\varepsilon_0)|_{H}$. The above limit holds
because $E_{\infty}(\hat{x},\varepsilon)$ is nothing but an $\varepsilon$ tubular
neighborhood in $CW_{x_0}$ of $Z(\hat{x},\varepsilon_0):=\{\gamma (\hat{x},o):\
\gamma\in Z(\varepsilon_0)\}$ contained in $p^{-1}(\hat{x})$; by (\ref{eqn:
H_convergence}) and the constancy of $\sqrt{\det H}$ along the fiber
$p^{-1}(\hat{x})\subset W_{x_0}$, we have
\begin{align*}
|Z(\hat{x},\varepsilon_0)|\ =\ &\int_{Z(\varepsilon_0)}\sqrt{\det H}(\hat{x})\
\darea_{Z_{x_0}}\\
=\ &v(\varepsilon_0)\sqrt{\det H}(\hat{x}).
\end{align*} 
Notice that $v(\varepsilon_0)$ is a constant independent of $\hat{v}\in
V_{x_0}$. Also compare this with \cite[(2-11) and (2-12)]{Fukaya89}, and notice
that here it is our choice of the (sufficiently small) $\varepsilon_0$ that
enables us to explicitly relate the fiber-wise limit volume ratio
$|E_{\infty}(\hat{x},\varepsilon)| \epsilon^{-n}$ with $\sqrt{\det H}(\hat{x})$.

On the other hand, for any $\varepsilon<\frac{1}{10} \min\{\varepsilon_0,r_0\}$
fixed, and all sufficiently large $i$, 
 we could consider $\tilde{V}_{i}(\hat{x},\epsilon)\subset C\tilde{V}_{x_0,i}$,
 the fundamental domain of the universal covering of
 $\hat{c}_i^{-1}(B_{\hat{g}_X}(\hat{x}, \varepsilon))$ containing the base point
 $\tilde{s}_i(\hat{x})\in C\tilde{V}_{x_0,i}$. By abusing notations, we
 may regard $\tilde{V}_{i}(\hat{x}, \epsilon)$ as a neighborhood of
 $(\hat{x},o) \in CW_{x_0}$. The Hausdorff distance (measured within
 $(CW_{x_0},\tilde{g}_{i |})$ between $\cup_{\gamma\in C\Gamma_i(\varepsilon_0)}
\gamma\tilde{U}_{i}(\hat{x},\varepsilon)$ and $E_i(\hat{x},\varepsilon)$ is
bounded above by $\Psi(\delta_i)$, and therefore
 \begin{align}\label{eqn: limit_CGamma}
 \lim_{i\to\infty}|E_i(\hat{x},\varepsilon)|_{\tilde{g}_{i |}}\ =\
 \lim_{i\to\infty}|C\Gamma_i(\varepsilon_0)|
 |\tilde{V}_i(\hat{x},\varepsilon)|_{\tilde{g}_{i |}},
 \end{align}
 where $|C\Gamma_i(\varepsilon_0)|$ denotes the number of elements in
 $C\Gamma_i(\varepsilon_0)$, a number that tends to infinity as $i\to \infty$.
 
 But $|\tilde{V}_i(\hat{x},\varepsilon)|$ could also be computed by the co-area
 formula as following: 
 \begin{align}\label{eqn: U_integral}
 |\tilde{V}_i(\hat{x},\varepsilon)|_{\tilde{g}_{i |}}\ =\
 \int_{\tilde{c}_i(E_i(\hat{x},\varepsilon))}|\hat{c}_i^{-1}(y)|_{\hat{g}_{i|}}\ 
 \dvol_{\hat{g}_X}(y)+ |\tilde{V}_i(\hat{x},\varepsilon)|_{\tilde{g}_{i|}}
 \Psi(\delta_i),
 \end{align}
 since $\hat{c}_i:C\hat{V}_{x_0,i}\to V_{x_0}$ is a $\Psi(\delta_i)$ almost
 Riemannian submersion, see \cite[\S 4]{Fukaya87ld}. Notice that the Hausdorff
 distance (measured within $(V_{x_0},\hat{g}_X)$) between $\tilde{c}_i
 (E_i(\hat{x}, \varepsilon))$ and $B_{\hat{g}_X}(\hat{x}, \varepsilon)$ is
 bounded  above by $\Psi(\delta_i)$, then by (\ref{eqn: limit_E}), (\ref{eqn:
 E_and_det_H}), (\ref{eqn: limit_CGamma}) and (\ref{eqn: U_integral}) we have
 \begin{align}\label{eqn: det_H_integral}
  \sqrt{\det H}(\hat{x})\ =\ &\frac{1}{v(\varepsilon_0)}\lim_{\varepsilon\to
 0}\lim_{i\to  \infty}\frac{|C\Gamma_i(\varepsilon_0)|}{\omega_n\varepsilon^n}
 \int_{\tilde{c}_i(E_i(\hat{x},\varepsilon))}|\hat{c}_i^{-1}(z)|_{\hat{g}_{i|}}\
 \dvol_{\hat{g}_X}(z).
 \end{align}
 
 To relate this with $\hat{\chi}_C(\hat{x})$, let us recall the definition
 (\ref{eqn: hat_chi_C}), and that for all $i$, $CW_{x_0}$ is the universal
 covering of $C\hat{V}_{x_0,i}$, equipped with the covering metric. Moreover, we
 could consider the subsets $W_i(\varepsilon_0):=\{ \gamma
 (\hat{x},o):\ \gamma\in C\Gamma_i(\varepsilon_0),\ \hat{x}\in V_{x_0}\}$
 of $CW_{x_0}$. Clearly $\{W_i(\varepsilon_0)\}$ sub-converges in the pointed
 Gromov-Hausdorff sense to a fixed open neighborhood
 $W_{\infty}(\varepsilon_0)$ of the zero section $V_{x_0}\times \{o\} \subset
 CW_{x_0}$, and we have the limit lower bound 
 \begin{align*}
 \lim_{i\to \infty}|C\Gamma_i(\varepsilon_0)||C\hat{V}_{x_0,i}|_{\hat{g}_{i|}}\
 =\ \lim_{i\to \infty}|W_i(\varepsilon_0)|_{\tilde{g}_{i |}}\ =\
 |W_{\infty}(\varepsilon_0)|_{\tilde{g}_{\infty |}}\ >\ 0,
 \end{align*}
 By this lower bound, together with (\ref{eqn: hat_chi_C}), (\ref{eqn:
 det_H_integral}) and the fact that $d_{H}(\tilde{c}_i
 (E_i(\hat{x},\varepsilon)), B_{\hat{g}_X}(\hat{x},\varepsilon))\le
 \Psi(\delta_i)$, we could obtain for for any $\hat{x}\in V_{x_0}$ that
 \begin{align*}
 \sqrt{\det H}(\hat{x})\ =\ &\frac{1}{v(\varepsilon_0)}\lim_{\varepsilon \to
 0}\lim_{i\to \infty}
 \frac{|C\Gamma_i(\varepsilon_0)||C\hat{V}_{x_0,i}|_{\hat{g}_{i|}}}
 {\omega_n\varepsilon^n} \int_{\tilde{c}_i(E_i(\hat{x},\varepsilon))}
 \frac{|\hat{c}_i^{-1}(z)|_{\hat{g}_{i |}}}{|C\hat{V}_{x_0,i}|_{\hat{g}_{i
 |}}}\ \dvol_{g_X}(z)\\
 =\ 
 &\frac{|W_{\infty}(\varepsilon_0)|_{\tilde{g}_{\infty |}}}{v(\varepsilon_0)}
 \lim_{\varepsilon\to 0}\fint_{B_{\hat{g}_X}(x,\varepsilon)}\hat{\chi}_C\
 \dvol_{\hat{g}_X}\\
 =\ &C_{5.3}(x_0)^{-1}\ \hat{\chi}_C(\hat{x}),
 \end{align*}
 where
 $C_{5.3}(x_0):=v(\varepsilon_0)
 |W_{\infty}(\varepsilon_0)|_{\tilde{g}_{\infty |}}^{-1}$ is a constant solely
 depending on $x_0\in \tilde{\mathcal{R}}$.
\end{proof}

By (\ref{eqn: rep_hat_chi_C}), this proposition says that the limit central
density function is locally represented as the density function of the limit
central distribution on $W_{x_0}$. At this stage, Theorem \ref{thm:
limit_central_density} is a direct consequence of the equation (\ref{eqn:
rep_hat_chi_C}) and Proposition \ref{prop: central_density}.

\section{Limits of controlled immortal Ricci flows}
In this section we prove Theorems \ref{thm: main1} and \ref{thm: main2}. As
mentioned in the introduction, the new difficulty for the Ricci flow case
compared to the Ricci flat case in \cite{NaberTian} is that Ricci flatness
directly gives, via O'Neill's formula, a set of elliptic equations on the limit
nilpotent bundle over the regular part, while for Ricci flows one has to push
the time to infinity in order to obtain a static equation. The desired static
equation is a consequence of the asymptotic vanishing of the \emph{time
derivative} of the $\mathcal{F}$- and $\mathcal{W}_+$-functionals to be
discussed in the first sub-section, and it also relies on our previous
discussion on the collapsing and convergence of integrals. In the second
sub-section we show that only orbifold type singularities may occure on the
collapsing limit, via a maximum principle argument inspired by the work of
Naber and Tian~\cite[Page 127]{NaberTian}, and this will finish the proofs of
Theorems \ref{thm: main1} and \ref{thm: main2}.

\subsection{The gradient Ricci soliton metrics on the limit unwrapped
neighborhoods}
\addtocontents{toc}{\protect\setcounter{tocdepth}{1}}
We begin with considering the behavior of the $\mathcal{F}$-functional for an 
immortal Ricci flows $(M,g(t))$ with a uniform curvature bound, especially one
satisfying the assumption of Theorem~\ref{thm: main1}. We will always fix a
solution $u\in C^{\infty}(M\times [0,\infty))$ to (\ref{eqn: conjugate_heat}).

Recall that $\mathcal{F}(g(t),u(t))\nearrow 0$ as $t\to \infty$, by the uniform
curvature bound on $M\times [t,\infty)$ we must have the following
\begin{lemma}\label{lem: vanishing_F'}
Along the immortal Ricci flow $(M,g(t))$ with uniformly bounded curvature, and
for any solution $u$ to (\ref{eqn: conjugate_heat}) on $M\times [0,\infty)$,
we have
\begin{align*}
\lim_{t\to \infty}\mathcal{F}'(g(t),u(t))\ =\ 0.
\end{align*}
\end{lemma}
\begin{proof}
We abbreviate $\mathcal{F}(t)=\mathcal{F}(g(t),u(t))$ and recall that
\begin{align*} 
\mathcal{F}'(t)\ =\ 2\int_M\left|\Rc_{g(t)}-\nabla^2_{g(t)} \ln u(t)\right|^2\
u(t)\ \dvol_{g(t)}.
\end{align*}
Moreover, we have the evolution equaitons (summing over repeating indicies)
\begin{align*}
\partial_t\Rc_{ij}\ =\ \Delta \Rc_{ij} - 2\Rm_{iklj}\Rc_{kl}-
2\Rc_{ik}\Rc_{jk}\quad  
\text{and}\quad \partial_t\Gamma_{ij}^k\
=\ -g^{kl}(\nabla_i\Rc_{jl}+\nabla_j\Rc_{il}-\nabla_l\Rc_{ij}),
\end{align*}
and therefore we could further compute
\begin{align*}
&\partial_t|\Rc-\nabla^2\ln u|^2\\
=\ &2 \Rc_{ij}(\Rc_{ik}-\nabla^2_{ik}\ln u)(\Rc_{jk}-\nabla^2_{jk}\ln u)
+2(\Delta\Rc_{ij}-2\Rm_{iklj}\Rc_{kl}-2\Rc_{ik}\Rc_{jk})(\Rc_{ij}-\nabla^2_{ij}\ln
u)\\
&-2\left(\nabla^2_{ij}(\ln u)'-(\nabla_j\Rc_{ik}+\nabla_i\Rc_{jk}
-\nabla_k\Rc_{ij})\nabla_k\ln u\right) (\Rc_{ij}-\nabla^2_{ij}\ln u).
\end{align*}

Since $|\Rm|(t)\le C$, by Shi's estimate~\cite{Shi}, there exists some
$\delta_0(C,m)>0$ and $C(l,m)>0$, such that for any $s\in [t,t+\delta_0]$ and
any $l\in \mathbb{N}$, $|\nabla^l \Rm|(s)\le \min\{C,C(l,m)(s-t)^{-l}\}$.
Moreover, by the usual parabolic estimate $|\nabla^2\ln u|t+|\nabla \ln
u|^2t\le C$, and thus we can estimate for any $s\in [t,t+\delta_0]$: 
\begin{align*}
\begin{split}
\mathcal{F}''(s)\ =\ &\int_M\partial_s\left|\Rc_{g(s)}-\nabla_{g(s)}^2\ln
u\right|^2 u(s)\ \dvol_{g(s)}-\int_M\left|\Rc_{g(s)}-\nabla_{g(s)}^2\ln
 u\right|^2\Sc_{g(s)}u(s)\ \dvol_{g(s)}\\
\le\ &C(C,m)\int_M\left((|\nabla^2\Rm|+|\Rm|^2+|\nabla\Rm||\nabla
\ln u|+|\nabla^2\ln u|)(|\Rm|+|\nabla^2 \ln u|)\right)u\
\dvol_{g}\\
\le\ &C(C,m,\delta_0).
\end{split}
\end{align*}

This local boundedness of $\mathcal{F}''(t)$ ensures that
$\lim_{t\to \infty}\mathcal{F}'(t)= 0$: Otherwise, we could find a sequence
$t_i\to \infty$ such that $\mathcal{F}'(t_i)\ge \varepsilon_0>0$; then by
the above bound of $\mathcal{F}''(t)$, we have $\mathcal{F}'(s)\ge
\frac{\varepsilon_0}{2}$ for all $s\in [t_i-\delta_1,t_i+\delta_1]$, where
$\delta_1=\frac{\varepsilon_0}{2}C(C,m,\delta_0)^{-1}$; and therefore
$\mathcal{F}(t_i+\delta_1)-\mathcal{F}(t_i-\delta_1)\ge \varepsilon_0\delta_1>0$
for any $i$, implying $\mathcal{F}(t)\to \infty$ as $t\to \infty$ by the
monotonicity of $\mathcal{F}(t)$. This contradicts the fact that
$\lim_{t\to\infty}\mathcal{F}(t)=0$.
\end{proof}

Now suppose that an unbounded sequence of time slices along an immortal Ricci
flow satisfying the assumption of Theorem~\ref{thm: main1} collapses, we can
show the existence of a gradient steady Ricci soliton metric on the limit 
unwrapped neighborhoods around orbifold points:
\begin{prop}\label{prop: steady_soliton_metric}
Let $(M,g(t))$ be an immortal Ricci flow satisfying the assumption of
Theorem~\ref{thm: main1}. For any $t_i\to \infty$, suppose $(M,g(t_i))$
converges to a lower (Hausdorff) dimensional metric space $(X,d)$, and 
$\forall x_0\in \tilde{\mathcal{R}}$, let $W_{x_0}$ be the limit unwrapped
neighborhood of $x_0$, as constructed in Theorem~\ref{thm: canonical_nbhd},
together with $\tilde{g}_i$, the natural covering metrics on $W_{x_0}$ induced
by $g(t_i)$. Then the limit metric $\tilde{g}_{\infty}$ on $W_{x_0}$ satisfies the
gradient steady Ricci soliton equation.
\end{prop} 
\begin{proof} 
For any unbounded time sequence $\{t_i\}$, let $u\in C^{\infty}(M\times
[0,\infty))$ solving (\ref{eqn: conjugate_heat}) be constructed as in \S 2.3.
By Lemma~\ref{lem: vanishing_F'}, we know that
\begin{align*}
\lim_{i\to \infty}\mathcal{F}'(g(t_i),u(t_i))\ =\ 0,
\end{align*}
and for the collapsing sequence
$\{(M,g(t_i))\}$, we put $\rho_i:=u(t_i)$ and $w_i:=2\left|\Rc_{g(t_i)}-
\nabla_{g(t_i)}^2\ln u(t_i)\right|^2$. To obtain a limit function, we need to
uniformly control $\|\rho_i\|_{C^1(M)}|M|_{g(t_i)}$.

By the uniform curvature and diameter bounds along the Ricci flow, we could
apply \cite[Propositions 5.1 and 5.3]{Foxy1809} to the space-time $M\times
(t_i-1,t_i+1]$ and obtain uniformly positive bounds of $|M|_{g(t_i)}\rho_i$,
and \cite[Proposition 5.1]{Foxy1809} together with \cite[Theorem
3.3]{QiZhang06} further guarantee the uniform upper bound of
$|M|_{g(t_i)}|\nabla \rho_i|$. Consequently, Proposition \ref{prop:
collapsing_integral} applies and we get a continuous limit function $w_X:X\to
[0,\infty)$ such that
\begin{align*}
\int_Xw_X\ \rho_X\ \dmu_X\ =\ &\lim_{i\to \infty}\int_M w_i\ \rho_i\
\dvol_{g(t_i)}\\
=\ &\lim_{t_i\to \infty}\mathcal{F}'(g(t_i),u(t_i))\\
=\ &0.
\end{align*}
Therefore $w_X\equiv 0$ on $X$. 

Now for any $x_0\in \tilde{\mathcal{R}} \subset X$ fixed, we let $W_{x_0}$ be
the limit unwrapped neighbrohood of $x_0$, as constructed in
Theorem~\ref{thm: canonical_nbhd}. Let $\tilde{g}_i$ on $W_{x_0}$ be the
covering metrics of $g_i$ and let $\tilde{g}_{\infty}$ be the smooth limit of
$\tilde{g}_i$ on $W_{x_0}$, whose existence is guaranteed by the uniform
regularity and injectivity radius lower bound of $\{\tilde{g}_i\}$. By
Proposition~\ref{prop: local_convergence}, for $\tilde{w}_i$, the pull-back of
$w_i$ to $W_{x_0}$, converges in the $C^{\infty}$ topology to some
$w_{\infty}$, since the unifrom regularity of the metrics $\tilde{g}(t_i)$
guaratees the uniform regularity of their curvature, and $u(t_i)$ has uniform
regularity control given by the conjugate heat equation (and the curvature
bound). Notice that this limit function $w_{\infty}$ is constant along the limit
$N$ orbits and has its values agree with $w_X$ on each fiber. Consequently,
$w_{\infty}\equiv 0$ on $W_{x_0}$. Similarly, we also have $\tilde{u}(t_i)$, the
pull-back of $u(t_i)$ to $W_{x_0}$ via the covering map, converges
to some $u_{\infty}$ in the $C^{\infty}$ topology on $W_{x_0}$.

On the other hand, by the smooth convergence of $\tilde{g}(t_i)$ on $W_{x_0}$,
we have
\begin{align}\label{eqn: convergence_tilde_w_i}
\begin{split}
\lim_{i\to \infty}\tilde{w}_i\ =\
&2\lim_{i\to\infty}\left|\Rc_{\tilde{g}(t_i)}-\nabla_{\tilde{g}(t_i)}^2\ln
\tilde{u}(t_i)\right|^2\\
=\ &2\left|\Rc_{\tilde{g}_{\infty}}-\nabla_{\tilde{g}_{\infty}}^2\ln
u_{\infty}\right|^2.
\end{split}
\end{align}
This implies that on $W_{x_0}$, the limit Riemannian metric $\tilde{g}_{\infty}$
satisfies the gradient steady soliton equation:
\begin{align}\label{eqn: setady_soliton_metric}
\Rc_{\tilde{g}_{\infty}}-\nabla_{\tilde{g}_{\infty}}^2\ln u_{\infty}\ \equiv\ 0.
\end{align}
Especially, the Ricci flow on $W_{x_0}$ becomes the one generated by 
$\mathcal{L}_{-\nabla \ln u_{\infty}}$.
\end{proof}

Similar to the discussion above, we expect to show, for a Ricci flow
satisfying the assumptions of Theorem~\ref{thm: main2}, that any rescaled
sequence $\{(M,t_i^{-1}g(t_i))\}$ with $t_i\to\infty$ will sub-converge to produce
certain limit expanding gradient Ricci soliton metric. Here we recall that if
the diameter growth is of order $t^{\frac{1}{2}}$, and the global volume ratio
$|M|_{g(t)}\diam (M,g(t))^{-m}$ fails to have a uniformly positive lower bound
along the Ricci flow, then the $\mathcal{W}_+$-functional will be unbounded
(see \cite{FIN05}), and thus the argument leading to the asymptotic vanishing
of $\mathcal{F}'$ in proving Lemma~\ref{lem: vanishing_F'} will not work for
$\mathcal{W}_+'(t)$. An even more serious issue is the unfavorable rescaling
effect: contrary to the constancy of the $\mathcal{F}$- or
$\mathcal{W}$-functionals leading to corresponding gradifent Ricci soliton
equations by the vanishing of their derivatives, the asymptotic vanishing of
$\mathcal{W}'_+$ cannot guarantee the existence of a gradient expanding Ricci
soliton equation when taking rescaled limits. Therefore a condition concerning
$t\boldsymbol{\mu}'(t)$ in (\ref{eqn: main2}) is necessary. We now prove the
following
\begin{prop}\label{prop: expanding_soliton_metric}
Let $(M,g(t))$ be an immortal Ricci flow satisfying the assumptions of
Theorem~\ref{thm: main2} such that
$\limsup_{t\to\infty}t\boldsymbol{\mu}_+'(t)=0$, and let $\{t_i\}$ be an
unbounded seqeuence.
Assume that the sequence $\{(M, t_i^{-1} g(t_i))\}$ collapses to a lower (Hausdorff)
dimensional metric space $(X,d)$. For any $x_0\in \tilde{\mathcal{R}}$, let
$W_{x_0}$ be the limit unwrapped neighborhood of $x_0$, as constructed in
Theorem~\ref{thm: canonical_nbhd}, together with $\tilde{g}_i$, the natural
covering metric on $W_{x_0}$ induced by $t_i^{-1}g(t_i)$. Then the limit metric
$\tilde{g}_{\infty}$ on $W_{x_0}$ satisfies the gradient expanding Ricci
soliton equation.
\end{prop}

\begin{proof}
For any sequence $t_i\to \infty$, let $u_i$ be the minimizer (whose existence
guaranteed by \cite[Theorem 17. (a)]{FIN05}) of the
$\boldsymbol{\mu}_+$-functional, i.e.
$\mathcal{W}_+(g(t_i),u_i,t_i)=\boldsymbol{\mu}_+(t_i)$. By (\ref{eqn: mu+'}) we
also know that
\begin{align*}
\lim_{i\to\infty}t_i\ \mathcal{W}_+'(g(t_i),u_i,t_i)\ =\ 0,
\end{align*} 
and by the expression
\begin{align}\label{eqn: t^2W'}
t_i\ \mathcal{W}_+'(g(t_i),u_i)\ =\
\int_M2t_i^2\left|\Rc_{g(t_i)}-\nabla_{g(t_i)}^2\ln
u_i+\frac{g(t_i)}{2t_i}\right|^2\ u(t_i)\ \dvol_{g(t_i)},
\end{align}
we know that after the rescaling $g(t_i)\mapsto g_i:=t_i^{-1}g(t_i)$ and 
$t_i\mapsto 1$, and with the notations $\rho_i:=t_i^{\frac{m}{2}}u_i$ and
$w_i:=\left|\Rc_{g_i}-\nabla_{g_i}^2\ln \rho_i+\frac{1}{2}g_i\right|^2_{g_i}$,
we have
\begin{align*}
\lim_{i\to \infty} \int_Mw_i\ \rho_i\ \dvol_{g_i}\ =\ 0.
\end{align*}

In order to apply Theorem \ref{thm: integral} and obtain gradient expanding
Ricci soliton metric on a limit unwrapped neighborhood, we still need to check
the regularity of the minimizers of the $\boldsymbol{\mu}_+$-functional.
Since each $u_{i}$ minimizes $\boldsymbol{\mu}_+(t_i)$, by the scaling
invariance property of the $\mathcal{W}_+$-functional, we also have
$\boldsymbol{\mu}_+(t_i)=\mathcal{W}_+(g_i,\rho_i,1)$. Therefore, the
functional
\begin{align*}
 W^{1,2}(M)\ni v\quad \mapsto\quad \int_{M}\left(2|\nabla
v|^2+\Sc_{g_i}v^2+2v^2\ln v+m(1+\ln 2\sqrt{\pi})v^2\right)\ \dvol_{g_i}
\end{align*}
achieve its minimum value $\boldsymbol{\mu}_+(t_i)$ at the function
$\sqrt{\rho_i}$, subject to the conditions $\int_{M}v^2\ \dvol_{g_i}=1$ and
$v\ge 0$. Let $v_i:=\sqrt{\rho_i}$, then
$\mathcal{W}_i(v_i)=\boldsymbol{\mu}_+(t_i)$, and by \cite{Rothaus}, $v_i$
satisfies the following Euler-Lagrange equation:
\begin{align}\label{eqn: EulerLagrange}
2\Delta_{g_i}v_i\ =\ \left(2\ln v_i+\Sc_{g_i}+m(1+\ln
2\sqrt{\pi})-\boldsymbol{\mu}_+(t_i)\right)v_i.
\end{align}
On the other hand, we could bound $\boldsymbol{\mu}_+(t_i)$ from above by
plugging in $v=|M|_{g_i}^{-\frac{1}{2}}$ and applying the curvature bound of
$g_i$:
\begin{align*}
\boldsymbol{\mu}_+(t_i)\ \le\ &\fint_M \Sc_{g_i}-\ln |M|_{g_i}\
\dvol_{g_i}+m(1+\ln 2\sqrt{\pi})\\
\le\ &-\ln |M|_{g_i}+m(m+\ln 2\sqrt{\pi}).
\end{align*}
Consequently, by (\ref{eqn: EulerLagrange}) we can then check at the maximum of
$v_i$ (existence guaranteed by the continuity of $v_i$ and compactness of $M$)
that
\begin{align*}
2\ln v_i+\Sc_{g_i}+m(1+\ln 2\sqrt{\pi})\ \le\ m\ln 4\pi e-\ln |M|_{g_i}.
\end{align*}
And it follows that $v_i$ is bounded from above as 
\begin{align}\label{eqn: vi2_max} 
\max_M v_i^2\ \le\  C(m)|M|_{g_i}^{-1}.
\end{align}
On the other hand, by the Cheng-Yau gradient estimate \cite{ChengYau} and the
uniform curvature bound, we have some $C_{CY}(m)>0$ such that
\begin{align}\label{eqn: vi2'_max}
\max_M|\nabla \ln v_i|\ \le\ C_{CY}(m),
\end{align}
and the resulting Harnack inequality $\min_M v_i^2\ \ge\ C(m)\max_M v_i^2$
holds, in view of the uniform diameter upper bound of $(M,g_i)$, which is
ensured by the assumption on the diameter growth (of order $t^{\frac{1}{2}}$).
Consequently, the unit total mass of $v_i^2$ implies that
\begin{align}\label{eqn: vi2_min}
\min_M v_i^2\ \ge\ C(m)|M|_{g_i}^{-1}.
\end{align}
Now the estimates (\ref{eqn: vi2_max}), (\ref{eqn: vi2_min}) and (\ref{eqn:
vi2'_max}) together enable us to apply Theorem \ref{thm: integral} to
$\rho_i=v_i^2$ and $w_i$, whose $C^k$ regularity guaranteed by the uniform
regularity control of the metric $g_i$, as well as bootstrapping the elliptic
equation (\ref{eqn: EulerLagrange}) whose coefficients are initially controlled
by (\ref{eqn: vi2_max}) and (\ref{eqn: vi2_min}). Following exactly the same
argument as the proof of Proposition~\ref{prop: steady_soliton_metric} we get
to the desired conclusion, with a limit metric $\tilde{g}_{\infty}$ and
potential function $u_{\infty}$ on $W_{x_0}$.
\end{proof}


\begin{remark}\label{rmk: necessity}
One may wonder if we could pull the integrands of $\mathcal{F}'(g(t_i),u(t_i)$
or $t_i\mathcal{W}_+'(g(t_i),u_i)$ back to the frame bundle $FM_i$, consider the
convergence of integrals directly as in Proposition \ref{prop:
collapsing_integral}, and then apply the known unwrapping results as
\cite[Theorem 2.1]{NaberTian4d} or \cite[Theorem 1.1]{NaberTian} to obtain a
limit quantity on a limit unwrapped neighborhood and then take $O(m)$ quotient.
This would work, but we notice that on the frame bundles the pull-back
integrands are not the corresponding geometric quantities of the pull-back metrics, and the
vanishing limit integrand does not provide the desired geometric information
directly as does (\ref{eqn: convergence_tilde_w_i}). Decoding the information of
the vanishing limit integrand on the unwrapped frame bundle amounts to the same
work done in proving Theorem \ref{thm: canonical_nbhd}.
\end{remark}

\subsection{Proof of the main results} In this sub-section we prove the two main
theorems: we show that the possible collapsing limit $(X,d)$ of the
sequence of Ricci flow time slices given as in Theorem \ref{thm: main1} could
only develop orbifold type singularities, and to show Theorem \ref{thm:
main2}, we will assume a type-III Ricci flow to satisfy, in addition to
the $t^{\frac{1}{2}}$ diameter growth condition, also 
\begin{align}\label{eqn: contrapositive}
\limsup_{t\to \infty}t\boldsymbol{\mu}_+'(t)\ =\ 0\quad \text{and}\quad
\liminf_{t\to \infty}|M|_{g(t)}\diam (M,g(t))^{-m}\ =\ 0,
\end{align}
and then deduce a contradiction in a similar way to proving
Theorem~\ref{thm: main1}. 

We begin with a proof of Theorem~\ref{thm: main1}. Recalling the discussion
in \S 2.1.1 on the singularity types of $X$, we remember that our major goal is
to rule out the existence of $\tilde{\mathcal{S}}$, which is characterized as the
vanishing set of $\chi_C$, as discussed in \S 5. We also notice that the
sequence $\{(M,g_i)\}$, with $g_i:=g(t_i)$, enjoy the uniform curvature and
diameter bounds as mentioned in the introduction, see (\ref{eqn: regularity}).

We will rely on a maximum principle argument: were
$\tilde{\mathcal{S}} \not=\emptyset$, $\chi_C\ge 0$ vanishes on
$\tilde{\mathcal{S}}$, leaving us a positive global maximum within
$\tilde{\mathcal{R}}$, by the continuity of $\chi_C$ and the compactness of
$X$; on the other hand, since this maximum is achieved within
$\tilde{\mathcal{R}}$, we could express $\chi_C$ in the limit unwrapped
neighborhood around the maximum point, as a constant multiple $\sqrt{\det H}$
--- the volume density of the limit central distribution, as shown in \S 4;
exploiting the local Riemannian submersion structure around in the limit
unwrapped neighborhood via O'Neill's formula, we see that Propositions
\ref{prop: steady_soliton_metric} and \ref{prop: expanding_soliton_metric}
guarantee $\ln \det H$ to be $\ln (\det H)^{\frac{1}{2}}u_{\infty}$-sub-harmoic,
but the maximum principle applies to show that $\ln\det H$ is locally constant;
this will imply that $\chi_C$ is a positive constant, whence the non-existence
of $\tilde{\mathcal{S}}$.

More specifically, recall that in \S 4 we have constructed on $X$ a limit
central density function $\chi_C\ge 0$ which is continuous and vanishes exactly
on $\tilde{\mathcal{S}}$. Letting $A=\max_X \chi_C$, we know that there is some
$x\in \tilde{\mathcal{R}}$ where $\chi_C(x)=A$, by the compactness of $X$.
By the continuity of $\chi_C$, we see that $\chi_C^{-1}(A)$ is a closed and
non-empty subset of $X$, and we will show that it is open.

Now fixing any $x_0\in \tilde{\mathcal{R}}\cap \chi_C^{-1}(A)$, we let
$U_{x_0}$, $V_{x_0}$ and $W_{x_0}$ be as given in Theorem~\ref{thm:
canonical_nbhd}. We let $G_{x_0}$ denote the orbifold group and
$f:W_{x_0}\to V_{x_0}$ denote the fibration. By Propositions \ref{prop:
steady_soliton_metric} and \ref{prop: expanding_soliton_metric}, we know that on
$W_{x_0}$ there is a limit metric $\tilde{g}_{\infty}$ together with a potential
function $u_{\infty}$ satisfying the gradient steady or expanding Ricci soliton
equations. There is a simply connected nilpotent Lie group $N$ which acts on
$(W_{x_0}, \tilde{g}_{\infty})$ freely and isometrically, and the $N$ orbits
are exactly the $p$ fibers. 
Moreover, there is a limit central distribution $\mathfrak{C}$ of Killing
vector fields tangent to the $N$ orbits, so that $\mathfrak{C}$ provides a
Riemannian foliation of the manifold $(W_{x_0},\tilde{g}_{\infty})$.
Integrating vector fields in $\mathfrak{C}$ we obtain the action of a simply
connected abelian group $Z\triangleleft C(N)$ on $W_{x_0}$.
This action is still free and isometric, and therefore $W_{x_0}\slash Z$ is a
smooth Riemannian manifold when equipped with the quotient metric
$[\tilde{g}_{\infty}]$. We let $[p]:W_{x_0}\slash Z\to V_{x_0}$ denote the
quotient fibration. Letting $H$ denote the metric $\tilde{g}_{\infty}$
restricting to $\mathfrak{C}$, then $H$ is invariant under the $N$ action,
whence the constancy of $\det H$, the volume form of $H$, along each $N$ orbit,
thus descending to a mooth function on $V_{x_0}$, still denoted by $\det H$.
Finally, by Theorem \ref{thm: limit_central_density}, we know that there is a
constant $C(x_0)>0$ such that $\chi_C=C(x_0)\sqrt{\det H}$ throughout $V_{x_0}$.

By the Riemannian foliation structure on $W_{x_0}$, we could compute
$\nabla^2_{\tilde{g}_{\infty}}\ln u_{\infty}$ as following (with the notational
convention of \cite{Lott10}):
\begin{align}
\begin{split}
-\ln u_{\infty;\ ab}\ =\ 
&-\frac{1}{2}H_{ab,\ \alpha}\ln u_{\infty,\ \alpha}+\frac{1}{4}H_{ab,\
\alpha}H^{cd}H_{cd,\ \alpha};\\
-\ln u_{\infty;\ a\alpha}\ =\ 
&-\frac{1}{2}H_{ab}\mathbf{A}_{\alpha\beta}^j\ln u_{\infty,\
\beta};\\
-\ln u_{\infty;\ \alpha\beta}\ =\ &-\ln u_{\infty;\
\alpha\beta}+\frac{1}{2}H^{ab}H_{ab;\ \alpha \beta}-\frac{1}{2}H_{ab,\
\alpha}H_{ab,\ \beta}.
\end{split}
\label{eqn: total_hessian_barf}
\end{align}
Here $a,b,c,d$ denote the indicies along the fibers of $[p]$, and $\alpha,\beta$
denote the indicies in those directions perpendicular to the fibers of $[p]$,
and $\mathbf{A}$ measures the failure of the integrability of the basic vector
fields labelled by indicies $\alpha,\beta$, see also \cite{Besse, ONeill}. On
the other hand, $\Rc_{\tilde{g}_{\infty}}$ can be computed via O'Neill's formula
\cite{Besse, ONeill} --- we notice that the fiber metric $H$ is flat. Adding
$\Rc_{\tilde{g}_{\infty}}$ and $-\nabla_{\tilde{g}_{\infty}}^2\ln u_{\infty}$,
and collecting with respect to types, we get:
\begin{align*}
\begin{split}
&\Rc_{\tilde{g}_{\infty}}-\nabla_{\tilde{g}_{\infty}}^2\ln u_{\infty}\\ =\
&-\frac{1}{2}\left(H_{ab;\ \alpha\alpha}+\frac{1}{2}(\ln \det
H)_{,\ \alpha} H_{ab,\ \alpha}-H^{cd}H_{ac,\ \alpha}H_{bd,\
\alpha} -\frac{1}{2}\mathbf{A}_{\alpha\beta}^c
\mathbf{A}_{\alpha\beta}^dH_{ac}H_{bd} +H_{ab,\ \alpha}\ln u_{\infty,\
\alpha}\right)\\
&+\frac{1}{2}\left(\mathbf{A}_{\alpha\beta;\ \beta}^a+H^{ab}H_{bc,\ \beta}
\mathbf{A}_{\alpha\beta}^c +\frac{1}{2}(\ln \det H)_{,\ \beta}
\mathbf{A}_{\alpha\beta}^a -\mathbf{A}^a_{\alpha\beta}\ln u_{\infty,\ \beta}
\right)\\
&+\left((\Rc_{[\tilde{g}_{\infty}]})_{\alpha\beta} -\frac{1}{2}(\ln \det
H)_{;\ \alpha\beta}-\frac{1}{4}H_{ab,\ \alpha}H_{ab,\ \beta}
 -\frac{1}{2}\mathbf{A}^a_{\alpha\gamma}\mathbf{A}^a_{\beta\gamma} -\ln
 u_{\infty;\ \alpha\beta}\right).
\end{split}
\end{align*}

Since the above decomposition of $\Rc_{\tilde{g}_{\infty}}$ is orthogonal with
respect to the types, by Proposition \ref{prop: steady_soliton_metric} we have
the following equation of symmetric two tensors holding on $W_{x_0}$:
\begin{align*}
H_{ab;\ \alpha\alpha}+\frac{1}{2}(\ln \det H)_{,\ \alpha} H_{ab,\ \alpha}
-H^{cd}H_{ac,\ \alpha}H_{bd,\ \alpha}
-\frac{1}{2}\mathbf{A}_{\alpha\beta}^c\mathbf{A}_{\alpha\beta}^dH_{ac}H_{bd}
+H_{ab,\ \alpha} \ln u_{\infty,\ \alpha}\ =\ 0.
\end{align*}
Consequently, since $H$ is positive definite on $W_{x_0}$, we could trace the
above equation of tensors by $H$ to obtain the numerical equation 
\begin{align*}
\Delta^{\perp}_{\tilde{g}_{\infty}}\ln \det H+\frac{1}{2}|\nabla^{\perp} \ln \det
H|^2-\tilde{g}_{\infty}(\nabla^{\perp} \ln \det H,\nabla^{\perp} \ln
u_{\infty})\ =\ \frac{1}{2}|\mathbf{A}|^2.
\end{align*}
Here by $\Delta_{\tilde{g}_{\infty}}^{\perp}$ and $\nabla^{\perp}$ we mean
taking derivatives in directions perpenticular to the leaves of $\mathfrak{C}$.
Since $\mathfrak{C}$ provides a Riemannian foliation, and that $\ln
\det H$ and $\ln u_{\infty}$ are constant along the leaves of $\mathfrak{C}$, we
see that the above equation descends to one valid on the quotient
$(W_{x_0}\slash Z,[\tilde{g}_{\infty}])$:
\begin{align}\label{eqn: elliptic_equality}
\Delta_{[\tilde{g}_{\infty}]}\ln \det H+\frac{1}{2}|\nabla \ln \det
H|^2-[\tilde{g}_{\infty}](\nabla \ln \det H,\nabla \ln u_{\infty})\
=\ \frac{1}{2}|\mathbf{A}|^2.
\end{align}
Moreover, as functions on $W_{x_0}$, $\ln \det H$ and $\ln u_{\infty}$ are even
constant along entire $p$ fibers, and $\sqrt{\det H}$ is a constant multiple of
$\chi_C$, seen as functions on $W_{x_0} \slash N$. Therefore $\ln \det H$ is
equal to its local maximum ($=C(x_0)^{-2}A^2$) everywhere on $[p]^{-1}(x_0)
\subset W_{x_0}\slash Z$. Now applying the maximum principle argument to $\ln
\det H$ (regarded as a function on $W_{x_0}\slash Z$) at any point on
$[p]^{-1}(x_0)$, by (\ref{eqn: elliptic_equality}) we know that it must be a
positive constant throughout $W_{x_0}\slash Z$. 
 Consequently, by the constancy of $\ln \det H$ on any $p$ fiber, and the fact
 that $\sqrt{\det H}$ is a constant multiple of $\chi_C$ on $W_{x_0}\slash N$,
 we know that $\chi_C= A$ within $U_{x_0}$ and thus $U_{x_0}\subset
 \chi_C^{-1}(A)$. But $U_{x_0}$ is open in $X$, and thus $\chi_C^{-1}(A)$ is
 also open in $X$, implying $\chi_C=A>0$ all over $X$, whence the non-existence
 of corner singularity, or equivalently $\tilde{\mathcal{S}}=\emptyset$. This
 is the desired conclusion of Theorem \ref{thm: main1}.
 
 Now to prove Theorem~\ref{thm: main2}, we consider a type-III Ricci flow
 satisfying $\diam (M,g(t))\le Dt^{\frac{1}{2}}$ and (\ref{eqn:
 contrapositive}). Pick a sequence $\{(M,g_i)\}$ with $g_i:=t_i^{-1}g(t_i)$
 such that $\lim_{i\to\infty}|M|_{g_i}\diam (M,g_i)^{-m}=0$. The assumptions on the
 curvature of the flow and Shi's estimates ensure that $\{(M,g_i)\}$ satisfy the
 regularity control (\ref{eqn: regularity}). Moreover, we could find an
 orbifold point $x_0$ in the possible collapsing limit $(X,d)$ where
 $\chi_C(x_0)$ achieves its global maximum. Then we could set up as before to
 write down a gradient expanding Ricci soliton equation in the unwrapped
 neighborhood $W_{x_0}$ around $x_0$, as concluded from Proposition \ref{prop:
 expanding_soliton_metric}. But this time, due to the structure of the expanding
 soliton equation, (\ref{eqn: elliptic_equality}) becomes 
 \begin{align}\label{eqn: elliptic_equality2}
 \Delta_{[\tilde{g}_{\infty}]}\ln \det H+\frac{1}{2}|\nabla \ln \det
H|^2-[\tilde{g}_{\infty}](\nabla \ln \det H,\nabla \ln u_{\infty})\
=\ \frac{1}{2}|\mathbf{A}|^2+\dim Z.
 \end{align}
 We can then argue by the maximum principle as before, to conclude that $\ln
 \det H$ is constant on $W_{x_0}$. Therefore, $\dim Z=0$.

However, we recall the formation of $Z$: before taking limit, there are simply
connected nilpotent Lie groups $N_i$ acting on $W_{x_0}$ freely, and so does
their centers $C(N_i)$; and $Z$ is nothing but the accumulation points of the
orbits of $C(N_i)$, after taking limit. Therefore, $\dim Z\ge \dim C(N_i)$ for
all $i$ sufficiently large. Notice that by the nilpotency of $N_i$, $\dim
C(N_i)>0$ unless $\dim N_i=0$. Therefore, $\dim Z>0$ were the sequence
$\{(M_i,g_i)\}$ to collapse.

This contradiction eliminates the possible volume collapsing of $\{(M_i,g_i)\}$,
or equivalently, the global volume ratio of $\{(M,g(t_i))\}$ has a uniformly
positive lower bound, contradicting the selection of $\{t_i\}$, guaranteed by
(\ref{eqn: contrapositive}). And the fallacy of (\ref{eqn: contrapositive})
establishes Theorem~\ref{thm: main2}.

\section{Connections with known results}
It is known that negatively curved compact manifolds tend to be rigid. Important
results along this direction include Gromov's uniform volume lower bound for
compact manifolds with negatively bounded sectional curvatrue \cite{Gromov78a}.
This theorem has been generalized by Rong \cite{Rong98} to the case of
negatively Ricci-curved manifolds with bounded curvature and diameter. In this
appendix we give a short proof of Rong's result as an application of Theorems
\ref{thm: canonical_nbhd} and \ref{thm: limit_central_density}:
\begin{theorem}[Theorem 0.4 of \cite{Rong98}]\label{thm: Rong} Let $(M^m,g)$ be
an $m$-dimensional closed Riemannian manifold satisfying the following
conditions:
\begin{enumerate}
  \item $\diam (M,g)\le D$;
  \item  $\Rc_g\le -\lambda(m-1)g$ for some $\lambda\in (0,1)$;
    \item the sectional curvature is uniformly bounded below by $-1$.
\end{enumerate}
Then there is a constant $v=v(m,D,\lambda)>0$ such that 
\begin{align*}
\vol (M,g)\ \ge\ v.
\end{align*}
\end{theorem}

\begin{proof} 
We prove by a contradiction argument. Suppose that there were a sequence
$\{(M_i,g_i)\}$ such that $|M_i|_{g_i} \to 0$ as $i\to \infty$. We start with
noticing that assumptions (2) and (3) ensure the sectional curvature to be
uniformly bounded between $-1$ and $(1-\lambda)(m-1)$. Invoking \cite[Theorem
1.1]{DWY96}, we then obtain suitable smoothings of the given metrics, and
obtain a sequence of nearby metrics $g'_i$ satisfying
\begin{enumerate}
  \item $\sup_{M_i}|\nabla^l\Rm_{g'_i}|\le C_l$;
  \item $\Rc_{g'_i}\le -\frac{\lambda}{2}(m-1)g'_i$;
  \item $D\diam (M_i,g_i)\le \diam (M_i,g'_i)\le D^{-1}\diam (M_i,g_i)$
  for some $D\in (0,1)$ independent of $i$; and
  \item $C|M_i|_{g_i}\le |M_i|_{g'_i}\le C^{-1}|M_i|_{g_i}$ for some
  $C\in (0,1)$ independent of $i$.
\end{enumerate} 

If $|M_i|_{g_i}\to 0$ as $i\to \infty$, by (1), (3) and (4) above we know that
the sequence $\{(M_i,g'_i)\}$ collapses with bounded curvature and diameter, and
a sub-sequence, still denoted by the original one, converges to a lower
(Hausdorff) dimensional metric space $(X,d)$. Moreover, there is a continuous 
non-negative limit central density function $\chi_C\ge 0$ on $X$, as guaranteed
by Theorem \ref{thm: limit_central_density}. Now arguing as before around the
global maximum $x_0$ of $\chi_C$, which has to be an orbifold point. We can work
in the limit unwrapped neighborhood for $x_0$, as constructed in Theorem
\ref{thm: canonical_nbhd}. 

Notice that on $W_{x_0}$, the fiber-wise covering metrics $\tilde{g}'_i$ already
has strictly negative Ricci curvature by (2), and by the regularity assumption
(1), we have a limit metric $\tilde{g}_{\infty}$ on $W_{x_0}$ satisfying 
\begin{align*}
\Rc_{\tilde{g}_{\infty}}\ \le\ -\frac{\lambda}{2}(m-1)\tilde{g}_{\infty}.
\end{align*} 
Now we follow the argument in \S 6.2, focusing on the limit central distribution
$\mathfrak{C}$ and the abelian group $Z$ it generates. Recalling that $Z$ acts
on $W_{x_0}$ freely, and applying the O'Neill's formula to the volume form $\det
H$ of the $\mathfrak{C}$ leaves, we obtain the elliptic inequality satisfied by
the $\ln \det H$:
\begin{align*}
\Delta^h_{[\tilde{g}_{\infty}]}\ln \det H\ \ge\ \lambda (m-1)\dim Z,
\end{align*}
for some smooth function $h$ on $W_{x_0}\slash Z$. This equation is the same as
(\ref{eqn: elliptic_equality2}) and the exact same argument proving
Theorem~\ref{thm: main2} tells that $\{(M,g_i')\}$ cannot collapse, whence the
volume non-collapsing of the original metrics $\{g_i\}$ by (4). This
contradiction establishes the desired volume lower bound.
\end{proof}

\subsection*{Acknowledgement} 
I would like to thank Yu Li, Xiaochun Rong and Bing Wang for useful discussions
during the preparation of the paper. I would also like to thank Richard
Bamler, John Lott and Song Sun for helpful comments on the non-collapsing and
Ricci flatness of the long-time limits by the evolution of immortal Ricci flows
with uniformly bounded curvature and diameter, as well as on the case of
comapct type-III Ricci flows.

\vspace{0.5in}


\begin{thebibliography}{100}


\bibitem{Bamler18} Richard H. Bamler, Long-time behavior of 3-dimensional Ricci
flow: introduction. \emph{Geom. Topol.} 22 (2018) 757-774.


\bibitem{Besse} Arthur L. Besse, Einstein manifolds. Ergebnisse der Mathematik
und ihrer Grenzgebiete (3), \emph{Springer-Verlag, Berlin,} 1987. xii+510 pp.
ISBN:
3-540-15279-2

\bibitem{Bohm15} Christoph B\"ohm, On the long time behavior of homogeneous
Ricci flows. \emph{Comment. Math. Helv.} 90 (2015), no. 3, 543-571.

\bibitem{BLS17} Christoph B\"ohm, Ramiro A. Lafuente and Miles Simon, Optimal
curvature estimates for homogeneous Ricci flows. \emph{Int. Math. Res. Not.}
rnx256, https://doi.org/10.1093/imrn/rnx256

\bibitem{BL18} Christoph B\"ohm and Ramiro A. Lafuente, Immortal homogeneous
Ricci flows. \emph{Invent. Math.} 212 (2018), no. 2, 461-529.

\bibitem{BuserKarcher81} Peter Buser and Hermann Karcher, Gromov's almost flat
manifolds. Ast\'erisque, 81. \emph{Soci\'et\'e Math\'ematique de France,
Paris,} 1981. 148 pp.


\bibitem{CG86} Jeff Cheeger and Mikhail Gromov, Collapsing Riemannian manifolds
while keeping their curvature bounded. I. \emph{J. Differential Geom.} 23
(1986), no. 3, 309-346.

\bibitem{CG90} Jeff Cheeger and Mikhail Gromov, Collapsing Riemannian manifolds
while keeping their curvature bounded. II. \emph{J. Differential Geom.} 32
(1990), no. 1, 269-298.


\bibitem{CCR01} Jianguo Cao, Jeff Cheeger and Xiaochun Rong, Splitting and
Cr-structures for manifolds with nonpositive sectional curvature. \emph{Invent.
Math.} 144 (2001), no. 1, 139-167.

\bibitem{CGT82} Jeff Cheeger, Mikhail Gromov and Michael Taylor, Finite
propagation speed, kernel estimates for functions of the Laplace operator, and
the geometry of complete Riemannian manifolds. \emph{J. Differential Geom.} 17
(1982), no. 1, 15-53.

\bibitem{CFG92} Jeff Cheeger, Kenji Fukaya and Mikhail Gromov, Nilpotent
structures and invariant metrics on collapsed manifolds. \emph{J. Amer.
Math. Soc.} 5 (1992), 327-372.

\bibitem{CR96} Jeff Cheeger and Xiaochun Rong, Existence of polarized
$F$-structures on collapsed manifolds with bounded curvature and diameter.
\emph{Geom. Funct. Anal.} 6 (1996), no. 3, 411-429.

\bibitem{ChengYau} Shiu Yuen Cheng and Shing-Tung Yau, Differential equations on
Riemannian manifolds and their geometric applications. \emph{Comm. Pure
Appl. Math.} 28(1975), 333-354.


\bibitem{DWY96} Xianzhe Dai, Guofang Wei and Rugang Ye, Smoothing Riemannian
metrics with Ricci curvature bounds. \emph{Manuscripta Math.} 90 (1996), 49-61.

\bibitem{dHF14} Matias del Hoyo and Rui Loja Fernandes, Riemannian metrics on
Lie groupoids. \emph{J. Reine Angew. Math.} 735 (2018), 143-173.


\bibitem{FIN05} Mikhail Feldman, Tom Ilmanen and Lei Ni, Entropy and reduced
distance for Ricci expanders. \emph{J. Geom. Anal.} 15 (2005), no. 1, 49-62.
  
  \bibitem{Fukaya87} Kenji Fukaya, Collapsing of Riemannian manifolds and
  eigenvalues of Laplace operator. \emph{Invent. Math.} 87 (1987), 517-547.
  
  \bibitem{Fukaya87ld} Kenji Fukaya, Collapsing Riemannian manifolds to ones of
  lower dimensions. \emph{J. Differential Geom.} 25 (1987), 139-156.

 \bibitem{Fukaya88} Kenji Fukaya, A boundary of the set of the Riemannian
manifolds with bounded curvatures and diameters. \emph{J. Differential Geom.} 28
(1988), 1-21.

\bibitem{Fukaya89} Kenji Fukaya, Collapsing Riemannian manifolds to ones with
lower dimension II. \emph{J. Math. Soc. Japan} 41 (1989), 333-356.

\bibitem{GGHR89} Eduardo Gallego, Luciano Gualandri, Gilbert Hector and
 Agust\'i Revent\'os, Groupo\"ides Riemanniens. \emph{Publ. Mat.} 33 (1989),
 No. 3, 417-422.



\bibitem{GL13} Alexander Gorokhovsky and John Lott, The index of a transverse
Dirac-type operator: the case of abelian Molino sheaf. \emph{J. Reine. Angew.
Math.} 678 (2013), 125-162.


\bibitem{Gromov78a} Mikhail Gromov, Manifolds with negative curvature. \emph{J.
Differential Geom.} 13 (1978), no. 2, 223-230.


\bibitem{Gromov78b} Mikhail Gromov, Almost flat manifolds. \emph{J. Differential
Geom.} 13 (1978), no. 2, 231-241.


 \bibitem{GLP} Mikhail Gromov, Structures m\'etriques
pour les vari\'et\'es Riemanniennes, Editions Cedic, Paris (1981).



\bibitem{Hamilton95} Richard Hamilton, A compactness property for solutions of
the Ricci flow. \emph{Amer. J. Math.} 117 (1995), 545-572.

\bibitem{HamiltonSurvey} Richard Hamilton, Formation of singularities in the
Ricci flow. \emph{Surveys in Diff. Geom.} 2 (1995), 7-136.

\bibitem{Hamilton99} Richard Hamilton, Non-singular solutions of the Ricci flow
on threemanifolds. \emph{Comm. Anal. Geom.} 7 (1999), 695-729.



\bibitem{Hilaire14} Christian Hilaire, Ricci flow on Riemannian
groupoids. \emph{Preprint}, arXiv: 1411.6058.

\bibitem{Foxy1705} Shaosai Huang, $\varepsilon$-Regularity and structure of
four dimensional shrinking Ricci solitons. \emph{Int. Math. Res. Not.} rny069,
https://doi.org/10.1093/imrn/rny069

\bibitem{Foxy1809} Shaosai Huang, Notes on Ricci flows with collapsing initial
data (I): Distance distortion. \emph{Preprint}, arXiv: 1809.07394.






\bibitem{Lott03} John Lott, Some geometric properties of the Bakry-\'Emery-Ricci
tensor. \emph{Comment. Math. Helv.} 78 (2003), 865-883.

\bibitem{Lott07} John Lott, On the long-time behavior of type-III Ricci flow
solutions, \emph{Math. Ann.} 339 (2007), 627-666.

\bibitem{Lott10} John Lott, Dimensional reduction and the long-time behavior of
Ricci flow, \emph{Comment. Math. Helv.} 85 (2010), 485-534.

\bibitem{Lott17} John Lott, The collapsing geometry of Ricci-flat $4$-manifolds.
To appear, \emph{Comment. Math. Helv.}, arXiv: 2017.06780.


\bibitem{Milnor76} John Milnor, Curvature of left invariant metrics on Lie
groups. \emph{Adv. Math.} 21 (1976), 293-329.






\bibitem{Molino88} Pierre Molino, Riemannian foliations. Translated from the
French by Grant Cairns. With appendices by G. Cairns, Y. Carri\`ere, \'E. Ghys,
E. Salem and V. Sergiescu. Progress in Mathematics, 73. \emph{Birkh\"auser
 Boston, Inc., Boston, MA}, 1988. xii+339 pp. ISBN: 0-8176-3370-7
 
 \bibitem{NaberTian4d} Aaron Naber and Gang Tian, Geometric structures of
 collapsing Riemannian manifolds I. Surveys in geometric analysis and
relativity, 439-466, Adv. Lect. Math. (ALM), 20, \emph{Int. Press, Somerville,
MA}, 2011.
 
\bibitem{NaberTian} Aaron Naber and Gang Tian, Geometric structures of
collapsing Riemannian manifolds II. \emph{J. Reine Angew. Math.} 744 (2018),
103-132.


\bibitem{ONeill} Barrett O'Neill, The fundamental equations of a submersion.
\emph{Michigan Math. J.} 13 (1966), 459-469.

\bibitem{Perelman} Grisha Perelman, The entropy formula for the Ricci flow and
its geometric applications. \emph{Preprint}, arXiv: math\slash 0211159.

\bibitem{Perelman2} Grisha Perelman, Ricci flow with surgery on three-manifolds.
\emph{Preprint}, arXiv: math\slash 0303109.





\bibitem{Rong98} Xiaochun Rong, A Bochner theorem and applications. \emph{Duke
Math. J.} 91 (1998), no. 2, 381-392.

\bibitem{Rothaus} Oscar S. Rothaus, Logarithmic Sobolev inequalities and the
spectrum of Schr\"odinger operators \emph{J. Funct. Anal.} 42 (1981), 110-120.



\bibitem{Ruh} Ernst A. Ruh, Almost flat manifolds. \emph{J. Differential Geom.}
17 (1982), 1-14.

\bibitem{Shi} Wan-Xiong Shi, Deforming the metric on complete Riemannian
manifolds. \emph{J. Differential Geom.} 30 (1989), no. 1, 223-301.


\bibitem{QiZhang06}  Qi S. Zhang, Some gradient estimates for the heat kernel
equation on domains and for an equation by Perelman.
\emph{Int. Math. Res. Not.} (2006), Art. ID 92314, 39 pp.


\end{thebibliography}
\end{document}